\documentclass[11pt, eqno]{article}
\usepackage{bbm}
\usepackage{mathrsfs}
\usepackage{amsfonts}
\usepackage{amssymb}
\usepackage{graphicx}
\usepackage[all]{xy}
\usepackage{amsthm}
\usepackage{amsmath}
\usepackage{amsmath,amssymb,latexsym,color}
\usepackage[mathscr]{eucal}
\usepackage{CJK}
\usepackage{cases}
\usepackage{graphics}

\usepackage{psfrag}
\usepackage{subfigure}
\usepackage{color}

\usepackage{amssymb,latexsym}
\usepackage{amsmath,latexsym}
\usepackage{amscd}

\newcommand{\N}{{\mathbb N}}

\newcommand{\R}{{\mathbb R}}

\newtheorem{theorem}{Theorem}[section]
\newtheorem{cor}[theorem]{Corollary}
\newtheorem{corollary}[theorem]{Corollary}

\newtheorem{thm}[theorem]{Theorem}
\newtheorem{definition}[theorem]{Definition}

\newtheorem{remark}[theorem]{Remark}

\newtheorem{lemma}[theorem]{Lemma}

\newtheorem{prop}[theorem]{Proposition}
\newtheorem{proposition}[theorem]{Proposition}
\newtheorem{claim}[theorem]{Claim}

\begin{document}

\title{Generalizations of Ekeland-Hofer and Hofer-Zehnder\\ symplectic capacities
and applications}
\date{April 3, 2023}
\author{Rongrong Jin and Guangcun Lu
\thanks{Corresponding author
\endgraf \hspace{2mm} Partially supported
by the NNSF  11271044 of China.
\endgraf\hspace{2mm} 2010 {\it Mathematics Subject Classification.}
 53D35, 53C23 (primary), 70H05, 37J05, 57R17 (secondary).}}
 \maketitle \vspace{-0.3in}

\abstract{
 In this paper we construct analogues of Ekeland-Hofer and Hofer-Zehnder symplectic capacities based  on a class of Hamiltonian boundary value problems  motivated by Clarke's and Ekeland's work,  and study generalizations of some important results about the original two capacities  (for example,  the famous Weinstein conjecture,  representation formula for $c_{\rm EH}$ and $c_{\rm HZ}$, and a theorem by Evgeni Neduv).
   } \vspace{-0.1in}
\medskip\vspace{12mm}

\maketitle 

\vspace{2mm}

\tableofcontents

\section{Introduction and main results}\label{section:1}
\setcounter{equation}{0}

A. Weinstein \cite{We78} and P. Rabinowitz \cite{Ra78}
proved, respectively, the existence of periodic orbits on a convex energy surface
and a strictly starshaped hypersurface of a Hamiltonian system in $\mathbb{R}^{2n}$.
Based on these, in 1978 A. Weinstein \cite{We79} proposed his famous conjecture: {\it every hypersurface
of contact type in symplectic manifolds carries a closed characteristic}.
Here a compact connected smooth hypersurface $\mathcal{S}$ in a symplectic manifold
$(M, \omega)$  is said to be of {\bf contact type} if there exists a vector
field $X$ defined in an open neighborhood $U$ of  $\mathcal{S}$ in $M$
which is transverse to $S$ and {satisfies $L_X\omega=\omega$ in $U$.
Such a vector field $X$  is called a {\bf Liouville field}}.
(\textsf{All compact manifolds or hypersurfaces in this paper are considered to
be boundaryless without special statements}.)
A {\bf closed characteristic}  of $\mathcal{S}$ is an embedded circle
$P\subset \mathcal{S}$ satisfying $TP=\mathcal{L}_\mathcal{S}{|}_P$, where
$\mathcal{L}_\mathcal{S}\rightarrow \mathcal{S}$ is {the} distinguished line bundle
defined by
$$\mathcal{L}_\mathcal{S}={\Big\{}(x,\xi)\in T\mathcal{S}\ {\Big |}\
 {\omega}_x(\xi,\eta)=0\;\hbox{for all}\;
\eta\in T_{x}\mathcal{S}{\Big \}}.
$$
In 1986, C.Viterbo \cite{Vi87} first proved the Weinstein conjecture in $(\mathbb{R}^{2n},\omega_0)$
with the global variational methods for periodic solutions of general Hamiltonian systems
initiated by P. Rabinowitz \cite{Ra78, Ra79} and A. Weinstein \cite{We78}.
Hereafter $\omega_0$ denotes the standard symplectic structure given by
$\sum_idq_i\wedge dp_i$ with the linear coordinates $(q_1,\cdots,q_n,p_1,\cdots,p_n)$.

Motivated by the above studies,
 I.~Ekeland and H.~Hofer \cite{EH89}
introduced a class of symplectic invariants (called symplectic capacities)
for subsets in $(\mathbb{R}^{2n},\omega_0)$ and reproved
the famous Gromov's nonsqueezing theorem in \cite{Gr85} and
{a $C^0$ rigidity theorem due to Gromov and Eliashberg.}
H. Hofer and E. Zehnder \cite{HoZe90} constructed a symplectic capacity
for any symplectic manifold, called the Hofer-Zehnder capacity.
The second named author of this article introduced the concept of  pseudo symplectic
capacities which is a mild generalization of symplectic
capacities, constructed a pseudo symplectic capacity  as a
generalization of the Hofer-Zehnder capacity and
established an estimate for it in terms of Gromov-Witten invariants (\cite{Lu3}).

For a symplectic matrix of order $2n$, $\Psi\in{\rm Sp}(2n,\mathbb{R})$,
as a generalization of the existence of closed characteristics  on the boundary $\mathcal{S}$ of a compact and convex set $D$ in $(\mathbb{R}^{2n},\omega_0)$ containing the origin in its interior,  Clarke \cite{Cl82, Cl83}
 proved: \textsf{ there exists a nonconstant  absolutely continuous curve $z:[0,T]\to \mathcal{S}$ for some $T>0$  such that  $J\dot{z}(t)\in \partial j_{D}(z(t))\;\hbox{a.e.}$ and that $z(T)=\Psi z(0)$.} Here  $\partial j_D$ is subdifferential of  the Minkowski functional $j_D$ of $D$ given by
 $$
 j_D(x)=\inf\left\{\lambda>0 \,\bigg|\, \frac{x}{\lambda}\in D\right\},
 $$
 and $J$ is the standard  complex structure on $\mathbb{R}^{2n}$ given by the matrix
\begin{equation}\label{e:standcompl}
J=\left(
           \begin{array}{cc}
             0 & -I_n \\
             I_n & 0 \\
           \end{array}
         \right)
\end{equation}
with the linear coordinates $(q_1,\cdots,q_n,p_1,\cdots,p_n)$ on
 $\mathbb{R}^{2n}$, where $I_n$ denotes the identity matrix of order $n$.
(We also use $J$ to denote the  standard  complex structure on $\mathbb{R}^{2k}$ for different $k\in\mathbb{N}_+$  without confusions.)

Clarke's result means that any linear symplectic transformation is realized on some orbit of any convex
energy surface, which was, in \cite[page 356]{Cl82}, viewed as a kind of converse
to the Goldstein's famous statement ``{\it the motion of a mechanical system corresponds to the
continuous evolution or unfolding of a canonical (i.e., symplectic) transformation}" \cite[\S8.6]{Go50}.
We generalize Clarke's results by constructing some analogues of Ekeland-Hofer and Hofer-Zehnder symplectic
capacities associated to symplectomorphisms. The finiteness of these "analogues"  is closely related to
boundary value problem of Hamiltonian systems, the non-periodic case.
The main difficulties in the constructions of these analogues are:
 \begin{description}
 \item[(i)] how to adapt the classical definitions such as admissible-function class, admissible deformations and  nonresonant conditions in \cite{EH89, HoZe90}
 to fit them in the present non-periodic case;
 \item[(ii)] how to solve new problems arisen in the related proofs by following the standard method in \cite{EH89, HoZe90}.
 \end{description}

{We introduce the following definitions about characteristics.}

\begin{definition}\label{def:character}
{\rm {\bf (i)} For a smooth hypersurface $\mathcal{S}$ in a symplectic manifold $(M, \omega)$
and $\Psi\in{\rm Symp}(M, \omega)$, a $C^1$ embedding $z$ from $[0,T]$ {\rm (for some $T>0$)}
 into $\mathcal{S}$ is called a $\Psi$-{\bf characteristic} on $\mathcal{S}$
 if $z(T)=\Psi z(0)$ and $\dot{z}(t)\in(\mathcal{L}_\mathcal{S})_{z(t)}\;\forall t\in [0,T]$.
Clearly, $z(T-\cdot)$ is a $\Psi^{-1}$-characteristic, and for any $\tau>0$ the embedding
$[0, \tau T]\to \mathcal{S},\;t\mapsto z(t/\tau)$ is also a
$\Psi$-characteristic.

\noindent{\bf (ii)} If $\mathcal{S}$ is the boundary of a compact convex set $D$ with nonempty interior in $(\mathbb{R}^{2n},\omega_0)$, and $\Psi\in{\rm Sp}(2n,\mathbb{R})$, corresponding to the definition of closed characteristics on $\mathcal{S}$ in Definition~1 of \cite[Chap.V,\S1]{Ek90} we say
a nonconstant  absolutely continuous curve $z:[0,T]\to \mathbb{R}^{2n}$ (for some $T>0$)
  to be a {\bf generalized characteristic} on $\mathcal{S}$
  if $z([0,T])\subset \mathcal{S}$ and
    $\dot{z}(t)\in JN_\mathcal{S}(z(t))\;\hbox{a.e.}$, where
    $N_\mathcal{S}(x)=\{y\in\mathbb{R}^{2n}\,|\, \langle u-x, y\rangle\le 0\;\forall u\in D\}$
    is the normal cone to $D$ at $x\in\mathcal{S}$.
    Moreover, if $z$ satisfies $z(T)=\Psi z(0)$ additionally, we call $z$ a {\bf generalized $\Psi$-characteristic} on $\mathcal{S}$. }
\end{definition}

Clearly, if $\mathcal{S}$ in (ii) is also  $C^{1,1}$  then generalized $\Psi$-characteristics on $\mathcal{S}$
are $\Psi$-characteristics  up to reparametrization. The notion of generalized characteristic might be defined on general
symplectic manifolds via nonsmooth analysis on manifolds,  but this is outside the scope of this paper
and would appear elsewhere.

The {\bf action} of an absolutely continuous curve
$x:[0,T]\rightarrow \mathbb{R}^{2n}$ is defined by
\begin{equation}\label{e:action1}
A(x)=\frac{1}{2}\int_0^T\langle -J\dot{x},x\rangle dt.
\end{equation}
Denote
  \begin{equation}\label{e:action1+}
  \Sigma^{\Psi}_{\mathcal{S}}=\{A(x)>0\,|\,x\;\text{is a generalized}\;\Psi\hbox{-characteristic on}\;\mathcal{S} \}.
  \end{equation}

The above Clarke's result may be formulated as: the boundary of a compact convex set $D$ in $(\mathbb{R}^{2n},\omega_0)$ containing the origin in its interior carries a generalized $\Psi$-{\bf characteristic}.
  Motivated by this and the Weinstein conjecture (\cite{We79})
 we state the following generalized version of the latter.\\

 \noindent{\bf Question $\Psi$}.\quad{\it
Let $\mathcal{S}$ be a hypersurface  of contact type in a symplectic manifold $(M, \omega)$
and $\Psi\in{\rm Symp}(M, \omega)$ such that $\mathcal{S}\cap {\rm Fix}(\Psi)\ne\emptyset$.
 Under what condition does there exist
a $\Psi$-{\rm characteristic} on $\mathcal{S}$ ?} \\

This question is closely related to the following.\\

\noindent{\bf Leaf-wise intersection problem}:
{\it Given a compact hypersurface $\mathcal{S}$ and a symplectomorphism $\Psi\in{\rm Symp}(M,\omega)$,
under what conditions on $\Psi$ and $\mathcal{S}$ {does there exist}  a point $x\in\mathcal{S}$
such that $\Psi x$ lies on {the} leaf $L_{\mathcal{S}}(x)$ through $x$ ? }Such $x$ is called a leaf-wise intersection point for $\Psi$ on $\mathcal{S}$.\\

Such a question was first addressed by Moser \cite{Mos78}.
Since then various forms or generalizations of it were studied.
See \cite{EH89a, Dra08, GinGu15}, \cite[\S1.1]{AF10a} and \cite[\S1.4]{Kan13} and references therein for a brief history of these problems.

Actually,  the above leaf-wise intersection question for a hypersurface $\mathcal{S}$  of contact type
 is slightly weaker than  Question $\Psi$. Indeed, it is clear that a $\Psi$-characteristic $\gamma:[0,T]\rightarrow \mathcal{S}$
yields a leaf-wise intersection point $\gamma(0)$.
Conversely, if $x$ is a leaf-wise intersection point, we take a
 smooth function $H:M\to\mathbb{R}$  having $\mathcal{S}$  as a regular energy surface,
 and obtain $L_{\mathcal{S}}(x)=\{\varphi^t(x)\,|\,t\in\mathbb{R}\}$ and so $\Psi(x)=\varphi^\tau(x)$ for some $\tau\in\mathbb{R}$,
where $\varphi^t$ is the Hamiltonian flow of $H$. If $\tau>0$, then $[0,\tau]\ni t\to\varphi^t(x)$
is a $\Psi$-characteristic on $\mathcal{S}$. If $\tau<0$ then $y:[0, -\tau]\to \mathcal{S},\;t\mapsto\varphi^{-t}(x)$ satisfies
$\dot{y}(t)=-X_H(y(t))=X_{-H}(y(t))$ and so it is a $\Psi$-characteristic on $\mathcal{S}$.
However,  it is possible that $\tau=0$, i.e., $\Psi(x)=x$, and we cannot get a $\Psi$-characteristic on $\mathcal{S}$ in this case.

As applications of our generalized capacities,
some answers to Question $\Psi$ and the leaf-wise intersection question above
are given in Corollaries~\ref{cor:EHcontact1},\ref{cor:EHcontact2}
and Section~\ref{sec:1.2}.
There exist several methods to study the Weinstein conjecture,
which were developed  based on pseudo-holomorphic curve theory,
for example, Gromov-Witten invariants,  symplectic (co)homology and contact homology.
Our future work is to develop the corresponding theories matching to Question $\Psi$.

{\bf Notations and conventions.} A domain in $\mathbb{R}^m$ is a connected open subset of $\mathbb{R}^m$.
For $r>0$ and $p=(p_1,\cdots,p_m)\in \mathbb{R}^m$ we write
\begin{eqnarray*}
 &&B^{m}(p, r)= \left\{(x_1,\cdots,x_m)\in\mathbb{R}^m\,\Big|\,\sum^m_{i=1}(x_i-p_i)^2<r^2\right\},\\
&&B^{m}(r):=B^{m}(0, r)\quad\hbox{and}\quad B^{m}:=B^{m}(1).
\end{eqnarray*}
For $R>0$  we write as usual
\begin{eqnarray*}
Z^{2n}(R)= \{(q,p)\in\mathbb{R}^n\times\mathbb{R}^n\,\, |  \, q_1^2+p_1^2< R^2 \}
  \end{eqnarray*}
with respect to the symplectic coordinates $(q,p)=(q_1,\cdots,q_n, p_1,\cdots,p_n)$ of
 $(\mathbb{R}^{2n},\omega_0)$.

\subsection{An extension of  Hofer-Zehnder symplectic capacity}\label{sec:1.HZ}

Let us recall the definition of the Hofer-Zehnder symplectic capacity.
Given a symplectic manifold $(M,\omega)$ let $\mathcal{H}(M,\omega)$ denote the set of  smooth functions $H \colon M\to\R$ for which
there exists an nonempty open subset $U=U(H)$ and a compact subset
$K=K(H)\subset M\setminus\partial M$ such that
\begin{description}
  \item[(i)] $H|_U=0$,
\item[(ii)] $H|_{M\setminus K}=m(H):=\max H$,
 \item[(iii)]  $0\leq H\leq m(H)$.
\end{description}
Denote by $X_H$ the Hamiltonian vector field defined by {$\omega(X_H, \cdot)=-dH$}.
A function $H\in \mathcal{H}(M,\omega)$  is called {\bf admissible} if
$\dot x=X_H(x)$ has no nonconstant periodic solutions of period less than or equal to $1$.
 Let $\mathcal{H}_{ad}(M,\omega)$ be the set of admissible Hamiltonians on
$(M,\omega)$. The {\bf Hofer-Zehnder symplectic capacity}
$c_{\rm HZ}(M,\omega)$ of $(M,\omega)$ was defined in \cite{HoZe90} by
$$
c_{\rm HZ}(M,\omega) = \sup \left\{\max H\,|\, H\in \mathcal{H}_{ad}(M,\omega) \right\}.
$$
This symplectic invariant may be used to establish the existence of closed
characteristics on an energy surface, {and} $c_{\rm HZ}(M,\omega)<\infty$ implies
that the Weinstein conjecture holds in $(M,\omega)$.

 Given a $\Psi\in{\rm Symp}(M,\omega)$ with ${\rm Fix}(\Psi)\ne\emptyset$, let
$$
\mathcal{H}^\Psi(M,\omega)=\{H\in \mathcal{H}(M,\omega)\,|\, U\cap {\rm Fix}(\Psi)\neq \emptyset\},
$$
where $U=U(H)$ is as in (i)-(iii). We call $H\in\mathcal{H}^\Psi(M,\omega)$
$\Psi$-{\bf admissible} if all solutions $x:[0, T]\to M$ of the Hamiltonian boundary value problem
\begin{equation}\label{bvp}
   \left\{
   \begin{array}{l}
     \dot{x}=X_H(x), \\
     x(T)=\Psi x(0)
   \end{array}
   \right.
\end{equation}
with $0<T\le 1$ are constant.
The set of all such $\Psi$-admissible Hamiltonians is denoted by
 $\mathcal{H}_{ad}^{\Psi}(M,\omega)$. As an analogue of the Hofer-Zehnder capacity of $(M, \omega)$
we call
\begin{equation}\label{cap}
c^\Psi_{\rm HZ}(M,\omega):=\sup \{\max H\,|\, H\in \mathcal{H}_{ad}^{\Psi}(M,\omega)\}
\end{equation}
 {\bf $\Psi$-Hofer-Zehnder capacity} (Abb., $\Psi$-HZ capacity)  or {\bf Hofer-Zehnder capacity relative to $\Psi$} of $(M,\omega)$.
Moreover, for an open subset $O\subset M$ with $O\cap{\rm Fix}(\Psi)\ne\emptyset$, we also define the $\Psi$-HZ capacity of $O$ by
\begin{equation}\label{cap+}
c^\Psi_{\rm HZ}(O,\omega)=\sup \{\max H\,|\, H\in \mathcal{H}_{ad}^{\Psi}(O,\omega)\},
\end{equation}
where $\mathcal{H}_{ad}^{\Psi}(O,\omega)$ consists of
{$H\in\mathcal{H}(O,\omega)$ such that $U(H)\cap{\rm Fix}(\Psi)\ne\emptyset$ and that the boundary value problem
(\ref{bvp}) has a nonconstant solution $x:[0, T]\to O$ implies $T>1$.}
It is not hard to check that { $c^{\Psi|_O}_{\rm HZ}(O,\omega)=c^\Psi_{\rm HZ}(O,\omega)$}
if $\Psi(O)=O$, where
$\Psi|_O$ is viewed as an element in ${\rm Symp}(O, \omega)$.
Moreover, if $\Psi=id_M$ we have clearly $c^\Psi_{\rm HZ}(M,\omega)=c_{\rm HZ}(M,\omega)$
and $c^\Psi_{\rm HZ}(O,\omega)=c_{\rm HZ}(O,\omega)$ for any open subset $O\subset M$.
{As $c_{\rm HZ}$, it follows immediately from the above definition that  $c^\Psi_{\rm HZ}$ has inner regularity}, i.e., for any precompact open subset
$O\subset M$ with $O\cap{\rm Fix}(\Psi)\ne\emptyset$, we have
\begin{equation}\label{cap+++}
c^\Psi_{\rm HZ}(O,\omega)=\sup\{c^\Psi_{\rm HZ}(K,\omega)\,|\, K\;\hbox{open},\;K\cap{\rm Fix}(\Psi) \ne\emptyset,\;
\overline{K}\subset O\}.
\end{equation}

It follows immediately from the definition that $c^\Psi_{\rm HZ}$ has the following properties:

\begin{proposition}\label{MonComf}
\begin{description}
  \item[(i)] {\rm (Conformality)}. $c^\Psi_{\rm HZ}(M,\alpha\omega)=\alpha c^\Psi_{\rm HZ}(M,\omega)$ for any $\alpha\in\mathbb{R}_{>0}$, and
    $c^{\Psi^{-1}}_{\rm HZ}(M,\alpha\omega)=-\alpha c^\Psi_{\rm HZ}(M,\omega)$ for any $\alpha\in\mathbb{R}_{<0}$.
\item[(ii)] {\rm (Monotonicity)}.
Suppose that $\Psi_i\in{\rm Symp}(M_i,\omega_i)$ ($i=1,2$).
If there exists a symplectic embedding
$\phi:(M_1,\omega_1)\to (M_2,\omega_2)$ of codimension zero such that $\phi\circ\Psi_1=\Psi_2\circ\phi$,
then for open subsets $O_i\subset M_i$ with $O_i\cap{\rm Fix}(\Psi_i)\ne\emptyset$ ($i=1,2$) and $\phi(O_1)\subset O_2$,
it holds that $c^{\Psi_1}_{\rm HZ}(O_1,\omega_1)\le c^{\Psi_2}_{\rm HZ}(O_2,\omega_2)$.
\end{description}
 \end{proposition}

 Clearly  Proposition~\ref{MonComf}(ii) shows that $c^{\Psi}_{\rm HZ}(M,\omega)$ is invariant for the centralizer
 of $\Psi$ in ${\rm Symp}(M,\omega)$, denoted by
 ${\rm Symp}_{\Psi}(M,\omega):=\{\phi\in {\rm Symp}(M,\omega)\,|\,\phi\circ\Psi=\Psi\circ\phi\}$
(i.e., the stabilizer at $\Psi$ for the adjoint action on ${\rm Symp}(M,\omega)$).
Moreover, for any $\Psi\in{\rm Symp}(M,\omega)$ and
 any open subset $O\subset M$ with $O\cap{\rm Fix}(\Psi)\ne\emptyset$, (ii) also implies
\begin{equation}\label{e:inv}
c^\Psi_{\rm HZ}(O,\omega)=c_{\rm HZ}^{\Phi\circ\Psi\circ\Phi^{-1}}(\Phi(O),\omega)\quad\forall\Phi\in {\rm Symp}(M,\omega).
\end{equation}

{In this paper, we mainly consider the standard linear symplectic space $(\mathbb{R}^{2n},\omega_0)$ and its linear symplectomorphisms.
We make the following conventions: }each symplectic matrix $\Psi\in{\rm Sp}(2n,\mathbb{R})$
is identified with the linear symplectomorphism on $(\mathbb{R}^{2n},\omega_0)$
which has the representing matrix $\Psi$ under the standard symplectic basis
of $(\mathbb{R}^{2n},\omega_0)$, $(e_1,\cdots,e_n,f_1,\cdots,f_n)$,
where the $i$-th {(resp. $(n+i)$-th)} coordinate of $e_i$ (resp. {$f_{i}$}) is $1$ and other coordinates are zero.

The following continuity holds for $c^{\Psi}_{\rm HZ}$ where $\Psi\in{\rm Sp}(2n,\mathbb{R})$.
\begin{proposition}\label{prop:convconti}
 For a bounded convex domain $A\subset\mathbb{R}^{2n}$, suppose that
$\Psi\in{\rm Sp}(2n, \mathbb{R})$ satisfies $A\cap{\rm Fix}(\Psi)\ne\emptyset$. Then for every $\varepsilon>0$ there exists some $\delta>0$ such that
 for any bounded convex domain $O\subset\mathbb{R}^{2n}$ intersecting with ${\rm Fix}(\Psi)$, it holds that
\begin{equation}\label{cap4+}
 |c^\Psi_{\rm HZ}(O,\omega_0)-c^\Psi_{\rm HZ}(A,\omega_0)|\le
\varepsilon
\end{equation}
provided that $A$ and $O$ have Hausdorff distance $d_{\rm H}(A,O)<\delta$.
\end{proposition}

\begin{proof}
Let $p\in A\cap{\rm Fix}(\Psi)$. Replacing $A$ and $O$
with $A-p$ and $O-p$ respectively, we may assume $0\in A$. For any $0<\epsilon\ll1$,
by \cite[Lemma~1.8.14]{Sch93} there exists $\delta>0$ such that any bounded convex domain $O\subset\mathbb{R}^{2n}$
with $d_{\rm H}(A,O)<\delta$ satisfies
$$
(1-\epsilon)A\subset O\subset (1+\epsilon)A.
$$
Then the result easily follows from Proposition~\ref{MonComf}(i)--(ii).
\end{proof}

As in \cite{Ba95} we can also get more results on continuity of $c^\Psi_{\rm HZ}$.
For example, as in the proof of \cite[Proposition~2.3]{Ba95} we have
the following outer regularity of $c^\Psi_{\rm HZ}$.
Let $\mathcal{S}$ be a smooth connected compact hypersurface of restricted contact type in $\mathbb{R}^{2n}$
with respect to a global Liouville vector field
    $X$ on $\mathbb{R}^{2n}$ and let $B_{\mathcal{S}}$ be the bounded component of
$\mathbb{R}^{2n}\setminus \mathcal{S}$. Suppose that
$\Psi\in{\rm Sp}(2n, \mathbb{R})$ satisfies
$B_{\mathcal{S}}\cap{\rm Fix}(\Psi)\ne\emptyset$ and that
$X(\Psi(x))=\Psi(X(x))$ for all $x$ near $\mathcal{S}$.
Then
$$
c^\Psi_{\rm HZ}(B_{\mathcal{S}},\omega_0)=\inf\{c^\Psi_{\rm HZ}(V,\omega_0)\,|\,V\subset \mathbb{R}^{2n}\;\hbox{is open and}\;\overline{B_{\mathcal{S}}}\subset V\}.
$$

One of the main results of this paper is the following analogue of
the representation formula for $c_{\rm HZ}$ due to Hofer and Zehnder \cite[Propposition~4]{HoZe90}.

\begin{thm}\label{th:convex}
For $\Psi\in{\rm Sp}(2n,\mathbb{R})$, let $D\subset \mathbb{R}^{2n}$ be a convex bounded domain containing a fixed point $p$ of $\Psi$ and
with boundary $\mathcal{S}=\partial D$. Then
there is a generalized $\Psi$-characteristic $x^{\ast}$ on $\mathcal{S}$ such that
\begin{eqnarray}\label{e:action2}
A(x^{\ast})&=&\min\{A(x)>0\,|\,x\;\text{is a generalized}\;\Psi\hbox{-characteristic on}\;\mathcal{S}\}\\
&=&c^\Psi_{\rm HZ}(D,\omega_0).\label{e:action-capacity1}
\end{eqnarray}
If $\mathcal{S}$ is of class $C^{1,1}$, (\ref{e:action2}) and (\ref{e:action-capacity1}) become
   \begin{equation}\label{e:action-capacity2}
   c^\Psi_{\rm HZ}(D,\omega_0)=A(x^{\ast})=\inf\{A(x)>0\,|\,x\;\text{is a}\;\Psi\hbox{-characteristic on}\;\mathcal{S}\}.
   \end{equation}
   \end{thm}

The proof of this theorem is in Section~\ref{sec:convex}.

\begin{remark}\label{rm:carrier}
{\rm A generalized $\Psi$-characteristic $x^{\ast}$ on $\mathcal{S}$  satisfying
(\ref{e:action2})-(\ref{e:action-capacity1}) is called a {\bf $c^\Psi_{\rm HZ}$-carrier} for $D$.
The proof of Theorem~\ref{th:convex} also shows that
a generalized $\Psi$-characteristic  on $\mathcal{S}$ is a $c^\Psi_{\rm HZ}$-carrier for $D$ if and only if
it may be reparametrized as a solution $x:[0,T]\to \mathcal{S}$ of
$-J\dot{x}^\ast(t)\in  \partial H(x^\ast(t))$ with $T=c^\Psi_{\rm HZ}(D,\omega_0)$ and satisfying
$x(T) =\Psi(x(0))$, where $H=j^2_D$.
Since $\{\partial H(x)|x\in\mathcal{S}\}$ is a bounded set in $\mathbb{R}^{2n}$,
 it follows from Arzela-Ascoli theorem that
all $c^\Psi_{\rm HZ}$-carriers for $D$ form a compact subset in $C^0([0,T],\mathcal{S})$
(and $C^1([0,T],\mathcal{S})$ if $\mathcal{S}$ is $C^1$), where $T=c^\Psi_{\rm HZ}(D,\omega_0)$.}
\end{remark}

\begin{remark}
{\rm Clearly, Theorem~\ref{th:convex} implies Clarke's main result in \cite{Cl82}.
When $\Psi=I_{2n}$ and  the boundary $\mathcal{S}=\partial D$ is
smooth, Hofer and Zehnder \cite[Proposition~4]{HoZe90} proved Theorem~\ref{th:convex},
and then K\"unzle \cite{Ku90, Ku96, Ku97} removed the smoothness assumption of $\mathcal{S}$
(also see Artstein-Avidan and  Ostrover \cite{AAO14} for a different proof).}
\end{remark}

Fix a symplectic  matrix  $\Psi\in{\rm Sp}(2n,\mathbb{R})$. Define
  \begin{equation}\label{e:g}
   g^{\Psi}:\mathbb{R}\rightarrow \mathbb{R}, \,s\mapsto \det (\Psi-e^{sJ}),
   \end{equation}
 where  $e^{tJ}=\sum^\infty_{k=0}\frac{1}{k!}t^k J^k$.
 The set of zero points of  $g^{\Psi}$ in $(0, 2\pi]$
 is a nonempty finite set.  Denote by
 \begin{equation}\label{e:TPsi}
\mathfrak{t}(\Psi)
 \end{equation}
 the smallest zero point of $g^{\Psi}$ in $(0, 2\pi]$.
Then $\mathfrak{t}(I_{2n})=2\pi$ and  $\mathfrak{t}(-I_{2n})=\pi$.(See Lemma~\ref{zeros}.)

As a consequence of Theorem~\ref{th:convex} we get
\begin{corollary}\label{cor:ellipsoid}
Let $E(q):=\{z\in\mathbb{R}^{2n}\,|\, q(z)<1\}$  be the ellipsoid given
by a positive definite quadratic form $q(z)=\frac{1}{2}\langle Sz, z\rangle$ on $\mathbb{R}^{2n}$,
where $S\in\mathbb{R}^{2n\times 2n}$ is a positive definite symmetric matrix. Then for any $\Psi\in{\rm Sp}(2n,\mathbb{R})$ there holds
\begin{eqnarray}\label{e:ellCap}
{c}^\Psi_{\rm HZ}(E(q))&=&\inf\{T>0\,|\, \det(\exp(TJS)-\Psi)\ne 0\}\\
&\le&\frac{r_n^2}{2}\inf_\Phi \mathfrak{t}(\Phi\Psi\Phi^{-1}),\label{e:ellCap+}
\end{eqnarray}
where $\Phi\in{\rm Sp}(2n,\mathbb{R})$ satisfies $\Phi(E(q))=\{z\in\mathbb{C}^n\,|\,
\sum^n_{j=1}|z_j/r_j|^2<1\}$ with $0<r_1\le r_2\le\cdots\le r_n$.
In particular, (\ref{e:ellCap}) implies
  \begin{equation}\label{e:ball}
  c^{\Psi}_{\rm HZ}(B^{2n})=\frac{ \mathfrak{t}(\Psi)}{2}.
  \end{equation}
\end{corollary}

\begin{proof}
Since the Hamiltonian vector field of the quadratic form $q(z)$ is $X_q(z)=JSz$, every $\Psi$-characteristic
 on $\partial E(q)$ may be parameterized as the form $[0, T]\ni t\mapsto\exp(tJS)z_0\in \partial E(q)$,
 where $q(z_0)=1$ and $\exp(TJS)z_0=\Psi z_0$. Hence (\ref{e:ellCap})
 follows from (\ref{e:action2})-(\ref{e:action-capacity1}) immediately.
 Observe that $B^{2n}(1)=E(q)$ with $S=2I_{2n}$.
The definition of $\mathfrak{t}(\Psi)$ and (\ref{e:ellCap}) lead to (\ref{e:ball}) directly.

Note that for  $\Psi\in{\rm Sp}(2n, \mathbb{R})$ and
an  open set  $O\subset (\mathbb{R}^{2n},\omega_0)$ containing
the origin, (i)-(ii) of Proposition~\ref{MonComf} imply
\begin{equation}\label{e:inv.1}
c^\Psi_{\rm HZ}(\alpha O,\omega_0)=\alpha^2 c_{\rm HZ}^{\Psi}(O,\omega_0),\quad\forall\alpha\ge 0.
\end{equation}
Since
$$
\Phi(E(q))=\left\{z\in\mathbb{C}^n\,\bigg|\,
\sum^n_{j=1}|z_j/r_j|^2<1\right\}\subset B^{2n}(0,r_n),
$$
it follows from (\ref{e:inv.1}),
Proposition~\ref{MonComf}(ii) and (\ref{e:ball}) that
\begin{eqnarray*}
{c}^\Psi_{\rm HZ}(E(q),\omega_0)&=&c^{\Phi\Psi\Phi^{-1}}_{\rm HZ}(\Phi(E(q)),\omega_0)\\
&\le&c^{\Phi\Psi\Phi^{-1}}_{\rm HZ}(B^{2n}(r_n),\omega_0)=\frac{r_n^2}{2}\mathfrak{t}(\Phi\Psi\Phi^{-1}).
\end{eqnarray*}
(\ref{e:ellCap+}) follows immediately.
\end{proof}

Here is another important consequence of Theorem~\ref{th:convex}.

\begin{corollary}\label{cor:CrokeW1}
Let $\Psi\in{\rm Sp}(2n,\mathbb{R})$ and $D\subset \mathbb{R}^{2n}$ be a convex bounded domain with
 boundary $\mathcal{S}=\partial D$. Suppose that $p\in D$ is a fixed point of $\Psi$.
\begin{description}
\item[(i)] If $D$ contains a ball $B^{2n}(p,r)$,  then for any  generalized
$\Psi$-characteristic $x$ on $\mathcal{S}$ with positive action
  it holds that
\begin{equation}\label{e:croke1}
A(x)\ge \frac{r^2}{2} \mathfrak{t}(\Psi).
\end{equation}
\item[(ii)] If $D\subset B^{2n}(p,R)$,  there exists a
generalized  $\Psi$-characteristic $x^\star$ on $\mathcal{S}$ such that
 \begin{equation}\label{e:croke2}
0<A(x^\star)\le \frac{R^2}{2} \mathfrak{t}(\Psi).
\end{equation}
\end{description}
 \end{corollary}

\begin{remark}
{\rm When $\Psi=I_{2n}$ and $\mathcal{S}$ is of class $C^1$,
(i) and (ii) were obtained respectively by Croke-Weinstein in \cite[Theorem~C]{CrWe81} and
by Ekeland in Proposition~5 of \cite[Chap.5,\S1]{Ek90}. Then  K\"unzle \cite{Ku96, Ku97}
removed the $C^1$-smoothness assumption of $\mathcal{S}$.}
\end{remark}

\begin{proof}[\bf Proof of Corollary~\ref{cor:CrokeW1}]
By a translation transformation (see the beginning of Section~\ref{sec:convex}),
 we only need to consider the case $p=0$.

For (i) of Corollary~\ref{cor:CrokeW1},   $B^{2n}(0,r)\subset D$ implies ${c}^\Psi_{\rm HZ}(B^{2n}(0,r),\omega_0)\le {c}^\Psi_{\rm HZ}(D,\omega_0)$. Moreover,  ${c}^\Psi_{\rm HZ}(B^{2n}(0,r),\omega_0)=\frac{r^2}{2} \mathfrak{t}(\Psi)$ by (\ref{e:inv.1}) and (\ref{e:ball}),
and ${c}^\Psi_{\rm HZ}(D,\omega_0)$ is equal to
the minimum of actions of all $\Psi$-characteristics with positive actions on $\mathcal{S}$
by Theorem~\ref{th:convex}. Thus
(\ref{e:croke1}) follows immediately.

Similarly, for  (ii) of Corollary~\ref{cor:CrokeW1} we have
$$
c^\Psi_{\rm HZ}(D,\omega_0)\le{c}^\Psi_{\rm HZ}(B^{2n}(0,R),\omega_0)=\frac{R^2}{2} \mathfrak{t}(\Psi).
$$
 By Theorem~\ref{th:convex}  there exists a generalized $\Psi$-characteristic $x^\ast$ on $\mathcal{S}$
such that $A(x^\ast)={c}^\Psi_{\rm HZ}(D,\omega_0)$.
\end{proof}

\subsection{An extension of Ekeland-Hofer capacity}\label{sec:1.EH}

  For each symplectic  matrix  $\Psi\in{\rm Sp}(2n,\mathbb{R})$ and $s\ge 0$, there is a Hilbert space $E_\Psi^s$ such that $E_\Psi^0=L^2([0,1],\mathbb{R}^{2n})$ and for $s\ge \frac{1}{2}$, $x\in E_{\Psi}^s$ implies that $x$ is continuous and satisfies $x(1)=\Psi x(0)$. $E_{\Psi}^{\frac{1}{2}}$ is the variational space we need in this article. When $\Psi=I_{2n}$, $E_{I_{2n}}^{\frac{1}{2}}$ is exactly $H^{\frac{1}{2}}(S^1,\mathbb{R}^{2n})$. For details see Section~\ref{section:space}.

In this subsection we fix a $\Psi\in{\rm Sp}(2n,\mathbb{R})$. Then $\mathbb{E}:=E_{\Psi}^{\frac{1}{2}}$ has an orthogonal splitting 
$$
\mathbb{E}=\mathbb{E}^{-}\oplus \mathbb{E}^0 \oplus \mathbb{E}^{+}
$$
(see (\ref{e:spaceDecomposition})).
We closely follow  Sikorav's approach \cite{Sik90} to Ekeland-Hofer capacity in \cite{EH89} and construct an analogue of the classical Ekeland-Hofer capacity.

\begin{definition}\label{defdeform}
  A continuous map $\gamma:\mathbb{E}\rightarrow \mathbb{E}$ is called an {\bf admissible deformation}
   if there exists a homotopy $(\gamma_u)_{0\le u\le 1}$ such that $\gamma_0={\rm id}$, $\gamma_1=\gamma$ and satisfies
\begin{description}
\item[(i)] $\forall u\in [0,1]$, $\gamma_u(\mathbb{E}\setminus(\mathbb{E}^-\oplus \mathbb{E}^0))\subset \mathbb{E}\setminus(\mathbb{E}^-\oplus \mathbb{E}^0)$,
 i.e. for any $x\in \mathbb{E}$ such that $x^+\neq 0$, there holds $\gamma_u(x)^+\neq 0$.
\item[(ii)]  $\gamma_u(x)=a(x,u)x^++b(x,u)x^0+c(x,u)x^-+K(x,u)$, where $(a,b,c,K)$ is a continuous map from $\mathbb{E}\times [0,1]$ to $(0,+\infty)^3\times \mathbb{E}$ and maps bounded sets to precompact sets.
\end{description}
\end{definition}
Let $\Gamma$ be the set of all admissible deformations.
For $H\in C^0(\mathbb{R}^{2n},\mathbb{R}_{\ge 0})$ satisfying:
\begin{description}
\item[(H1)] ${\rm Int}(H^{-1}(0))\ne \emptyset$ and contains a fixed point of $\Psi$,
\item[(H2)] there exists $z_0\in{\rm Fix}(\Psi)$, real numbers $a> \mathfrak{t}(\Psi)$ and $b$
such that $H(z)=a|z|^2+ \langle z, z_0\rangle+ b$ outside a compact subset of $\mathbb{R}^{2n}$,
\end{description}
we define $\Phi_H:\mathbb{E}\to\mathbb{R}$ by
\begin{eqnarray}\label{e:EH.1.1}
\Phi_H(x)= \frac{1}{2}(\|x^+\|^2_{\mathbb{E}}-\|x^-\|^2_{\mathbb{E}})-\int_0^1H(x(t))dt,
\end{eqnarray}
 and the $\Psi$-{\bf capacity} of $H$ by
\begin{eqnarray}\label{e:EH.1.2}
c^{\Psi}_{\rm EH}(H)=\sup_{{\gamma}\in\Gamma}\inf_{x\in {\gamma}(S^+)}\Phi_H(x),\quad\hbox{where}
\quad S^+=\{x\in \mathbb{E}^+\,|\,\|x\|_{\mathbb{E}}=1\}.
\end{eqnarray}
Then (H2) implies $c^{\Psi}_{\rm EH}(H)<+\infty$,  and  the conditions (H1)--(H2) imply $c^{\Psi}_{\rm EH}(H)>0$
if $H$ is smooth.  (See Propositions~\ref{prop:EH.1.3},~\ref{prop:EH.1.4}.)

It is easy to prove the following.

\begin{proposition}\label{prop:EH.1.2}
Let $H$, $K \in C^0(\mathbb{R}^{2n},\mathbb{R}_{\ge 0})$ satisfy (H1) and (H2).
\begin{description}
\item[(i)]{\rm (Monotonicity)} If $H\le K$ then $c^{\Psi}_{\rm EH}(H)\ge c^{\Psi}_{\rm EH}(K)$.
\item[(ii)] {\rm (Continuity)} $|c^{\Psi}_{\rm EH}(H)-c^{\Psi}_{\rm EH}(K)|\le \sup\{|H(z)-K(z)|\,|\, z\in\mathbb{R}^{2n}\}$.
\item[(iii)] {\rm (Homogeneity)}  $c^{\Psi}_{\rm EH}(\lambda^2H(\cdot/\lambda))=\lambda^2 c^{\Psi}_{\rm EH}(H)$ for $\lambda\ne 0$.
\end{description}
\end{proposition}

Let
\begin{eqnarray}\label{e:EH.1.5.1}
&&\mathscr{F}(\mathbb{R}^{2n})=\{H\in C^0(\mathbb{R}^{2n},\mathbb{R}_{\ge 0})\,|\,H\;\hbox{satisfies (H2)}\},
\\
&&\mathscr{F}(\mathbb{R}^{2n},B)=\{H\in \mathscr{F}(\mathbb{R}^{2n})\,|\,H\;\hbox{vanishes near $\overline{B}$}\}\label{e:EH.1.5.2}
\end{eqnarray}
for each $B\subset\mathbb{R}^{2n}$ such that $B\cap {\rm Fix}(\Psi)\neq \emptyset$.
 We define
\begin{equation}\label{e:EH.1.6}
c^{\Psi}_{\rm EH}(B)=\inf\{c^{\Psi}_{\rm EH}(H)\,|\, H\in \mathscr{F}(\mathbb{R}^{2n},B)\}\in [0, +\infty)
\end{equation}
if $B$ is bounded and $B\cap {\rm Fix}(\Psi)\neq \emptyset$, and
\begin{equation}\label{e:EH.1.7}
c^{\Psi}_{\rm EH}(B)=\sup\{c^{\Psi}_{\rm EH}(B')\,|\, B'\subset B,\;\hbox{$B'$ is bounded and $B'\cap{\rm Fix}(\Psi)\neq \emptyset$}\}
\end{equation}
if $B$ is unbounded and $B\cap {\rm Fix}(\Psi)\neq \emptyset$. We call $c^{\Psi}_{\rm EH}(B)$ in (\ref{e:EH.1.6}) and (\ref{e:EH.1.7})
{\bf $\Psi$-Ekeland-Hofer capacity} (Abb., $\Psi$-EH capacity)  or {\bf Ekeland-Hofer capacity relative to $\Psi$} of $B$.

We say $H\in C^0(\mathbb{R}^{2n},\mathbb{R}_{\ge 0})$ to be $\Psi$-{\bf  nonresonant}
if it satisfies (H2) with $\det(e^{2aJ}-\Psi)\neq 0$.
For each $B\subset\mathbb{R}^{2n}$ such that $B\cap {\rm Fix}(\Psi)\neq \emptyset$ we write
\begin{eqnarray*}
\mathscr{E}(\mathbb{R}^{2n},B, \Psi)=\{H\in \mathscr{F}(\mathbb{R}^{2n},B)\,|\,H\;\hbox{is $\Psi$-nonresonant}\}.
\end{eqnarray*}
Note that each $H\in \mathscr{E}(\mathbb{R}^{2n},B, \Psi)$ satisfies (H1) and that
$\mathscr{E}(\mathbb{R}^{2n},B, \Psi)$ is a {\bf cofinal family} of $\mathscr{F}(\mathbb{R}^{2n},B)$,
that is, for any $H\in\mathscr{F}(\mathbb{R}^{2n},B)$ there exists
$G\in \mathscr{E}(\mathbb{R}^{2n},B, \Psi)$ such that $G\ge H$.

\begin{remark}\label{rem:EH6}
{\rm
\begin{description}
\item[(i)] $c^\Psi_{\rm EH}(B)=c^\Psi_{\rm EH}(\bar{B})$.
\item[(ii)]  $\mathscr{F}(\mathbb{R}^{2n},B)$ in  (\ref{e:EH.1.6})-(\ref{e:EH.1.7}) can be replaced by its cofinal subset $\mathscr{E}(\mathbb{R}^{2n},B,\Psi)$, and
    can also be replaced by a smaller cofinal subset
   $\mathscr{E}(\mathbb{R}^{2n},B,\Psi)\cap C^\infty(\mathbb{R}^{2n},\mathbb{R}_{\ge 0})$.
\item[(iii)]  $c^{\Psi}_{\rm EH}(B+w)=c^{\Psi}_{\rm EH}(B)\;\forall w\in {\rm Fix}(\Psi)$,
where $B+w=\{z+w\,|\, z\in B\}$.
   \end{description}}
\end{remark}

The following properties of the $\Psi$-EH capacity of subsets in
 $\mathbb{R}^{2n}$ which contains fixed points of $\Psi$ follow easily from its definition and Proposition~\ref{prop:EH.1.2}.

\begin{proposition}\label{prop:EH.1.7}
Let $B\subset B'\subset\mathbb{R}^{2n}$.
Assume in addition that $B\cap{\rm Fix}(\Psi)\neq\emptyset$.
Then
\begin{description}
\item[(i)]{\rm (Monotonicity)} $c^{\Psi}_{\rm EH}(B)\le c^{\Psi}_{\rm EH}(B')$.
\item[(ii)] {\rm (Conformality)} $c^{\Psi}_{\rm EH}(\lambda B)=\lambda^2 c^{\Psi}_{\rm EH}(B),\;\forall\lambda>0$.
\item[(iii)] {\rm (Exterior regularity)} $c^{\Psi}_{\rm EH}(B)=\inf\{c^{\Psi}_{\rm EH}(U_\epsilon(B))\,|\,\epsilon>0\}$, where $U_\epsilon(B)$ is the $\epsilon$-neighborhood of $B$.
\end{description}
\end{proposition}

Moreover, let $\mathcal{S}, B_{\mathcal{S}}\subset \mathbb{R}^{2n}$
 and $X$, $\Psi$ be as below the proof of Proposition~\ref{prop:convconti}.
Using a proof similar to that of \cite[Proposition~2.3]{Ba95} we have
the following inner regularity of $c^\Psi_{\rm EH}$:
$$
c^\Psi_{\rm EH}(B_{\mathcal{S}},\omega_0)=\sup\{c^\Psi_{\rm EH}(V,\omega_0)\,|\,V\subset \mathbb{R}^{2n}\;\hbox{is open and}\;\overline{V}\subset B_{\mathcal{S}}\}.
$$

The following theorem gives variational explanation for $c^\Psi_{\rm EH}$, which is important for proofs of several subsequent theorems.
\begin{thm}\label{th:EH.1.6}
If $H\in C^\infty(\mathbb{R}^{2n},\mathbb{R})$ satisfies (H1)-(H2),
and is also $\Psi$-nonresonant, then $c^\Psi_{\rm EH}(H)$ is a positive critical value of $\Phi_H$ on $\mathbb{E}$.
\end{thm}
The proof of the above theorem is closely related to  Sikorav's approach in \cite{Sik90} and it will be completed by several propositions in Section~\ref{sec:EH.2}.

For $c^\Psi_{\rm EH}$ we have the following
representation formula, which generalizes the one for $c_{\rm EH}$ in \cite{EH89, EH90, Sik90}. We give its proof in Section~\ref{sec:EH.3}.

\begin{thm}\label{th:EHconvex}
Let $\Psi\in{\rm Sp}(2n,\mathbb{R})$ and $D\subset \mathbb{R}^{2n}$ be a  convex bounded domain with $C^{1,1}$  boundary $S=\partial D$ and  containing a fixed point  $p$ of $\Psi$
in the closure of $D$.
 Then  there exists a  $\Psi$-characteristic $x^{\ast}$ on $\partial D$ such that
 \begin{eqnarray}\label{fixpt}
A(x^{\ast})&=&\min\{A(x)>0\,|\,x\;\text{is a }\;\Psi\hbox{-characteristic on}\;\mathcal{S}\}\nonumber\\
&=&c^\Psi_{\rm EH}(D).
\end{eqnarray}
Moreover, if both $\partial D$ and $D$ contain fixed points of $\Psi$, then
 \begin{eqnarray}\label{fixpt.1}
c^\Psi_{\rm EH}(D)=c^{\Psi}_{\rm EH}(\partial D).
\end{eqnarray}
\end{thm}

\begin{remark}\label{rem:EHZconvex}
{\rm An approximation argument shows that
the condition ``$C^{1,1}$" for $D$ in  Theorem~\ref{th:EHconvex} is
not needed if ``$\Psi$-characteristic" is replaced by
``generalized $\Psi$-characteristic" (for details see {\rm Section~\ref{sec:convex4}) .
Thus this and Theorem~\ref{th:convex} imply
$c^{\Psi}_{\rm EH}(D)=c^{\Psi}_{\rm HZ}(D,\omega_0)$ for any  convex bounded domain
$D\subset \mathbb{R}^{2n}$ containing a fixed point $p$ of $\Psi$.
It follows from the definitions of both $c^{\Psi}_{\rm EH}$ and $c^{\Psi}_{\rm HZ}$
that $c^{\Psi}_{\rm EH}(D)=c^{\Psi}_{\rm HZ}(D,\omega_0)$  for any convex domain
$D\subset \mathbb{R}^{2n}$ containing a fixed point $p$ of $\Psi$, {not
necessarily bounded}.
Hereafter we shall use $c^{\Psi}_{\rm EHZ}(D)$ to denote $c^{\Psi}_{\rm EH}(D)=c^{\Psi}_{\rm HZ}(D,\omega_0)$
  without special statements. In this case a $c^\Psi_{\rm HZ}$-carrier
 is also called a $c^\Psi_{\rm EHZ}$-carrier}.}
\end{remark}

\begin{remark}
{\rm Recently, Artstein-Avidan and Ostrover \cite{AAO08}  established a Brunn-Minkowski type inequality for the
Ekeland-Hofer-Zehnder symplectic capacity $c_{\rm EHZ}$ of convex domains,  and in \cite{AAO14} used it  to derive  several
very interesting bounds and inequalities for the length of the shortest periodic billiard trajectory in a
smooth convex body in $\mathbb{R}^n$. In \cite{JinLu} we showed that  a Brunn-Minkowski type inequality for the
capacity $c^\Psi_{\rm EHZ}$ of convex domains is still true, and also proved some corresponding results
for a larger class of (non-periodic) billiard trajectories in a smooth convex body in $\mathbb{R}^n$.
These will be published elsewhere.}
\end{remark}

{For integers $n_1>0$ and $n_2>0$, let $n=n_1+n_2$ and
$$
\omega^{(1)}_0=\sum^{n_1}_{i=1}
dq_i\wedge dp_i\quad\hbox{and}\quad
\omega^{(2)}_0=\sum^n_{i=n_1+1}dq_i\wedge dp_i.
$$
For $S_i=\left(\begin{array}{cc}
       A_i & B_i \\
       C_i & D_i \\
     \end{array}
   \right)\in{\rm Sp}(2n_i,\mathbb{R})$ with $A_i,B_i,C_i,D_i\in \mathbb{R}^{n_i\times n_i}$, $i=1,2$, denote
   $$
     S_1\oplus S_2=\left(\begin{array}{cccc}
       A_1 &0 & B_1 &0 \\
       0   &A_2 &0 &B_2\\
       C_1 &0 &D_1 &0\\
       0 &C_2 &0 &D_2
     \end{array}
   \right).
      $$
Then
$S_1\oplus S_2\in \rm{Sp}(2n,\mathbb{R})$. Furthermore, $(S_1\oplus S_2)\oplus S_3=S_1\oplus(S_2\oplus S_3)$ for $S_i\in{\rm Sp}(2n_i,\mathbb{R})$, $i=1,2,3$.
Thus we can define $S_1\oplus\cdots\oplus S_k$ for $S_i\in{\rm Sp}(2n_i,\mathbb{R})$, $i=1,\cdots, k$.}

As a generalization of \cite[Th. 6.6.1]{Sik90} (also see \cite[Prop.5]{EH90} for special cases) we have the following theorem, which will be proved in Section~\ref{sec:EH.3}.

\begin{thm}\label{th:EHproduct}
Let $\Psi=\Psi_1\oplus\cdots\oplus\Psi_k$, where $\Psi_i\in{\rm Sp}(2n_i,\mathbb{R})$, $i=1,\cdots,k$.
 Then for compact convex subsets $D_i\subset\mathbb{R}^{2n_i}$ containing fixed points of
 $\Psi_i$ ($1\le i\le k$) it holds that
 \begin{equation}\label{e:product1}
 c^{\Psi}_{\rm EH}(D_1\times\cdots\times D_k)=\min_ic^{\Psi_i}_{\rm EH}(D_i).
 \end{equation}
 Moreover, if both $\partial D_i$ and ${\rm Int}(D_i)$ contain fixed points of $\Psi_i$
 for each $i=1,\cdots,k$, then
 \begin{equation}\label{e:product2}
 c^{\Psi}_{\rm EH}(\partial D_1\times\cdots\times \partial D_k)=\min_ic^{\Psi_i}_{\rm EH}(D_i).
 \end{equation}
 \end{thm}

An immediate consequence is:

\begin{corollary}\label{cor:EHproduct}
For $\Psi:=\Psi_1\oplus\cdots\oplus\Psi_n$
where each $\Psi_i\in{\rm Sp}(2,\mathbb{R})$ has the eigenvalue $1$  and $T^n=S^1(r_1)\times\cdots\times S^1(r_n)\subset\mathbb{R}^{2n}$, it holds that
$$
c^{\Psi}_{\rm EH}(T^n)=\frac{1}{2}\inf_i\left\{\mathfrak{t}(\Psi_i) r_i^2\right\}.
$$
\end{corollary}

\begin{remark}
{\rm Recall that a symplectic matrix $M\in {\rm Sp}(2,\mathbb{R})$ has nonzero fixed points
if and only if $1$ is its unique eigenvalue. All such symplectic matrixes have the form
\begin{equation}\label{e:eigen1}
\left( \begin{array}{cc}
             r & z \\
             z & (1+z^2)/r \\
           \end{array}
         \right)
         \left(
           \begin{array}{cc}
             \cos\theta & -\sin\theta\\
             \sin\theta & \cos\theta \\
           \end{array}
         \right)
\end{equation}
where $r\in (0, \infty)$, $z\in\mathbb{R}$ and $\theta\in \mathbb{R}/(2\pi\mathbb{Z}-\pi)$
satisfy $(r^2+z^2+1)\cos\theta=2r$, see \cite[page 48]{Long02}.}
\end{remark}

Note that (\ref{e:ball}) is a generalization for the normality
 $c_{\rm EH}(B^{2n}(1))=c_{\rm HZ}(B^{2n}(1), \omega_0)=\pi$.
In order to get some kind of results similar to the normality $c_{\rm EH}(Z^{2n}(1))=c_{\rm HZ}(Z^{2n}(1),\omega_0)=\pi$
we need to make stronger restrictions for the  symplectic  matrix  $\Psi\in{\rm Sp}(2n,\mathbb{R})$.

\begin{thm}\label{main2}
   For $\Psi\in{\rm Sp}(2n,\mathbb{R})$,
   suppose that there exists  $P\in {\rm Sp}(2n,\mathbb{R})$ such that
   $P^{-1}\Psi P=S_1\oplus S_2$ for some $S_1\in {\rm Sp}(2 ,\mathbb{R})$ and $S_2\in {\rm Sp}(2n-2,\mathbb{R})$. Then with $W^{2n}_{\Psi}(1):=PZ^{2n}(1)$ it holds that
   $$
   c^\Psi_{\rm HZ}(W^{2n}_\Psi(1),\omega_0)=c^{S_1\oplus S_2}_{\rm HZ}(Z^{2n}(1),\omega_0)=
   c^{S_1\oplus S_2}_{\rm EH}(Z^{2n}(1))
   =\frac{1}{2}\mathfrak{t}(S_1).
   $$
\end{thm}

\begin{proof}
The first (resp. second) equality follows from (\ref{e:inv})
(resp. Remark~\ref{rem:EHZconvex}) directly.
Take $n_1=1$, $n_2=n-1$ and $D_1=B^2(1)\subset\mathbb{R}^2(q_1,p_1)$
and $D_2=\mathbb{R}^{2n-2}(q_2,\cdots,q_n;p_2,\cdots,p_n)$  in (\ref{e:product1})
and let  $\omega^{(1)}_0=dq_1\wedge dp_1$ and $\omega^{(2)}_0=\sum^{n}_{i=2}dq_i\wedge dp_i$. We get
\begin{eqnarray*}
  c^{S_1\oplus S_2}_{\rm EH}(Z^{2n}(1))
  =\min\{c_{\rm EH}^{S_1}(B^2(1)), c_{\rm EH}^{S_2}(\mathbb{R}^{2n-2})\}
  =c_{\rm EH}^{S_1}(B^2(1))=\frac{1}{2}\mathfrak{t}(S_1).
   \end{eqnarray*}
\end{proof}

\begin{remark}
{\rm For $\Psi\in{\rm Sp}(2n,\mathbb{R})$ satisfying conditions in Theorem~\ref{main2}, there
exists an unbounded domain $W^{2n}_{\Psi}(1):=PZ^{2n}(1)$ with finite $\Psi$-HZ capacity.}
\end{remark}

\begin{remark}
{\rm Every orthogonal symplectic matrix $\Psi$ satisfies conditions in Theorem~\ref{main2}.
In fact, a symplectic matrix $\Psi\in{\rm Sp}(2n,\mathbb{R})$ is  orthogonal if and only if
 \begin{equation}\label{US}
  \Psi= \left(\begin{array}{cc}
       U & -V \\
       V & U \\
     \end{array}
   \right)
   \end{equation}
where $U, V\in {\rm GL}(n,\mathbb{R})$ satisfy  $U+\sqrt{-1}V\in {\rm U}(n,\mathbb{C})$.
  Moreover, the unitary matrix $U+\sqrt{-1}V$ can be diagonalized by another unitary matrix $S=T+\sqrt{-1}Q \in {\rm U}(n,\mathbb{C})$, i.e.,
   \begin{equation}\label{Udiagonal}
   U+\sqrt{-1}V=S{\rm diag}(e^{\sqrt{-1}\theta_1},\cdots, e^{\sqrt{-1}\theta_n})S^{-1},
   \end{equation}
   where $0< \theta_1\leq\cdots\leq \theta_n\leq 2\pi$ are uniquely determined by $U+\sqrt{-1}V$.
    Then the orthogonal symplectic matrix
   \begin{equation}\label{othsymp}
   P:=\left(
      \begin{array}{cc}
         T & -Q \\
           Q & T \\
            \end{array}
            \right)
    \end{equation}
   satisfies
   \begin{equation}\label{similar}
   \Psi=P\widetilde{\Psi}P^{-1},
   \end{equation}
   where
   \begin{equation}\label{diagonal}
   \widetilde{\Psi}=\left(
                   \begin{array}{cc}
                     {\rm diag}(\cos \theta_1,\cdots,\cos\theta_n) & -{\rm {\rm diag}}(\sin \theta_1,\cdots, \sin\theta_n) \\
                     {\rm diag}(\sin \theta_1,\cdots,\sin\theta_n) & {\rm {\rm diag}}(\cos \theta_1,\cdots,\cos\theta_n) \\
                   \end{array}
                 \right).
   \end{equation}
  Note that $\widetilde{\Psi}_1\oplus\cdots\oplus \widetilde{\Psi}_n$, where
$$
\widetilde{\Psi}_i=\left(
                   \begin{array}{cc}
                     \cos \theta_i & -\sin \theta_i \\
                     \sin \theta_i, & \cos \theta_i \\
                   \end{array}
                 \right),\qquad i=1,\cdots,n.
$$
Then by Theorem~\ref{main2} we obtain for  $W^{2n}_\Psi(1):=PZ^{2n}(1)$
\begin{equation}\label{circle.1}
c^\Psi_{\rm HZ}(W^{2n}_\Psi(1),\omega_0)=c^{\widetilde{\Psi}}_{\rm HZ}(Z^{2n}(1),\omega_0)=\frac{1}{2}\mathfrak{t}(\widetilde{\Psi}_1)=\frac{\theta_1}{2}.
\end{equation}
Moreover,  Lemma~\ref{zeros.1} leads to
$\mathfrak{t}(\widetilde{\Psi})=\mathfrak{t}(\Psi)=\theta_1$.
Therefore (\ref{circle.1}) becomes
$$
c^\Psi_{\rm HZ}(W^{2n}_\Psi(1),\omega_0)=c^{\widetilde{\Psi}}_{\rm HZ}(Z^{2n}(1),\omega_0)=\frac{\theta_1}{2}=\frac{\mathfrak{t}({\Psi})}{2}.
$$
}
\end{remark}

A compact boundaryless smooth connected hypersurface $\mathcal{S}$ in $(\mathbb{R}^{2n},\omega_0)$ is
 said of {\bf restricted contact type}  if  there exists a
  a globally defined Liouville vector field $X$
 (i.e. a  smooth vector field $X$ on $\mathbb{R}^{2n}$ satisfying
 $L_X\omega_0=\omega_0$) which is transversal to $\mathcal{S}$.
The following is a generalization of \cite[Proposition~6]{EH89}.
Its proof is given in Section~\ref{sec:EH.4}.

\begin{thm}\label{th:EHcontact}
Let $\Psi\in{\rm Sp}(2n,\mathbb{R})$ and $\mathcal{S}$ be a
hypersurface of restricted contact type in $(\mathbb{R}^{2n},\omega_0)$
 that admits a globally defined Liouville vector field $X$ transversal to it such that
\begin{equation}\label{e:EHcontact}
 X(\Psi z)=\Psi X(z),\,\forall z\in\mathbb{R}^{2n}.
 \end{equation}
 Suppose that $\mathcal{S}$ contains a fixed point of $\Psi$.  Then
 $$
   c^{\Psi}_{\rm EH}(B)=c^{\Psi}_{\rm EH}(\mathcal{S})\in\Sigma^{\Psi}_{\mathcal{S}},
   $$
  where $B$ is the bounded component of $\mathbb{R}^{2n}\setminus\mathcal{S}$.
\end{thm}

 Bates \cite{Ba98} extended \cite[Proposition~6]{EH89} to certain domains
 whose boundaries are not of restricted contact type. The corresponding  generalizations of
Theorem~\ref{th:EHcontact} are also possible.

\begin{corollary}\label{cor:EHcontact1}
Under the assumptions of Theorem~\ref{th:EHcontact}, $\mathcal{S}$ carries a $\Psi$-characteristic $\gamma$ with action
$c^{\Psi}_{\rm EH}(\mathcal{S})$. In particular, there exists a leaf-wise intersection point $\gamma(0)\in\mathcal{S}$
for $\Psi$.
\end{corollary}

  Therefore we get a positive answer to
Question $\Psi$ under the assumptions of Theorem~\ref{th:EHcontact}.
 For a centrally symmetric hypersurface $\mathcal{S}$ of restricted contact type in $(\mathbb{R}^{2n},\omega_0)$, if
   $[0,1]\ni t\to\Psi_t$ is an isotopy of the identity in ${\rm Symp}(\mathbb{R}^{2n},\omega_0)$
  which is odd, i.e., $\Psi_t(-x)=-\Psi_t(x)$ for all $(t,x)\in [0,1]\times \mathbb{R}^{2n}$,
  Ekeland and Hofer proved in  \cite{EH89a} that  $\mathcal{S}$ carries a leaf-wise intersection point for $\Psi_1$.
   Moreover, if this  $\mathcal{S}$ is also star-shaped, Albers and Frauenfelder \cite{AF12}
  strengthened this result and showed that $\mathcal{S}$ carries {infinitely} many
  leaf-wise intersection points or a leaf-wise intersection point  which sits on a closed characteristic.
   Hence for any given $\Psi\in{\rm Sp}(2n,\mathbb{R})$,  every centrally
   symmetric hypersurface of restricted contact type in $(\mathbb{R}^{2n},\omega_0)$
  carries a leaf-wise intersection point for $\Psi$ by Ekeland-Hofer theorem, and
  every centrally symmetric star-shaped hypersurface  in $(\mathbb{R}^{2n},\omega_0)$
  carries infinitely many leaf-wise intersection points or
  a leaf-wise intersection point  which sits on a closed characteristic  by Albers-Frauenfelder theorem.
  For a hypersurface $\mathcal{S}$ of restricted contact type in $(\mathbb{R}^{2n},\omega_0)$
  and a $\Psi\in{\rm Ham}_c(\mathbb{R}^{2n},\omega_0)$, Hofer \cite{Ho90}
  showed that $\mathcal{S}$  carries a leaf-wise intersection point if $\Psi$ has Hofer's norm $\|\Psi\|_H\le c_{\rm EH}(\mathcal{S})$.
    The final restriction on norm $\|\Psi\|_H$ is not needed if the Rabinowitz Floer homology of $(\mathbb{R}^{2n}, \mathcal{S})$
  does not vanish by \cite[Theorem~C]{AF10a}. Corollary~\ref{cor:EHcontact1} cannot be included in past results.
Corollary~\ref{cor:EHcontact1} also implies the following  generalization of the main result in
 \cite{Ra78} if the surfaces considered therein are smooth.

\begin{cor}\label{cor:EHcontact2}
Let $\Psi\in{\rm Sp}(2n, \mathbb{R}^{2n})$ and
let $\mathcal{S}\subset (\mathbb{R}^{2n},\omega_0)$ be a smooth star-shaped hypersurface with respect to
a fixed point $p$ of $\Psi$, that is, $p$ is an interior point of the bounded part surrounded by $\mathcal{S}$ and has the property that every ray issuing from the point $p$ intersects $\mathcal{S}$ in exactly one point and so transversally. Then $\mathcal{S}$ carries a $\Psi$-characteristic with action $c^{\Psi}_{\rm EH}(\mathcal{S})$, in particular,  a leaf-wise intersection point for $\Psi$.
\end{cor}

Indeed, let $\phi$ be the translation $x\mapsto x-p$, $\forall x\in\mathbb{R}^{2n}$.
Then $\phi\in {\rm Symp}(\mathbb{R}^{2n},\omega_0)$ and commutes with $\Psi$.
Then $\phi(\mathcal{S})$ is  a star-shaped hypersurface with respect to the origin.
As in the proof of Theorem~\ref{th:convex} $\phi$ maps  $\Psi$-characteristics  on $\mathcal{S}$
onto such characteristics  on $\phi(\mathcal{S})$ in one-to-one and preserving-action way.
Hence we can assume that $\mathcal{S}$ is a star-shaped hypersurface with respect to
the origin. Then $\mathcal{S}$ is a hypersurface of restricted contact type in $(\mathbb{R}^{2n},\omega_0)$
that admits a globally defined Liouville vector field $X$ given by $X(x)=\frac{x}{2}$ for $x\in \mathbb{R}^{2n}$. Clearly, $X$
satisfies (\ref{e:EHcontact}) and hence Corollary~\ref{cor:EHcontact1} yields Corollary~\ref{cor:EHcontact2}.

Recently,  Ekeland proposed a very closely related problem \cite[Problem~4, \S5]{Ek17},
which may be formulated as follows in our notations.

\noindent{\bf Problem E}.\quad
For a general Hamiltonian $H$ on $\mathbb{R}^{2n}$, assuming simply $H(0) = 0$
and $H(x)\to\infty$  when $|x|\to\infty$, so that energy surfaces $H(x)=h$ are
bounded, what about the existence and multiplicity of solutions to the boundary-value problem (\ref{bvp}) with $\Psi\in{\rm Sp}(2n,\mathbb{R})$?

{If $H^{-1}(h)$ is regular and star-shaped,
 Corollary~\ref{cor:EHcontact2} implies that
 either the boundary-value problem (\ref{bvp}) or
$$
  \dot{x}=X_H(x)\quad\&\quad x(T)=\Psi^{-1} x(0)
 $$
  has a solution on $H^{-1}(h)$.}

\subsection{Applications to Hamiltonian dynamics}\label{sec:1.2}

Recall that a thickening of a compact and regular energy surface
$\mathcal{S}=\{x\in M\,|\, H(x)=0\}$ in a symplectic manifold $(M,\omega)$
is an open and bounded neighborhood $U$ of $\mathcal{S}$ which is filled
with compact and regular energy surfaces having energy values near
$E=0$, that is,
\begin{equation}\label{thicken}
U=\bigcup_{\lambda \in I}\mathcal{S}_\lambda,
\end{equation}
where $I=(-\varepsilon,\varepsilon)$ and
$\mathcal{S}_\lambda=\{x\in U\,|\,H(x)=\lambda\}$ is diffeomorphic to
the given surface $\mathcal{S}=\mathcal{S}_0$ for each $\lambda\in I$.
Suppose that $\varepsilon\le 1$.
As a generalization of  \cite[p. 106, Theorem~1]{HoZe94} we have:

\begin{thm}\label{dense}
   Let $\mathcal{S}$ , $\mathcal{S}_\lambda$ and $U$ be described as above, and $\Psi\in{\rm Symp}(M, \omega)$. Assume that $\mathcal{S}\cap {\rm Fix}(\Psi)\ne\emptyset$ and that $c^{\Psi}_{\rm HZ}(U,\omega)<\infty$.
   Then there exists a sequence $\lambda_j\rightarrow 0$ such that
   each energy surface $\mathcal{S}_{\lambda_j}$ carries a $\Psi$-characteristic.
\end{thm}

The proof is standard. Pick an $\varepsilon_0\in (0,\varepsilon)$ and
 a $C^\infty$ function $f:\mathbb{R}\rightarrow \mathbb{R}_{\ge 0}$ such that
\begin{eqnarray*}
f(s)=c^{\Psi}_{\rm HZ}(U,\omega)+1\; \text{for}\;    \varepsilon_0<|s|<\varepsilon, \quad &
f(s)=0\;\text{for}& |s|<\frac{\varepsilon_0}{2},\\
f'(s)<0\;\text{for}\; -\varepsilon_0<s<-\frac{\varepsilon_0}{2},\quad &
f'(s)>0\;\text{for}& \frac{\varepsilon_0}{2}<s<\varepsilon_0.
\end{eqnarray*}
Let $F(x)=f(H(x))$ for $x\in U$. Note that $\mathcal{S}=\mathcal{S}_0$  satisfies
$\mathcal{S}\cap{\rm Fix}(\Psi)\ne \emptyset$. Hence $F \in\mathcal{H}^\Psi(U,\omega) $ with $m(F)=c^{\Psi}_{\rm HZ}(U,\omega)+1$.
 By the definition of $c^{\Psi}_{\rm HZ}(U,\omega)$
there exists a nonconstant smooth curve $x:[0, 1]\to \mathbb{R}^{2n}$ satisfying $\dot{x}=X_F(x(t))=f'(H(x(t)))X_H(x(t))$ and $x(1)=\Psi x(0)$. Clearly
$H(x(t))$ is constant and $x$ is a $\Psi$-characteristic sitting  on $\mathcal{S}_{\varepsilon'}$,
 where {$\varepsilon_0/2<|\varepsilon'|<\varepsilon_0$}.
By choosing $\varepsilon_0$ sufficiently small,  Theorem~\ref{dense} follows.

\begin{cor}\label{cor1}
Let $\Psi\in{\rm Symp}(M, \omega)$ and let $\mathcal{S}\subset (M, \omega)$
be a hypersurface of restricted contact type that admits a globally defined Liouville vector field $X$
satisfying  $X(\Psi(x))=d\Psi(x)[X(x)]$ for all $x\in M$.
If $\mathcal{S}\cap {\rm Fix}(\Psi)\ne\emptyset$ and
$c^\Psi_{\rm HZ}(U,\omega)<+\infty$ for some neighborhood $U$ of  $\mathcal{S}$
then there exists a $\Psi$-characteristic on  $\mathcal{S}$.
\end{cor}

Indeed, let  $\phi^t$ denote the local flow of the Liouville vector field $X$.
Since $\mathcal{S}$ is compact and $X$ is transversal to $\mathcal{S}$, there exists a sufficiently small $\varepsilon>0$ such that the map
$$
{\psi: \mathcal{S}\times (-\varepsilon, \varepsilon)\rightarrow U\subset M,\;(x,t)\mapsto \phi^t(x)}
$$
is a differomorphism (by shrinking $U$ if necessary), and that
 {$\phi^t(\Psi(x)) =\Psi(\phi^t(x))$ for all $(t, x)\in (-\varepsilon, \varepsilon)\times\mathcal{S}$}
 because of $X(\Psi (x))=d\Psi (x)[X(x)]$ for all $x\in M$. Define $H:U\to\mathbb{R}$
 by $H(x)=\lambda$ if $x=\psi(y,\lambda)\in U$.
Let $\mathcal{S}_\lambda=H^{-1}(\lambda)=\psi(\mathcal{S}\times \{\lambda\})$
for $\lambda\in(-\varepsilon,\varepsilon)$. By Theorem~\ref{dense}
there exists $\lambda\in  (-\varepsilon,\varepsilon)$ arbitrarily close to $0$ such that
$\mathcal{S}_\lambda$ carries a $\Psi$-characteristic $y$.
Note that $d\phi^\lambda: \mathcal{L}_\mathcal{S}\to \mathcal{L}_{\mathcal{S}_\lambda}$
is a bundle isomorphism. From these we derive that $x(t)=\phi^{-\nu}y(t)$ is a
$\Psi$-characteristic on $\mathcal{S}$. (See the arguments above Proposition~\ref{prop:EH.4.1}
for details.)

Clearly,  Corollary~\ref{cor1} may be applied to $(M, \omega)=(\mathbb{R}^{2n},\omega_0)$
and $\Psi\in{\rm Sp}(2n,\mathbb{R})$. But the result obtained is weaker than
Theorem~\ref{th:EHcontact}.

The same assumptions as under Theorem~\ref{dense} may yield the following result for the leaf-wise intersection question,
which can be viewed as a partial generalization of the result in \cite{MaSc05} and
will be proved in Section~\ref{sec:appl}.

\begin{thm}\label{MaSch}
Under the assumptions of Theorem~\ref{dense}
the following holds:
\begin{description}
\item[(i)] There exists a subset $\Delta\subset (0,\varepsilon)$ of full Lebesgue measure $m(\Delta)=\varepsilon$ such that for every $\delta\in \Delta$
    either $S_{\delta}\cup S_{-\delta}$ contains a fixed point of $\Psi$,
    or $S_{\delta}$ carries  a $\Psi$-characteristic $y:[0,T]\to S_{\delta}$  satisfying
  $\dot{y}=X_H(y)$ or $S_{-\delta}$ carries  a $\Psi$-characteristic $y:[0,T]\to S_{-\delta}$  satisfying
  $\dot{y}=-X_H(y)$. (If $\Psi$ has only finitely many fixed points in $U$, $\Delta$ may be chosen so that the first case does not occur.)
\item[(ii)] There exists a subset $\Lambda\subset I\setminus\{0\}$ of full Lebesgue measure $m(\Lambda)=m(I)$ such that for every
nonzero parameter $\lambda\in \Lambda$  the associated energy surface $S_{\lambda}$ carries
either a fixed point of $\Psi$ or a $\Psi$-characteristic $y:[0,T]\to S_{\lambda}$  satisfying
  $\dot{y}=X_H(y)$ or $\dot{y}=-X_H(y)$.
  (If $\Psi$ has only finite fixed points in $U$, $\Lambda$ may be chosen so that the first case does not occur.)
    Consequently, each $S_{\lambda}$ with $\lambda\in\Lambda$ carries a leaf-wise intersection point for $\Psi$.
  \end{description}
\end{thm}

\begin{remark}
{\rm Clearly, the statements (i) and (ii) in the above theorem cannot be contained in each other.
When $\Psi=id_M$ the $\Psi$-characteristics become closed characteristics and
 Hofer and Zehnder showed in \cite[p. 118, Theorem~4]{HoZe94} that
for some subset $\Lambda\subset I$ of full Lebesgue measure $m(\Lambda)=m(I)$
 every energy surface $S_{\lambda}$ with $\lambda\in\Lambda$   carries  a closed characteristic,
 provided $(M,\omega)$ has finite Hofer-Zehnder capacity and $\mathcal{S}\subset (M,\omega)$ bounds a compact  symplectic manifold.
   Macarini and  Schlenk in \cite{MaSc05} removed out the last additional assumption.
    (Actually, when $(M, \omega)=(\mathbb{R}^{2n}, \omega_0)$ and $\Psi=id_{\mathbb{R}^{2n}}$
 Struwe \cite{Str90} refined the arguments by Hofer and Zehnder \cite{HoZe87}
 to  prove such a result in 1990.)}
\end{remark}

\begin{cor}\label{cor2}
Let $0$ be a regular value of $H\in C^2(\mathbb{R}^{2n})$ such that
$\mathcal{S}:=H^{-1}(0)$ is compact and connected. Let
$U=\cup_{\lambda \in I}\mathcal{S}_\lambda$, where $I=(-\varepsilon,\varepsilon)$ and
$\mathcal{S}_\lambda=\{x\in \mathbb{R}^{2n}\,|\,H(x)=\lambda\}$, be
 a thickening of $\mathcal{S}$ in  $(\mathbb{R}^{2n},\omega_0)$.
Suppose that $\Psi\in{\rm Sp}(2n, \mathbb{R})$ and $\mathcal{S}\cap{\rm Fix}(\Psi)\ne\emptyset$.
Then the corresponding conclusions to those of Theorem~\ref{MaSch} hold.
\end{cor}

Since $\mathcal{S}$ is compact, there exists sufficiently large $R>0$ such that
$U\subset B^{2n}(R)$. By the monotonicity and positive conformality property of the capacity $c^{\Psi}_{\rm HZ}$ together with (\ref{e:ball}), we get that $c^{\Psi}_{\rm HZ}(U,\omega_0)\le c^\Psi_{\rm HZ}(B^{2n}(R),\omega_0)\le 2R^2 \mathfrak{t}(\Psi)$. Corollary~\ref{cor2} follows.

  As a generalization of Struwe's main result of \cite{Str90},
  we have the following result which is stronger than
Corollary~\ref{cor2}. It is proved in Section~\ref{sec:appl}.

\begin{thm}\label{Str}
Suppose that $1$ is a regular value of $H\in C^2(\mathbb{R}^{2n})$ and $\mathcal{S}:=H^{-1}(1)$
is compact and connected. {\rm (}Thus there exists $\delta_0>0$ such that each $\beta\in [1-\delta_0, 1+\delta_0]$
is a regular value of $H$ and $\mathcal{S}_\beta:=H^{-1}(\beta)$ is diffeomorphic to $\mathcal{S}=\mathcal{S}_1$.
Denote by $\gamma$ the diameter of $H^{-1}([1-\delta_0, 1+\delta_0])$.{\rm )}
Suppose for some $\Psi\in{\rm Sp}(2n, \mathbb{R})$ that
the bounded component of $\mathbb{R}^{2n}\setminus \mathcal{S}$ contains a fixed point
  of $\Psi$. Then for almost every $\beta\in (1-\delta_0, 1+\delta_0)$
the associated energy surface $\mathcal{S}_{\beta}:=H^{-1}(\beta)$ carries
a $\Psi$-characteristic $y$ satisfying $\dot{y}=X_H(y)$ and with action
$0<A(y)<16\mathfrak{t}(\Psi)\gamma^2$.
\end{thm}

\begin{remark}\label{HoZe87}
{\rm Recently, Ginzburg and G\"urel \cite{GinGu15} showed that there exists a closed,
 smooth hypersurface $\mathcal{S}\subset \mathbb{R}^{2n}$ ($2n\ge 4$) and
 a sequence of $C^\infty$-smooth autonomous Hamiltonians $F_k\to 0$ in $C^0$,
 supported in the same compact set, such that $\mathcal{S}$ and $\varphi_{F_k}(\mathcal{S})$
 have no leafwise intersections. Here $\varphi_{F_k}$ denotes the time-one map of the Hamiltonian flow of $F_k$.
 {This result suggests that Theorem~\ref{Str} is best possible
 in some sense since there may exist $\beta'$ near $1$ such that $H^{-1}(\beta')$ carries no $\Psi$-characteristics}.}
\end{remark}

\subsection{An extension of a theorem by Evgeni Neduv}\label{sec:1.Neduv}

{Evgeni Neduv  \cite[Theorem~4.4]{Ned01} showed that  differentiability of the Hofer-Zehnder capacity
can be used to derive  some results on fixed period problem of Hamiltonian systems. Similar differentiability also holds for the $\Psi$-HZ capacity where $\Psi\in\rm{Sp}(2n,\mathbb{R}^{2n})$ and can lead to a result on existence of solutions
$y:[0,T]\to\mathbb{R}^{2n}$ to the boundary value problem
$$
\dot{y}(t)=J\nabla\mathcal{H}(y(t)),\quad y(T)=\Psi y(0).
$$
with fixed $T$ and $\mathcal{H}$ a convex Hamiltonian satisfying certain asymptotic conditions. The main tool is the representation formula (Theorem~\ref{th:convex}).}

For a proper and strictly convex Hamiltonian $\mathscr{H}\in C^2(\mathbb{R}^{2n},\mathbb{R}_{\ge 0})$
such that $\mathscr{H}(0)=0$ and $\mathscr{H}''>0$ (which imply
$\mathscr{H}\ge 0$ by the Taylor's formula), if $e_0\ge 0$ is a regular value of $\mathscr{H}$ with
$\mathscr{H}^{-1}(e_0)\ne\emptyset$,
the set $D(e):=\{\mathscr{H}<e\}$
is a  strictly convex bounded domain in $\mathbb{R}^{2n}$ with $0\in D(e)$ and
with $C^2$-boundary $\mathcal{S}(e)=\mathscr{H}^{-1}(e)$ for each number $e$ {sufficiently close to} $e_0$.
Given $\Psi\in{\rm Sp}(2n,\mathbb{R})$, for any $e$ near $e_0$ let $\mathscr{C}(e):=c^\Psi_{\rm HZ}(D(e),\omega_0)$.
By Remark~\ref{rm:carrier}
all $c^\Psi_{\rm HZ}$-carriers for $D(e)$ form a compact subset in $C^1([0, \mathscr{C}(e)], \mathcal{S}(e))$. Hence
\begin{equation}\label{e:convexDiff}
\mathscr{I}(e):=\left\{T_x=2\int^{\mathscr{C}(e)}_0\frac{dt}{\langle\nabla\mathscr{H}(x(t)), x(t)\rangle}\,\Big|\,
 \hbox{$x:[0,\mathscr{C}(e)]\to\mathcal{S}(e)$ is a $c^\Psi_{\rm HZ}$-carrier for $D(e)$}\right\}
\end{equation}
is a compact subset in $\mathbb{R}$. Denote by $T^{\max}(e)$ and $T^{\min}(e)$
 the largest and smallest numbers in $\mathscr{I}(e)$.
{Every $c^\Psi_{\rm HZ}$-carrier $x$ for $D(e)$ can be reparameterized as a solution of
\begin{equation}\label{e:convexDiff+}
-J\dot{x}(t)=\nabla\mathscr{H}(x(t))\;\forall t\in [0, T_x]\quad\hbox{and}\quad x(T_{x})=\Psi x(0)
\end{equation}
on $\mathcal{S}(e)=\mathscr{H}^{-1}(e)$, where $T_x\in\mathscr{I}(e)$, satisfying
$x(t)\ne \Psi x(0)$ for any $t\in (0, T_{x})$. The final property means that
$T_{x}$ is the minimal period of $x$ if $\Psi=I_{2n}$.}

The following is our generalization for \cite[Theorem~4.4]{Ned01}. Its proof is given in Section~\ref{sec:convexDiff}.

\begin{thm}\label{th:convexDiff}
Let $\Psi\in{\rm Sp}(2n,\mathbb{R})$ and $\mathscr{H}\in C^2(\mathbb{R}^{2n},\mathbb{R}_{\ge 0})$ be as above.
Then $\mathscr{C}(e)$ has the left and right derivatives
at $e_0$, $\mathscr{C}'_-(e_0)$ and $\mathscr{C}'_+(e_0)$, and they satisfy
\begin{eqnarray*}
&&\mathscr{C}'_-(e_0)=\lim_{\epsilon\to0-}T^{\max}(e_0+\epsilon)=T^{\max}(e_0)\quad\hbox{and}\\
&&\mathscr{C}'_+(e_0)=\lim_{\epsilon\to0+}T^{\min}(e_0+\epsilon)=T^{\min}(e_0).
\end{eqnarray*}
Moreover, if  $[a,b]\subset (0,\sup\mathscr{H})$ is a regular interval of $\mathscr{H}$ such that
$\mathscr{C}'_+(a)<\mathscr{C}'_-(b)$, then for any  $r\in (\mathscr{C}'_+(a),
\mathscr{C}'_-(b))$ there exists $e'\in (a,b)$ such that $\mathscr{C}(e)$
is differentiable at $e'$ and $\mathscr{C}'_-(e')=\mathscr{C}'_+(e')=r=T^{\max}(e')=T^{\min}(e')$.
\end{thm}

As a monotone function on a regular interval $[a,b]$ of $\mathscr{H}$ as above,
$\mathscr{C}(e)$ satisfies $\mathscr{C}'_-(e)=\mathscr{C}'_+(e)$ for almost all values of $e\in [a, b]$
and thus both $T^{\max}$ and $T^{\min}$ are almost everywhere continuous.
Actually, both $T^{\max}$ and $T^{\min}$ have only at most countable discontinuous
points and are also Riemann integrable on $[a,b]$ (see \cite[Corollary~6.4]{Bl14}).

By Theorem~\ref{th:convexDiff}, for any regular interval $[a,b]\subset (0,\sup \mathscr{H})$
 of $\mathscr{H}$ with $\mathscr{C}'_+(a)\le\mathscr{C}'_-(b)$,
if $T\in [\mathscr{C}'_+(a), \mathscr{C}'_-(b)]$ then
(\ref{e:convexDiff+}) has a solution
$y:[0, T]\to \mathscr{H}^{-1}([a,b])$ such that $y(T)=\Psi y(0)$ and
 $y(t)\ne \Psi y(0)$ for any $t\in (0, T)$. For example, we have

 \begin{corollary}\label{cor:convexDiff}
Suppose that a proper and strictly convex Hamiltonian $\mathscr{H}\in C^2(\mathbb{R}^{2n},\mathbb{R}_{\ge 0})$
satisfies the conditions:
\begin{description}
\item[(i)] $\mathscr{H}(0)=0$, $\mathscr{H}''>0$ and every $e>0$ is a regular value of $\mathscr{H}$,
\item[(ii)] there exist two positive definite symmetric matrixes $S_0, S_\infty\in\mathbb{R}^{2n\times 2n}$ such that
\begin{equation}\label{e:convexDiff++}
\inf\{t>0\,|\, \det(\exp(tJS_0)-\Psi)\ne 0\}\le\inf\{t>0\,|\, \det(\exp(tJS_\infty)-\Psi)\ne 0\}
\end{equation}
    and that $\mathscr{H}(x)$ is equal to  $q(x):=\frac{1}{2}\langle S_0x, x\rangle$ (resp.
    $Q(x):=\frac{1}{2}\langle S_\infty x, x\rangle$) for $|x|$ small (resp. large) enough.
\end{description}
Then for every $T$ between the two numbers in (\ref{e:convexDiff++}) the corresponding system
(\ref{e:convexDiff+}) has a solution
$y:[0, T]\to \mathbb{R}^{2n}$ such that $y(T)=\Psi y(0)$ and
 $y(t)\ne \Psi y(0)$ for any $t\in (0, T)$.
\end{corollary}

In fact, if $e>0$ is small (resp. large) enough then
$D(e)$ is equal to $D_q(e):=\{q<e\}=\sqrt{e}E(q)$ (resp. $D_Q(e):=\{Q< e\}=\sqrt{e}E(Q)$)
and so
\begin{eqnarray*}
&&\hbox{$c^\Psi_{\rm HZ}(D(e),\omega_0)=c^\Psi_{\rm HZ}(D_q(e),\omega_0)=ec^\Psi_{\rm HZ}(E(q),\omega_0)$}\\
&&\hbox{(resp. $c^\Psi_{\rm HZ}(D(e),\omega_0)=c^\Psi_{\rm HZ}(D_Q(e),\omega_0)=ec^\Psi_{\rm HZ}(E(Q),\omega_0)$)}.
\end{eqnarray*}
Then  Corollary~\ref{cor:ellipsoid} implies
$\mathscr{C}'(a)=c^\Psi_{\rm HZ}(E(q),\omega_0)=\inf\{t>0\,|\, \det(\exp(tJS_0)-\Psi)\ne 0\}$ for $a>0$ small enough
and  $\mathscr{C}'(b)=c^\Psi_{\rm HZ}(E(Q),\omega_0)=\inf\{t>0\,|\, \det(\exp(tJS_\infty)-\Psi)\ne 0\}$ for $b>0$ large enough.
The conclusions of Corollary~\ref{cor:convexDiff}  follow from Theorem~\ref{th:convexDiff} immediately.

\noindent{\bf Organization of the paper}.
\begin{itemize}
  \item Section~\ref{section:space}: present our variational frame and related preparations.
  \item Section~\ref{sec:convex}: prove Theorem~\ref{th:convex}
by improving the arguments in \cite{HoZe87, HoZe90}.
  \item Section~\ref{sec:EH.2}: provide the variational explanation for our extended  Ekeland-Hofer capacity $c^\Psi_{\rm EH}$ (the proof of Theorem~\ref{th:EH.1.6}).
  \item Section~\ref{sec:EH.3}: prove Theorems~\ref{th:EHconvex},~\ref{th:EHproduct}.
  \item Section~\ref{sec:EH.4}: prove Theorem~\ref{th:EHcontact}.
  \item Section~\ref{sec:appl}: prove Theorems~\ref{MaSch}, \ref{Str}.
  \item Section~\ref{sec:convexDiff}: prove Theorem~\ref{th:convexDiff}.
\end{itemize}

\noindent{\bf Acknowledgements}: We would like to
thank Professor Jean-Claude Sikorav  for sending us his beautiful article \cite{Sik90}
and explaining some details. We  thank Dr. Jun Zhang for pointing out that Question $\Psi$
is an extension of the question about existence of leaf-wise intersection points and
for correcting some typing errors. We are also deeply grateful to the anonymous referees for giving very helpful
comments and suggestions to improve the exposition.

\section{Variational frame and related preparations}\label{section:space}
\setcounter{equation}{0}

In this section we shall give our variational frame
by suitably modifying those in \cite{HoZe87,EH89,HoZe90,HoZe94}.
For the sake of completeness some corresponding conclusions are also proved
in details though part of them appeared in \cite{Dong06} in different forms.

For a given symplectic  matrix  $\Psi\in{\rm Sp}(2n,\mathbb{R})$,
consider the Hilbert subspace of $H^1([0,1],\mathbb{R}^{2n})$,
$$
X=\{x\in H^1([0,1],\mathbb{R}^{2n})\,|\, x(1)=\Psi x(0)\}.
$$
 Since  $C^{\infty}_c([0,1])\subset X$ is dense in
$L^2([0,1],\mathbb{R}^{2n})$ (cf. \cite[Cor.4.23]{Br11}), so is $X$ in $L^{2}([0,1],\mathbb{R}^{2n})$.
Consider the unbounded linear  operator on $L^2([0,1],\mathbb{R}^{2n})$ with domain ${\rm dom} (\Lambda)=X$,
\begin{equation}\label{e:Lambda}
\Lambda:=-J\frac{d}{dt},
\end{equation}
which is also a bounded linear operator from $X$ ({with $H^1$ norm}) to $L^{2}([0,1],\mathbb{R}^{2n})$.
Denote by ~$E_1\subset \mathbb{R}^{2n}$ the eigenvector space which belongs to eigenvalue ~$1$ of ~$\Psi$.
By identifying ~$a\in E_1 $ with the constant path in $L^2([0,1],\mathbb{R}^{2n})$ given by
$\hat{a}(t)=a$ for all $t\in [0,1]$, we can identify
~${\rm Ker} (\Lambda)$ with ~$E_1$. We write ~${\rm Ker} (\Lambda)=E_1$ without occurring of confusions.
  Denote by $R(\Lambda)$ the range of $\Lambda$. The following proposition
  is a standard exercise in functional analysis. But we still give its detailed proof for
  the sake of completeness.

\begin{prop}\label{eigenvalues}
  \begin{description}
    \item[(i)] $R(\Lambda)$ is a closed subspace in $L^2([0,1],\mathbb{R}^{2n})$ and there exists the following orthogonal decomposition
        \begin{equation}\label{split2}
        L^2([0,1],\mathbb{R}^{2n})= {\rm Ker} (\Lambda)\oplus R(\Lambda).
        \end{equation}
    \item[(ii)]  The restriction ~$\Lambda_0:=\Lambda|_{R(\Lambda)\cap {\rm dom}(\Lambda)}$ is a
    bijection onto $R(\Lambda)$, and $\Lambda_{0}^{-1}:R(\Lambda)\rightarrow R(\Lambda)$ is a compact and
        self-adjoint operator if $R(\Lambda)$ is equipped with the $L^2$ norm.
  \end{description}
    \end{prop}

  \begin{proof}
{\bf Step 1} ({\it Prove that $R(\Lambda)$ is closed in ~$L^2([0,1])$}). Let $E_1^{\bot}$ be the orthogonal
 complement of $E_1$ with respect to the standard Euclidean inner product in $\mathbb{R}^{2n}$.
 Since $\dim E_1=2n$ if and only if $\Psi=I_{2n}$,
 the problem reduces to the periodic case  studied in past if $\dim E_1^{\bot}=0$.
 Hence we only  consider the non-periodic case in which  $\dim E_1^{\bot}\geq 1$.  Then
 $$
 \Psi-I_{2n}: E_1^{\bot}\rightarrow  (\Psi-I_{2n})(\mathbb{R}^{2n})
 $$
 is continuously invertible. Denote by ~$(\Psi-I)^{-1}$ its inverse and
 \begin{equation}\label{e:inverse}
 C:=\sup\{|(\Psi-I_{2n})^{-1}x\,|\, x\in (\Psi-I_{2n})(\mathbb{R}^{2n})\;\&\;
 |x|=1\},
 \end{equation}
 where $|\cdot|$ denotes the standard norm in $\mathbb{R}^{2n}$.

 Let  $(x_k)\subset R(\Lambda)$ be a sequence  converging to $x$ in $L^2([0,1],\mathbb{R}^{2n})$.
 For each $x_k$, we may choose its preimage to be
$$
u_k(t)=J\int_0^t x_k(s)ds +u_k(0),
$$
where $u_k(0)\in E_1^\bot$.
Then  $(u_k)$ is a Cauchy sequence in $X$ ({with $H^1$ norm}). In fact, since $u_k(1)=\Psi u_k(0)$, we get
$$
J\int_0^1 x_k(s)ds=(\Psi-I)u_k(0),
$$
 where $u_k(0)\in E_1^{\bot}$. Hence
$$
|u_k(0)-u_m(0)|\leq C\left|\int_0^1 (x_k-x_m)(s)ds \right|\leq C\|x_k-x_m\|_{L^2},
$$
and therefore
\begin{eqnarray*}
 \|u_k-u_m\|_{L^2}^2
 =\int_0^1 \bigg|J\int_0^t (x_k-x_m)(s)ds+u_k(0)-u_m(0) \bigg|^2dt
\leq  2 (C^2+1)\|x_k-x_m\|_{L^2}^2.
\end{eqnarray*}
It is obvious that $\|\dot{u}_k-\dot{u}_m\|_{L^2}=\|x_k-x_m\|_{L^2}$ and
thus
$$
\|u_k-u_m\|_{H^1}\leq \sqrt{2C^2+3}\|x_k-x_m\|_{L^2}\to 0
$$
as $k\to\infty$ and $m\to\infty$.
Let $u_k\rightarrow u$ in $X$. Then
$x_k=\Lambda u_k\rightarrow \Lambda u$ in $L^2([0,1],\mathbb{R}^{2n})$, and so $\Lambda u=x$, i.e. $x\in R(\Lambda)$.

\noindent{\bf Step 2} ({\it Prove that $L^2([0,1],\mathbb{R}^{2n})$ has the orthogonal decomposition as in
        (\ref{split2})}).
Note that $\mathbb{R}^{2n}$ has the following orthogonal splitting:
\begin{equation}\label{split1}
\mathbb{R}^{2n}=J {\rm Ker} (\Psi-I)\oplus R (\Psi-I),
\end{equation}
where $R (\Psi-I)=(\Psi-I)(E_1^\bot)=(\Psi-I)(\mathbb{R}^{2n})$.
In fact, for ~$a\in {\rm Ker} (\Psi-I) $ and ~$b=(\Psi-I)c\in R (\Psi-I)$, we have
 \begin{eqnarray*}
\langle Ja,b\rangle&=&\langle Ja,(\Psi-I)c \rangle=\langle J\Psi a,\Psi c\rangle-\langle Ja, c\rangle\\
&=&\langle \Psi^{t}J\Psi a, c\rangle-\langle Ja, c\rangle=\langle J a, c\rangle-\langle Ja, c\rangle=0.
\end{eqnarray*}
This and the dimension equality $\dim {\rm Ker} (\Psi-I)+\dim R(\Psi-I)=\dim \mathbb{R}^{2n}$ lead to (\ref{split1}).

For any given $x\in L^{2}([0,1])$, by (\ref{split1}) we can write
$$
J\int_0^1 x(s)ds =Ja+b,
$$
where $a\in {\rm Ker} (\Psi-I)$ and $b=(\Psi-I)c\in R(\Psi-I)$.
 Let
$$
u(t)=J\int_0^t(x(s)-a)ds+c\quad\forall t\in [0,1].
$$
Then $u\in X$ because
$$
u(1)=J\int_0^1(x(s)-a)ds+c=J\int_0^1 x(s)ds-Ja+c=\Psi c=\Psi u(0).
$$
It follows from this and the definition of $u$ that
$\Lambda u=x-a$.

Moreover, for $a\in {\rm Ker} (\Lambda)={\rm Ker} (\Psi-I)=E_1$ and $y=\Lambda w\in R(\Lambda)$, we compute
\begin{eqnarray*}
 \langle a, y\rangle_{L^2}= \int_0^1 \langle a ,-J\dot{w}\rangle
=\langle Ja,w(1)-w(0)\rangle =\langle J\Psi a, \Psi w(0)\rangle-\langle Ja,w(0)\rangle=0
\end{eqnarray*}
because $\Psi^tJ\Psi=J$.
Therefore the orthogonal decomposition in (\ref{split2}) follows immediately.

\noindent{\bf Step 3} ({\it Prove (ii)}).
Firstly, we prove that $R(\Lambda)\cap {\rm dom}(\Lambda)$ is a closed subspace
in $X$ (with $H^1$ norm).
Let  $(u_k)\subset R(\Lambda)\cap {\rm dom}(\Lambda)$ be a Cauchy sequence in $H^1$ norm. Then
it  converges to some $u\in X={\rm dom}( \Lambda)$ in the ~$H^1$ norm. Especially,
   $(u_k)$ converges to $u$ in the $L^2$ norm. Since $R(\Lambda)$ is closed in $L^2([0,1],\mathbb{R}^{2n})$
   we get that $u\in R(\Lambda)$. The claim is proved.

 Consider the operator $\Lambda_0:=\Lambda|_{R(\Lambda)\cap {\rm dom}(\Lambda)}$.
 Clearly, it is a bijective continuous linear map from a Hilbert subspace ${\rm dom} (\Lambda_0)$ of $X$
 to the Hilbert subspace $R(\Lambda)$ of $L^2([0,1],\mathbb{R}^{2n})$.
   Hence  the Banach inverse operator theorem yields a
   continuous linear operator $\Lambda_0^{-1}:R(\Lambda)\rightarrow {\rm dom} (\Lambda_0)$.
  Note that $i:{\rm dom} (\Lambda_0)\hookrightarrow\hookrightarrow R(\Lambda)$ (as a restriction of the compact
  inclusion map $H^{1}\hookrightarrow L^2$) is compact.
  Hence $i\circ \Lambda_0^ {-1 }:R(\Lambda)\rightarrow R(\Lambda) $ is compact.

We claim that $i\circ \Lambda_0^ {-1 }$ is also self-adjoint. In fact, for any $u, w\in X$ there holds
\begin{eqnarray*}
 \langle \Lambda u,w\rangle_{L^2}
&=&\int_0^1 \langle-J\dot{u},w\rangle dt\\
&=&\langle-Ju,w\rangle |_0^1-\int_0^1 \langle-Ju,\dot{w}\rangle dt\\
&=&-(\langle Ju(1),w(1)\rangle-\langle Ju(0),w(0)\rangle)-\int_0^1\langle J\dot{w},u\rangle dt
=\langle u,\Lambda w\rangle_{L^2}.
\end{eqnarray*}
Note that $\langle Ju(1),w(1)\rangle-\langle Ju(0),w(0)\rangle=0$ since $u, w\in X$ satisfy the boundary condition
$u(1)=\Psi u(0)$ and $w(1)=\Psi w(0)$.
For $x, y\in R(\Lambda )$, let us choose $u, w\in X\cap R(\Lambda)$ such that $\Lambda u=x$ and $\Lambda w=y$. Then
$\langle i\circ \Lambda_0^{-1}x,y\rangle
=\langle u, \Lambda w \rangle_{L^2}
=\langle \Lambda u,w \rangle_{L^2}
=\langle x, i\circ \Lambda_0^{-1}y\rangle_{L^2}$.
Hence we have proved that $i\circ \Lambda_0^{-1}:R(\Lambda)\rightarrow R(\Lambda)$
is a compact self-adjoint operator.
\end{proof}

\begin{remark}
{\rm Since $R(\Lambda)$ is a Hilbert subspace of $L^2([0,1],\mathbb{R}^{2n})$ which is separable,
by the standard linear functional analysis theory, there exists an orthogonal basis of
$R(\Lambda)$ which completely consists of eigenvectors of $i\circ \Lambda_0^{-1}$.
Note that ${\rm Ker}(i\circ \Lambda_0^{-1})=0$ and that $l\ne 0$ is
an eigenvalue of {$i\circ \Lambda_0^{-1}$} if and only if $1/l$ is
an eigenvalue of $\Lambda$ with the same multiplicity.
Let
\begin{equation}\label{eigenV}
\cdots\leq \lambda_{-k}\le\cdots\leq \lambda_{-1}<0<\lambda_1\leq\cdots\leq \lambda_k\leq \cdots
\end{equation}
 denote all eigenvalues of $\Lambda $, which satisfy $\lambda_k\to \pm\infty$ as $k\to\pm\infty$.
 By these and (\ref{split2}), we get a unit orthogonal  basis of $L^2([0,1],\mathbb{R}^{2n})$
 \begin{equation}\label{eigenV+}
 \{e_j\,|\,\pm j\in\mathbb{N}\}\cup\{e_0^i\}_{i=1}^q
 \end{equation}
such that ${\rm Ker}(\Lambda)={\rm Span}(\{e_0^i\}_{i=1}^q)$ and that each $e_j$
is an eigenvector corresponding to $\lambda_j$, $j=\pm 1,\pm 2,\cdots$.}
\end{remark}

\begin{remark}\label{rem:unitbase1}
{\rm The nonzero eigenvalues of $\Lambda=-J\frac{d}{dt}$ are exactly the zero points of the function  $g^{\Psi}$
   defined in (\ref{e:g}). In fact, let $\lambda$ be an eigenvalue of $\Lambda$ and
$e\in X$ be an eigenvector associated with it. Then
$-J\dot{e}(t)=\lambda e(t)\;\forall t\in [0, 1]$ and $e(1)=\Psi e(0)$.
 It follows that $e(t)=e^{\lambda t J}e(0)$ and $e(1)=e^{\lambda J}e(0)=\Psi e(0)$,
 where $e(0)\in\mathbb{R}^{2n}\setminus\{0\}$. Hence $\det(e^{\lambda J}-\Psi)=0$.
 By Lemma~\ref{zeros} function $g^\Psi$ has only finitely many zero points in $(0,2\pi]$,
 denoted by $t_1<\cdots<t_m$. Then
  all the eigenvalues of $\Lambda$ are
 \begin{equation}\label{eigen}
 \{t_l+2k\pi\,|\,1\leq l\leq m, k\in \mathbb{Z}\}.
 \end{equation}
Note that each eigenvector of $t_l+2k\pi$  has the form
 \begin{equation}\label{eigen+}
 e(t)=e^{(t_l+2k\pi)tJ}X,
 \end{equation}
where $X\in {\rm Ker}(e^{t_l J}-\Psi)$.
 By requiring $|X|=1$ we get that $|e(t)|\equiv 1$. Hence $e_j$ and $e_0^i$ in (\ref{eigenV+}) can be chosen to satisfy
 $$
 |e_j(t)|=|e_0^i(t)|\equiv 1\;\forall t.
 $$
 }
\end{remark}

\begin{remark}
{\rm If $\Psi\in {\rm Sp}(2n,\mathbb{R})\cap O(2n)$ is as in (\ref{US}),
by Lemma~\ref{zeros.1} the eigenvalues of $\Lambda$ associated to $\Psi$ are
\[
\{\theta_j+2k\pi\,|\,1\leq j\leq n, k\in \mathbb{Z}\},
\]
where $0< \theta_1\leq\cdots\leq \theta_n\leq 2\pi$ are as in (\ref{Udiagonal}),
and the corresponding eigensubspace to $\theta_j+2k\pi$ is generated by
$$
e^{(\theta_j+2k\pi)tJ}X_j\quad\hbox{and}\quad e^{(\theta_j+2k\pi)tJ}JX_j
$$
where $X_j$ is as in Lemma~\ref{zeros.1}.}
\end{remark}

Now we are in position to define the variational space needed in this article.
 Using the unit orthogonal  basis
given by (\ref{eigenV+}) every $x\in L^2([0,1],\mathbb{R}^{2n})$ can be uniquely written as
$$
x=x^0+\sum _{k<0}x_ke_k+\sum_{k>0} x_ke_k,
$$
where $x^0\in {\rm Ker}(\Lambda)$ and $\{x_k\,|\,\pm k\in\mathbb{N}\}\in\mathbb{R}$.
By Remark~\ref{rem:unitbase1} we can always assume
 $$
 |e_k(t)|\equiv 1\quad\hbox{for all}\quad t\in [0,1]\quad\hbox{and}\quad k=\pm 1,\pm 2,\cdots.
 $$
 For $s\geq 0$ define a linear subspace of $L^2([0,1],\mathbb{R}^{2n})$ by
\begin{equation}\label{espace}
E^s_\Psi=\left\{x\in L^2([0,1],\mathbb{R}^{2n})\,\Bigm|\,\sum _{ k\ne 0}|\lambda_k|^{2s}x_k^2<\infty \right\}.
\end{equation}
{We omit the subscript $\Psi$ in $E^s_\Psi$ if $\Psi$ is fixed and there is no confusion.}
It is easy to prove that
\begin{equation}\label{innerproduct}
\langle x,y\rangle_{E^s} =\langle x^0,y^0\rangle_{\mathbb{R}^{2n}}+\sum_{k\neq 0} |\lambda_k|^{2s}x_ky_k,\quad x, y\in E^s
\end{equation}
 defines a complete inner product on $E^s$.
Denote the associated norm by $\|\cdot\|_{E^s}$.
Note that $E^0=L^2$ and $\|\cdot\|_{E^0}=\|\cdot\|_{L^2}$.

Let $\mathbb{E}:=E^{\frac{1}{2}}$ be defined by (\ref{espace}) with $s=\frac{1}{2}$.
It has the orthogonal splitting
\begin{equation}\label{e:spaceDecomposition}
\mathbb{E}=\mathbb{E}^{-}\oplus \mathbb{E}^0 \oplus \mathbb{E}^{+},
\end{equation}
where $\mathbb{E}^-={\rm span}\{e_k, k<0\}$, $\mathbb{E}^0={\rm Ker }(\Lambda)$ and $\mathbb{E}^+={\rm span}\{e_k, k>0\}$.
Denote the associated projection on them by $P^-, P^0$ and $P^+$. For $x\in \mathbb{E}$, write
$x=x^-+x^0+x^+$, where $x^-\in \mathbb{E}^-$, $x^0\in \mathbb{E}^0$ and $x^+\in \mathbb{E}^+$.

Similar to {~$H^s(S^1,\mathbb{R}^{2n})$} defined in \cite{HoZe94}, we have

\begin{prop}\label{compact}
   Assume ~$t>s\geq 0$. Then the inclusion map $I_{t,s}:{E}^t\rightarrow {E}^s$ is compact.
\end{prop}
\begin{proof}
    Let $P_N: {E}^t \rightarrow {E}^s$ be the finite rank operator defined by
   $$
   P_N (x)=x^0+\sum_{0<|k|\leq N}x_ke_k
   $$
  for $x=x^0+\sum_{k\neq 0}x_ke_k$.
    It is a compact linear operator.
     Moreover,
     \begin{eqnarray*}
   \|(P_N-I_{t,s})x\|_{E^s}^2=\|\sum_{|k|>N}x_ke_k\|_{E^s}^2
   =\sum_{|k|>N}|\lambda_k|^{2s}x_k^2=\sum_{|k|>N}|\lambda_k|^{2(s-t)}|\lambda_k|^{2t}x_k^2
   \end{eqnarray*}
     and $|\lambda_k|^{2(s-t)}|\le\max(\lambda_N,|\lambda_{-N} | ) ^{2(s-t)}$ for each $|k|>N$, we deduce
     $$
     \sum_{|k|>N}|\lambda_k|^{2(s-t)}|\lambda_k|^{2t}x_k^2\le\max(\lambda_N,|\lambda_{-N} | ) ^{2(s-t)}\|x\|^2_{E^t}.
     $$
    Since $\lim_{k\rightarrow\pm\infty}\lambda_{k}=\pm\infty$,
   $\lim_{N\rightarrow+\infty} \|P_N-I_{t,s}\|^{op}=0$.
      Hence $I_{t,s}: E^t\rightarrow E^s$ is compact.
\end{proof}

\begin{prop}\label{e1}
  Assume $s>\frac{1}{2}$. If $x\in E^s$, then $x$ is continuous and satisfies $x(1)=\Psi x(0)$.
   Moreover, there exists a constant $c=c_s$ such that
  \begin{equation}\label{e:normprop2.4}
  \sup_{t\in [0,1]} |x(t)|\leq c\|x\|_{E^s}.
  \end{equation}
\end{prop}

\begin{proof}
   For $x=x^0+\sum_{k\neq 0}x_k e_k\in E^s$, since
   \begin{eqnarray}\label{e:prop2.4.2}
   |x^0| +\sum_{k\neq 0}|x_ke_k(t)|&=&|x^0| +\sum_{k\neq 0}|x_k| \nonumber \\
   &=&|x^0|+\sum_{k\neq 0}\frac{1}{|\lambda_k|^s}|\lambda_k|^s|x_k|\nonumber\\
   &\leq& |x^0|+(\sum_{k\neq 0}\frac{1}{|\lambda_k|^{2s}})^{\frac{1}{2}} (\sum_{k\neq 0}|\lambda_k|^{2s}|x_k|^2)^{\frac{1}{2}},
   \end{eqnarray}
  the series of functions $x^0+\sum_{k\neq 0}x_k e_k(t)$
  is absolutely uniformly convergent. In other words, $x(t)$ is the uniform limit of the
  function sequence $f_k(t):=x^0+ \sum_{0<|j|<k}  x_je_j(t)$. It follows that $x$ is continuous and
  $x(1)=\Psi x(0)$ since $f_k(1)=\Psi f_k(0)$ for all $k$. Moreover, (\ref{e:normprop2.4}) is a direct consequence of (\ref{e:prop2.4.2}).
  Note that we have   used the fact that    $\sum_{k\neq 0}\frac{1}{|\lambda_k|^{2s}}$ is finite
  for $s>\frac{1}{2}$ due to the form of the eigenvalues  of $\Lambda$ in (\ref{eigen}).
   \end{proof}

Let $\mathfrak{a}:\mathbb{E}\rightarrow\mathbb{R}$ be the functional given by
\begin{equation}\label{e:aaction}
\mathfrak{a}(x)=\frac{1}{2}(\|x^+\|^2_{\mathbb{E}}-\|x^-\|^2_{\mathbb{E}}).
\end{equation}
Then $\mathfrak{a}$ is smooth and has  gradient
$\nabla \mathfrak{a}(x)=x^+-x^-\in \mathbb{E}$.

\begin{remark}
{\rm  For $x\in C^{1}([0,1],\mathbb{R}^{2n})$ satisfying $x(1)=\Psi x(0)$,
  there holds
  \begin{equation}\label{a=A}
  \mathfrak{a}(x)=\frac{1}{2}\int_0^1\langle-J\dot{x},x\rangle dt=A(x).
  \end{equation}
  In fact, write
  $x=a^0+\sum_{k\neq 0}a_ke_k$ and $-J\dot{x}=b^0+\sum_{k\neq 0}b_ke_k$ in $L^2$.  For $k\neq 0$ we have
\begin{eqnarray}\label{e:-jdotk}
b_k&=&\int_0^1 \langle-J\dot{x}, e_k\rangle dt\nonumber\\
&=&\langle-Jx, e_k\rangle|_0^1-\int_0^1\langle-Jx, \dot{e}_k\rangle dt\nonumber\\
&=&-(\langle Jx(1),e_k(1)\rangle-\langle Jx(0),e_k(0)\rangle)-\int_0^1\langle x,J\dot{e}_k\rangle dt\nonumber\\
&=&-(\langle \Psi^{t}J\Psi x(0),e_k(0)\rangle-\langle Jx(0),e_k(0)\rangle)+\int_0^1\langle x,\lambda_ke_k\rangle dt\nonumber\\
&=&\lambda_ka_k.
\end{eqnarray}
Moreover, for $v\in {\rm Ker}(\Lambda)$ we have $\Psi v=v$ and thus
\begin{equation}\label{e:-jdot2}
\int_0^1\langle-J\dot{x}, v\rangle dt=-(\langle Jx(1),v\rangle-\langle Jx(0),v\rangle)=0.
\end{equation}
Hence $b^0=0$. It follows that
\begin{eqnarray*}
 \int_0^1\frac{1}{2}\langle-J\dot{x},x\rangle dt
=\frac{1}{2}\sum_{k\neq 0}\lambda_ka_k^2
=\frac{1}{2}\sum_{k>0}|\lambda_k|a_k^2-\sum_{k<0}|\lambda_k|a_k^2
=\frac{1}{2}(\|x^+\|^2_{\mathbb{E}}-\|x^-\|^2_{\mathbb{E}})
=\mathfrak{a}(x).
\end{eqnarray*}
{Note that if $x$ does not satisfy the boundary condition $x(1)=\Psi x(0)$,
the equality (\ref{a=A}) does not hold in general since $\|x\|_{\mathbb{E}}$ is a norm associated to $\Psi$.}}
  \end{remark}

{From now on we assume that $H:\mathbb{R}^{2n}\rightarrow\mathbb{R}$ is a smooth function satisfying the condition (H2) in
Section~\ref{sec:1.EH}.}
{ Then there exist positive numbers $C_1$ and $C_2$  such that
\begin{equation}\label{bound}
|\nabla H(z)|\leq 2a|z|+C_1\quad\hbox{and}\quad |H_{zz}|\leq C_2 \; \forall z\in\mathbb{R}^{2n}
\end{equation}
for $|z|$ sufficiently large.
Thus we have the well defined functional
$$
\hat{b}:L^2([0,1];\R^{2n})\rightarrow\mathbb{R},\; x\mapsto \int_0^{1}H(x(t))dt.
$$
It is also differentiable and has $L^2$-gradient
$\nabla\hat{b}(x)=\nabla H(x)$ for $x\in L^2([0,1];\R^{2n})$
(cf. \cite{HoZe94}).}

 Let $j: \mathbb{E}\rightarrow L^2$ be the inclusion map and
  $j^{\ast}:L^2\rightarrow \mathbb{E}$  the adjoint operator of it, i.e.
   $\langle j(x),y\rangle_{L^2}=\langle x, j^{\ast}(y)\rangle_{\mathbb{E}}$
 for all $x\in \mathbb{E}$ and $y\in L^2$.  Define a functional
  $$
  \mathfrak{b}:\mathbb{E}\rightarrow\mathbb{R},\;x\mapsto \hat{b}(j(x)).
  $$
It is not hard to prove that $\mathfrak{b}$ is differentiable and  has $\mathbb{E}$-gradient
$\nabla \mathfrak{b}(x)=j^{\ast}\nabla H(x)$ for $x\in \mathbb{E}$.

{Arguing as in \cite{HoZe94}, we have the following propositions.}
 \begin{prop}\label{jast}
   For $y  \in L^2$,  $j^{\ast}(y)\in E^1$ and $j^{\ast}$ is a compact operator.
 \end{prop}

\begin{prop}\label{Lip}
   The  gradient $\nabla \mathfrak{b}:\mathbb{E}\rightarrow \mathbb{E}$ is compact and satisfies the global Lipschitz condition
   $$
   \|\nabla \mathfrak{b}(x)-\nabla\mathfrak{b}(y)\|_{\mathbb{E}}\leq C_3\|x-y\|_{\mathbb{E}}\,\,\forall x, y\in \mathbb{E}$$
   for some constant $C_3>0$.
   Moreover, there exist positive numbers $C_4$ and $C_5$ such that $|\mathfrak{b} (x)|\leq C_4 \|x\|^2_{L^2}+C_5 $, $\forall x\in \mathbb{E}$.
 \end{prop}

 Proposition~\ref{Lip} shows that the functional $\Phi_H:=\mathfrak{a}-\mathfrak{b}$
 is differentiable and its gradient $\nabla \Phi_{H}$ satisfies a global Lipschitz condition.
 Hence the negative gradient flow of $\Phi_{H}$ defined by
\begin{eqnarray*}
\frac{d\phi^t(x)}{dt}=-\nabla\Phi_{H}(\phi^t(x))\quad\hbox{and}\quad
\phi^0(x)=x
\end{eqnarray*}
exists for all $t\in\mathbb{R}$ and $x\in \mathbb{E}$ and it admits the representation
  $$
  \phi^t(x) =e^tx^-+x^0+e^{-t}x^{+}+K(t,x),
  $$
  where $K:\mathbb{R}\times\mathbb{E}\rightarrow \mathbb{E}$ is continuous and maps bounded sets into precompact sets \textcolor{red}{(cf. \cite{HoZe94})}.

  Next we study the regularity of the critical points of $\Phi_{H}$.

 \begin{prop}\label{prop:solution}
  If $x\in \mathbb{E}$ is a critical point of $\Phi_{H}$ on $\mathbb{E}$, then
$x$ is smooth and satisfies
$$
\dot{x}=J\nabla H (x) \quad\hbox{and}\quad
x(1)=\Psi x(0).
$$
\end{prop}
\begin{proof}
   Let $x\in \mathbb{E}$ be a critical point of $\Phi_{H}$. Then
   \begin{equation}\label{cpt}
   x^+-x^-=j^{\ast}(\nabla H(x)).
   \end{equation}
   {Write in the space $L^2([0,1],\R^{2n})$
   \begin{eqnarray*}
   x=a^0+\sum_{k\neq 0}a_k e_k,\quad
   \nabla H(x)=b^0+\sum_{k\neq 0}b_ke_k\quad\hbox{and}\quad y=y_0+\sum_{k\neq 0}y_k e_k
   \end{eqnarray*}
  for any $y\in \mathbb{E}$. Using
    $\langle j^{\ast}(\nabla H(x)), y\rangle_{\mathbb{E}}=\langle \nabla H(x), j(y)\rangle_{L^2}$, a direct computation yields
   $j^{\ast}(\nabla H(x))=b^0+\sum_{k\neq 0}\frac{1}{|\lambda_k|}b_k$ (cf. \cite{JinLu}).
   It follows that (\ref{cpt}) becomes
   \begin{eqnarray}\label{e:prop2.9.1}
   0=b^0\quad\hbox{and}\quad \lambda_ka_k=b_k\;\forall k\ne 0,
   \end{eqnarray}
   and therefore $x\in E^1$, where $E^1$ is as in (\ref{espace}). By Proposition~\ref{e1} $x$ is continuous and satisfies
   $x(1)=\Psi x(0)$. Hence $\nabla H(x)$ is also continuous. Using (\ref{split1}), we may write}
   \[
   \int_0^1 J\nabla H(x) dt=Jd+(\Psi-I)c,
   \]
   where $d\in {\rm Ker}(\Psi-I)$ and $c\in\mathbb{R}^{2n}$.
   Define ~$\xi(t)=\int_0^t(J\nabla H(x(s))-Jd)ds+c$. Then $\xi\in C^1([0,1],\mathbb{R}^{2n})$ and $\xi(1)=\Psi c=\Psi \xi(0)$.
   Writing $\xi=\xi^0+\sum_{k\neq 0}\xi_ke_k$ and computing as
   in (\ref{e:-jdotk}) and (\ref{e:-jdot2}), we get that
   \begin{equation}\label{1}
   -J\dot{\xi}=\sum_{k\neq 0}\lambda_k\xi_ke_k.
   \end{equation}
   On the other hand,
   \begin{equation}\label{2}
   -J\dot{\xi}=\nabla H(x)-d=-d+\sum_{k\neq 0}b_ke_k.
   \end{equation}
   Note here that $b^0=0$ and $d\in {\rm Ker}(\Psi-I)={\rm Ker}\Lambda$. Comparing (\ref{1}) and (\ref{2}) we get that
   \begin{eqnarray}\label{e:prop2.9.2}
   d=0\quad\hbox{and}\quad \lambda_k\xi_k=b_k \;\forall k\ne 0.
   \end{eqnarray}
   Since $\xi$ and $x$ are continuous, the second equalities in (\ref{e:prop2.9.1}) and (\ref{e:prop2.9.2}) lead to
   $\xi(t)-x(t)=const$, i.e.,
   $\xi(t)-x(t)=\xi(0)-x(0)=c-x(0)$. Hence
   \[
   x(t)=\int_0^t(J\nabla H(x(s))-Jd)ds+c-c+x(0)=\int_0^t J\nabla H(x(s))ds+x(0).
   \]
   Therefore $x\in C^1[0,1]$ and satisfies
   $\dot{x}=J\nabla H(x)$, which implies that $x$ is smooth.
\end{proof}

\begin{proposition}\label{prop:PSmale}
If $H\in C^\infty(\mathbb{R}^{2n},\mathbb{R})$ is $\Psi$-{\bf  nonresonant},
i.e. for $|z|$ sufficiently large, $H(z)=a|z|^2+ \langle z, z_0\rangle+ b$ with $z_0\in{\rm Fix}(\Psi)$ and $a$,
$b\in\mathbb{R}$ such that $\det(e^{2aJ}-\Psi)\neq 0$,
 then each sequence $(x_k)\subset \mathbb{E}$ such that $\nabla\Phi_{H}(x_k)\to 0$
 has a convergent subsequence. In particular,  $\Phi_H$ satisfies the (PS) condition.
\end{proposition}

\begin{proof}
   Since $\nabla\Phi_{H}(x)=x^+-x^--\nabla \mathfrak{b}(x)$ for any $x\in\mathbb{E}$, we have
   \begin{equation}\label{e:critic}
   x_k^+-x_k^--\nabla \mathfrak{b}(x_k)\rightarrow 0.
   \end{equation}
   \noindent{\bf Case 1}. {\it $(x_k)$ is bounded in $\mathbb{E}$}. Then $(x^0_k)$ is a bounded sequence
   in ${\rm Ker}(\Lambda)$ which has finite dimension.  Hence $(x^0_k)$ has a convergent subsequence.
   Moreover, since $\nabla \mathfrak{b}$ is compact,    $(\nabla \mathfrak{b}(x_k))$ has a convergent
   subsequence and so both $(x_k^+)$ and $(x_k^-)$ have convergent subsequences in $\mathbb{E}$.
   Hence $(x_k)$ has a convergent subsequence.

   \noindent{\bf Case 2}. {\it $(x_k)$ is unbounded in $\mathbb{E}$}.  Without loss of generality, we may assume
   $$
   \lim_{k\rightarrow +\infty }\|x_k\|_{\mathbb{E}}=+\infty.
   $$
   Let $y_k=\frac{x_k}{\|x_k\|_{\mathbb{E}}}-\frac{1}{2a}z_0$ where $z_0\in{\rm Fix }(\Psi)$. Then $|y_k^0|\le\|y_k\|_E\le 1+|\frac{z_0}{2a}|$ and (\ref{e:critic}) implies
   \begin{equation}\label{e:critic+}
   y_k^+-y_k^--j^{\ast}\left(\frac{ \nabla H(x_k)}{\|x_k\|_{\mathbb{E}}}\right)\rightarrow 0.
   \end{equation}
   Note that by (\ref{bound}) we have
   $$
   \left\|\frac{ \nabla H(x_k)}{\|x_k\|_{\mathbb{E}}}\right\|_{L^2}^2\leq \frac{8a^2 \|x_k\|_{L^2}^2+2C_1^2 }{\|x_k\|_{\mathbb{E}}^2}\leq C_6
   $$
   for some constant $C_6>0$, that is, $(\nabla H(x_k)/\|x_k\|_{\mathbb{E}})$ is bounded in $L^2$. Hence the sequence
   $j^{\ast}\left(\frac{ \nabla H(x_k)}{\|x_k\|_{\mathbb{E}}}\right)$ is compact.
   (\ref{e:critic+}) shows that $(y_k)$ has a convergent subsequence in $\mathbb{E}$. Without loss
   of generality, we may assume that  $y_{k}\rightarrow  y$ in $\mathbb{E}$. Since
    (H2) implies
    $$
    H(z)=Q(z):=a|z|^2+ \langle z, z_0\rangle+ b
    $$
    for $|z|$ sufficiently large, there exists a constant $C_7>0$ such that
   $$
   |\nabla H(z)-\nabla Q(z)|\leq C_7,\quad \forall z\in\mathbb{R}^{2n}.
   $$
   It follows that
   \begin{eqnarray*}
    \left\|\frac{\nabla H(x_k)}{\|x_k\|_{\mathbb{E}}}-\nabla Q(y)\right\|_{L^2}
   &\leq&\left\|\frac{\nabla H(x_k)}{\|x_k\|_{\mathbb{E}}}-\nabla Q(y_k)\right\|_{L^2}+\left\|\nabla Q(y_k)-\nabla Q(y)\right\|_{L^2}\\
   &\leq& \left\|\frac{\nabla H(x_k)-\nabla Q(x_k)}{\|x_k\|_{\mathbb{E}}}\right\|_{L^2}+\frac{|z_0|}{\|x_k\|_{\mathbb{E}}}+2a\|y_k-y\|_{L^2}\\
   &\leq &\frac{C_7}{\|x_k\|_{\mathbb{E}}}+\frac{|z_0|}{\|x_k\|_{\mathbb{E}}}+2a\|y_k-y\|_{L^2}
   \rightarrow 0
   \end{eqnarray*}
  as $k\to\infty$. This implies that $j^{\ast}\left(\frac{\nabla H(x_k)}{\|x_k\|_{\mathbb{E}}}\right)$ tends to $j^{\ast}(\nabla Q(y))$ in $\mathbb{E}$, and thus we arrive at
   \begin{eqnarray*}
   y^+-y^--j^{\ast}(\nabla Q (y))=0\quad\hbox{and}\quad \left\|y+\frac{z_0}{2a} \right\|_{\mathbb{E}}=1.
   \end{eqnarray*}
   Arguments similar to Proposition \ref{prop:solution}  show that $y$ is smooth and satisfies
    \begin{eqnarray}\label{e:y-boundary}
   \dot{y}=J\nabla Q(y)\quad\hbox{and}\quad y(1)=\Psi y(0).
   \end{eqnarray}
   Then we have
   $$
   y(t)+\frac{1}{2a}z_0=e^{2aJt}(y(0)+\frac{1}{2a}z_0).
   $$
   Noting that $z_0\in {\rm Fix}(\Psi)$, by the second condition in (\ref{e:y-boundary})  we deduce that
   $$
   y(1)+\frac{1}{2a}z_0=e^{2aJ}(y(0)+\frac{1}{2a}z_0)=\Psi(y(0)+\frac{1}{2a}z_0).
   $$
   Since $H$ is $\Psi$-nonresonant, i.e. $\det(e^{2aJ}-\Psi)\neq 0$, we get that $y(0)+\frac{1}{2a}z_0=0$, which implies that $y(t)+\frac{1}{2a}z_0\equiv 0$.
   However, we have already got that $\|y+\frac{1}{2a}z_0\|_{\mathbb{E}}=1$. This  contradiction shows that the second case does not occur.
\end{proof}

\section{Proof of Theorem~\ref{th:convex}}\label{sec:convex}
\setcounter{equation}{0}

By the assumptions in Theorem~\ref{th:convex} $D$ contains a fixed point $p$ of $\Psi$.
{The symplectomorphism
\begin{equation}\label{e:5convex}
\phi:(\mathbb{R}^{2n},\omega_0)\to (\mathbb{R}^{2n},\omega_0),\;x\mapsto x-p
\end{equation}
satisfies $\phi\circ\Psi=\Psi\circ\phi$. Hence
$c^\Psi_{\rm HZ}(D,\omega_0)=c^\Psi_{\rm HZ}(\phi(D),\omega_0)$. Moreover,
for a (generalized) $\Psi$-characteristic $z:[0,T]\to \partial D$, it is easily checked
that  $y=\phi\circ z$ is  a (generalized) $\Psi$-characteristic on $\partial(\phi(D))=\phi(\partial D)$
and satisfies $y(T)=\Psi(y(0))$ and $A(y)=A(z)$.
Hence from now on we may assume $p=0$, i.e. $0\in{\rm int}(D)$ in this section.}

\subsection{Proof of (\ref{e:action2})}\label{sec:convex1}

The goal of this subsection is to establish
the existence of a generalized $\Psi$-characteristic with minimal action
 via the Clarke dual variational  principle in \cite{Cl79} (see also \cite{MoZe05, HoZe94} in smooth case
 and \cite{Ek90, AAO14} in nonsmooth case for detailed arguments).
 The steps are classic besides that we need to take the boundary condition into consideration.
For the sake of clearness, we present the detailed proof.

Let  $j_D: \mathbb{R}^{2n}\rightarrow\mathbb{R}$ be the Minkowski (or gauge) functional associated to $D$.
Then the Hamiltonian function  $H:\mathbb{R}^{2n}\to \mathbb{R}$ defined by $H(z)=(j_D(z))^2$
is convex (and so continuous by \cite[Cor.10.1.1]{Roc70} or \cite[Prop.2.31]{Kr15}).
 There exists some constant $R_1\geq 1$ such that
\begin{equation}\label{e:dualestimate0}
 \frac{|z|^2}{R_1}\leq H(z)\leq R_1|z|^2\quad
\forall z\in\mathbb{R}^{2n}.
\end{equation}
This implies that the Legendre transformation of $H$ defined by
$$
H^{\ast}(w)=\max_{\xi\in \mathbb{R}^{2n}}(\langle w ,\xi \rangle-H(\xi)),
$$
where  $\langle\cdot,\cdot\rangle$ is the standard Euclidean inner product,
is a convex function from $\mathbb{R}^{2n}$ to $\mathbb{R}$ (and thus continuous).
Moreover, there exists a constant $R_2\geq 1$ such that
\begin{equation}\label{e:dualestimate}
 \frac{|z|^2}{R_2}\leq H^{\ast}(z)\leq R_2|z|^2\quad
\forall z\in\mathbb{R}^{2n}.
\end{equation}
Note that $H^\ast$ is also $C^{1,1}$ in $\mathbb{R}^{2n}$
with uniformly Lipschitz constant if   $\mathcal{S}$ is
$C^{1,1}$ and strictly convex. (See \cite[Cor.10.1.1]{Roc70}.)

Recall that in Section~\ref{section:space} we
denote by  $E_1\subset \mathbb{R}^{2n}$ the eigenvector space which belongs to the eigenvalue  $1$ of $\Psi$ and $E_1^{\bot}$ the orthogonal
 complement of  $E_1$ with respect to the standard Euclidean inner product in $\mathbb{R}^{2n}$.
 When $\dim E_1^{\bot}=0$,  the problem reduces to the periodic case. Hence we only  consider the non-periodic case in which  $\dim E_1^{\bot}\geq 1$.
 Let $C$ be given by (\ref{e:inverse}), i.e., the norm of
 $((\Psi-I_{2n})|_{E_1^{\bot}})^{-1}$.

Consider the following subspace of $H^1([0,1],\mathbb{R}^{2n})$
\begin{equation}\label{e:constrant1}
\mathcal{F}=\{x\in H^1([0,1],\mathbb{R}^{2n})\,|\,x(1)=\Psi x(0)\; \& \; x(0)\in E_1^{\bot}\}
\end{equation}
and its subset
\begin{equation}\label{e:constrant2}
\mathcal{A} =\{x\in \mathcal{F}\,|\,A(x)=1 \},\quad\hbox{where}\quad A(x)=\frac{1}{2}\int_0^1\langle-J\dot{x},x\rangle dt.
\end{equation}
Then $\mathcal{A}$ is a regular submanifold of $\mathcal{F}$. In fact, for any $x\in \mathcal{F}$ and $\zeta\in T_x\mathcal{F}=\mathcal{F}$,
$$
dA(x)[\zeta]= \int_0^{1}\langle -J\dot{\zeta},x\rangle dt+\frac{1}{2}\langle -Jx,\zeta\rangle|_0^1=\int_0^{1}\langle -J\dot{\zeta},x\rangle dt
$$
since
\begin{eqnarray*}
\langle Jx(1),\zeta(1)\rangle-\langle Jx(0),\zeta(0)\rangle
&=&\langle J\Psi x(0),\Psi\zeta(0)\rangle-\langle Jx(0),\zeta(0)\rangle\\
&=&\langle \Psi^{t}J\Psi x(0), \zeta(0)\rangle-\langle Jx(0),\zeta(0)\rangle\\
&=&\langle Jx(0), \zeta(0)\rangle-\langle Jx(0),\zeta(0)\rangle
=0.
\end{eqnarray*}
Thus $dA\neq 0$ on $\mathcal{A}$ because
$$
dA(x)[x]=\int_0^{1}\langle -J\dot{x}, x\rangle dt=2,\quad\forall x\in \mathcal{A}=A^{-1}(1).
$$

\noindent{\bf Step 1}.\quad{\it The functional $I:\mathcal{F}\to\mathbb{R}$ defined by
$$
I(x)=\int_0^1H^{\ast}(-J\dot{x})dt
$$
has a positive infimum  on $\mathcal{A}$ denoted by $\mu:=\inf _{x\in\mathcal{A}}I(x)$.} In fact for $x\in \mathcal{F}$
\[
\int_0^1 \dot{x} dt=x(1)-x(0)=(\Psi-I)x(0),\quad\hbox{where}\quad x(0)\in E_1^\bot.
\]
Hence $|x(0)|\leq  C \|\dot{x}\|_{L^2}$
where $C$ is given by (\ref{e:inverse}). Then it is easily estimated that
\begin{eqnarray}\label{poincare}
\| x\|_{L^2}^2
&=&\int_0^1 \left|\int_0^{t}\dot{x}(s)ds+x(0)\right|^2dt\nonumber\\
&\leq &2\int_0^1\left(\left|\int_0^{t}\dot{x}(s)ds\right|^2+|x(0)|^2\right)dt\nonumber\\
&= &2(\|\dot{x}\|_{L^2}^2+|x(0)|^2)\nonumber\\
&\leq & 2(1+C^2) \|\dot{x}\|_{L^2}^2.
\end{eqnarray}
Moreover,  if $x\in\mathcal{A}$ then there holds
\begin{equation}\label{actioncons}
2=2A(x)\leq \|x\|_{L^2}\|\dot{x}\|_{L^2}\leq \sqrt{2(1+C^2)}\|\dot{x}\|_{L^2}^2.
\end{equation}
It follows from these and (\ref{e:dualestimate}) that for any $x\in\mathcal{A}$
\begin{equation}\label{relation}
I(x)=\int_0^1 H^{\ast}(-J\dot{x})dt \geq \frac{1}{R_2}\|\dot{x}\|_{L^2}^2\geq C_1:=\frac{2}{R_2\sqrt{2(1+C^2)}}.
\end{equation}

\noindent{\bf Step 2}.\quad{\it There exists $u\in  \mathcal{A}$ such  that
$I(u)=\mu$. }
 Let $(x_n)\subset\mathcal{A}$ be a sequence satisfying $\lim_{n\rightarrow+\infty}I(x_n)=\mu$.
  By (\ref{relation}) for some $C_2>0$ we have
 $$
 R_2C_1\le\|\dot{x}_n\|_{L^2}^2\leq R_2I(x_n)\leq C_2,\quad n=1,2,\cdots.
 $$
 It follows from  (\ref{poincare}) and (\ref{actioncons}) that for all $n\in\N$,
$$
\frac{2}{C_2}\leq \frac{2}{\|\dot{x}_n\|_{L^2}^2}\leq\|x_n\|_{L^2}^2\leq 2(1+C^2)\|\dot{x}_n\|_{L^2}^2\leq  2(1+C^2)C_2.
$$
Hence $(x_n)$ is a bounded sequence in $H^1([0,1],\mathbb{R}^{2n})$. After passing to a subsequence
if necessary, we may assume that  $(x_n)$ converges weakly to some $u$ in  $H^1([0,1], \mathbb{R}^{2n})$.
 By the Arzel\'{a}-Ascoli theorem, there also exists $\hat{u}\in C^{0}([0,1],\mathbb{R}^{2n})$ such that
$$
\lim_{n\rightarrow+\infty}\sup_{t\in [0,1]}|x_n(t)-\hat{u}(t)|=0.
$$
Then a standard argument gives that $\hat{u}(t)=u(t)$ almost everywhere.
Since $x_n\rightarrow u$
in $C^{0}([0,1],\mathbb{R}^{2n})$, we get that $u(1)=\Psi u(0)$ and $u(0)\in E_1^{\bot}$. Moreover,
$u\in\mathcal{A}$ because
\begin{eqnarray*}
A(u)=\frac{1}{2}\int_0^1 \langle Ju, \dot{u}\rangle dt
&=&\lim_{n\rightarrow+\infty}\frac{1}{2}\int_0^1 \langle Ju, \dot{x}_{n}\rangle dt\\
&=&\lim_{n\rightarrow+\infty} \frac{1}{2}\int_0^1 (\langle Jx_n, \dot{x}_{n}\rangle+\langle J(u-x_n), \dot{x}_{n}\rangle) dt
=1.
\end{eqnarray*}

Consider the functional
$$
\hat{I}: L^2([0,1],\mathbb{R}^{2n})\to\mathbb{R},\;u\mapsto\int^1_0H^\ast(u(t))dt.
$$
Then $I(x)=\hat{I}(-J\dot{x})$ for any $x\in\mathcal{F}$.
Since $H^{\ast}$ is  convex, so is $\hat{I}$.
(\ref{e:dualestimate}) also implies that $\hat{I}$ is continuous and thus
has nonempty subdifferential $\partial\hat{I}(v)$ at each point $v\in L^2([0,1],\mathbb{R}^{2n})$. Moreover,
by Corollary~3 in \cite[Chap. II,\S3]{Ek90} we know
$$
\partial\hat{I}(v)=\{w\in L^2([0,1],\mathbb{R}^{2n})\,|\, w(t)\in \partial H^\ast(v(t))\;\hbox{a.e. on}\;[0,1]\}.
$$
By definition of subdifferential it follows that
\begin{eqnarray}\label{mini}
I(u)-I(x_n)&=&\hat{I}(-J\dot{u})-\hat{I}(-J\dot{x}_n)\nonumber\\
&\leq& \int_0^1 \langle w(t),-J( \dot{u}(t)-\dot{x}_n(t))\rangle dt
\end{eqnarray}
for any $w\in \partial\hat{I}(-J\dot{u})=\{w\in L^2([0,1],\mathbb{R}^{2n})\,|\, w(t)\in \partial H^\ast(-J\dot{u}(t))\;\hbox{a.e. on} \;[0,1]\}$.
Since that $(x_n)$ converges weakly to some $u$ in  $H^{1}([0,1], \mathbb{R}^{2n})$ implies that
$(\dot{x}_n)$ converges weakly to some $\dot{u}$ in  $L^2([0,1], \mathbb{R}^{2n})$,
we deduce that the left hand of (\ref{mini}) converges to $0$. Therefore
 $$
\mu\le I(u)\le\lim_{n\rightarrow+\infty}I(x_n)=\mu.
$$
The desired claim is proved.

\noindent{\bf Step 3}.\quad{\it There exists a generalized $\Psi$-characteristic on $\mathcal{S}$,
$x^\ast: [0,\mu]\rightarrow \mathcal{S}$,  such that $A(x^\ast)=\mu$.}
  Since $u$ is the minimum point of  $I|_{\mathcal{A}}$,
  applying the Lagrange multiplier theorem (cf. \cite[Theorem~6.1.1]{Cl83}) we get some $\lambda\in\mathbb{R}$ such that
$0\in\partial (I+\lambda A)(u)=\partial I(u)+\lambda A'(u)$.
Let us write $\Lambda_{\mathcal{F}}$ as the operator $\Lambda$ in (\ref{e:Lambda})
viewed as an operator from the Hilbert space $\mathcal{F}$ equipped with the $H^1$ norm to the Hilbert space $L^2$.
It is a closed linear operator. Then
$I=\hat{I}\circ \Lambda_{\mathcal{F}}$ and by Corollary~6 in \cite[Chap.II,\S2]{Ek90}
we arrive at
\begin{eqnarray*}
\partial I(u)&=& (\Lambda_{\mathcal{F}})^\ast\partial \hat{I}(\Lambda_\mathcal{F}(u))\\
&=&\{(\Lambda_{\mathcal{F}})^\ast w\,|\,
w\in L^2([0,1],\mathbb{R}^{2n})\;\&\; w(t)\in \partial H^\ast(-J\dot{u}(t))\;\hbox{a.e. on}\;[0,1]\}.
\end{eqnarray*}
Hence there exists a function $w\in L^2([0,1],\mathbb{R}^{2n})$ with $w(t)\in \partial H^\ast(-J\dot{u}(t))$ a.e.
on $[0,1]$ such that $(\Lambda_{\mathcal{F}})^\ast w+\lambda A'(u)=0$, i.e.,
\begin{eqnarray*}
0=\int_0^1 \langle w(t),-J\dot{\zeta}(t)\rangle dt+\lambda
\int_0^1 \langle  u(t),-J\dot{\zeta}(t)\rangle dt\quad \forall\zeta\in \mathcal{F}.
\end{eqnarray*}
This implies
\begin{equation}\label{mini.1}
w(t)+\lambda u(t)=a_0\quad\hbox{a.e. on}\quad [0,1]
\end{equation}
for some $a_0\in{\rm Ker} (\Psi-I)$. Then
\begin{eqnarray}\label{mini.2}
\langle w, -J\dot{u}\rangle&=&\int_0^1 \langle w(t),-J\dot{u}(t)\rangle dt\nonumber\\
&=&\int_0^1\langle a_0-\lambda u(t),-J\dot{u}(t)\rangle dt=-2\lambda.
\end{eqnarray}
Since the convex functional $\hat{I}$ is $2$-positively homogeneous, we use
the Euler formula \cite[Theorem~3.1]{YangWei08} to obtain
$$
\langle w, -J\dot{u}\rangle=2\hat{I}(-J\dot{u})=2I(u)=2\mu
$$
and thus $-\lambda=\mu$ by (\ref{mini.2}).
By (\ref{mini.1}), $a_0+\mu u(t)=w(t)\in\partial H^\ast(-J\dot{u}(t))$
so that $-J\dot{u}\in \partial H(a_0+\mu u(t))$ a.e. on [0,1].
Define $v:[0,\mu]\to\mathbb{R}^{2n}$ by \begin{equation}\label{mini.4}
v(t):=\mu u(t/\mu)+a_0.
\end{equation}
Then $v$ satisfies
 \begin{equation}\label{mini.5}
-J\dot{v}(t)\in  \partial H(v(t))\quad\hbox{a.e. on } \quad [0,\mu]
\end{equation}
and
$$
v(\mu)=\mu u(1)+a_0=\Psi(\mu u(0)+a_0)=\Psi v(0).
$$
The Legendre reciprocity formula (cf. \cite[Proposition~II.1.15]{Ek90}) in convex analysis yields
\begin{eqnarray}\label{e:vconstant}
\int_0^\mu H(v(t))dt&=&-\int_0^\mu H^\ast(-J\dot{v}(t))dt+\int_0^\mu\langle v(t),-J\dot{v}(t)dt\nonumber\\
&=&-\mu\int_0^1H^\ast(-J\dot{u}(s))ds+\mu^2\int_0^1\langle-J\dot{u}(s),u(s)
\rangle ds\nonumber\\
&=&-\mu^2+2\mu^2=\mu^2.
\end{eqnarray}
By \cite[Theorem~2]{Ku96} (\ref{mini.5})
implies $H(v(t))$ is constant and hence by (\ref{e:vconstant}) $H(v(t))\equiv\mu$ for all $t\in [0, \mu]$ so that  $v$ is nonconstant.
It follows that
 \begin{equation}\label{e:desiredChar}
x^\ast:[0,\mu]\rightarrow \mathcal{S},\; t\mapsto\frac{v(t)}{\sqrt{\mu}}=\sqrt{\mu}u(t/\mu)+ a_0/\sqrt{\mu}
\end{equation}
satisfies
$$
-J\dot{x}^\ast(t)\in  \partial H(x^\ast(t))\quad \hbox{a.e. on}\quad [0,\mu]
$$
and
$$
x^\ast(\mu)=\Psi x^\ast(0),\quad H(x^\ast(t))\equiv1,\quad\quad A(x^\ast)=\mu.
$$
That is to say $x^\ast$ is a $\Psi$-characteristic on $\mathcal{S}$ with action $A(x^\ast)=\mu$.

\noindent{\bf Step 4}. \quad{\it For any generalized $\Psi$-characteristic $y$
on $\mathcal{S}$ with positive action, there holds $A(y)\ge \mu$.
}
{By Lemma~2 in \cite[Chap.V,\S1]{Ek90} (or its proof), after reparameterization we may
assume that $y:[0,T]\mapsto \mathcal{S}$ is an absolutely continuous map satisfying
 \begin{equation}\label{e:repara}
     -J\dot{y}(t)=\partial {H}(y(t)), \;{\rm a.e.,}\quad
     y(T)=\Psi y(0).
\end{equation}
(See \cite[Lemma~4.2]{JinLu} for details). Then
\begin{equation}\label{e:desiredChar1}
A(y)=T\quad\hbox{and}\quad H(y(t))\equiv 1.
\end{equation}
 Since $\{w\in\partial H(x)\,|\,x\in\mathcal{S}\}$ is a bounded set in $\mathbb{R}^{2n}$ (by the proof of \cite[Lemma~4.2]{JinLu}),
  $y$ is in $W^{1,\infty}([0,T],\mathbb{R}^{2n})$ and in particular $y\in W^{1,2}([0,T],\mathbb{R}^{2n})$.}
Choose $a\in \mathbb{R}$ and $b\in E_1$ so that
$$
y^{\ast}:[0,1]\rightarrow \mathbb{R}^{2n},\;  t\mapsto y^{\ast}(t)=a y(tT)+b
$$
belongs to $\mathcal{A}$. Then $1=A(y^\ast)=a^2A(y)=a^2T$.
Since
$$
-J\dot{y}^\ast(t)=-aTJ\dot{y}(tT)\in aT\partial H({y}(Tt))=
\partial H(aT{y}(Tt))\quad\hbox{a.e. on }\quad[0,1],
$$
there holds  $aT{y}(Tt)\in\partial H^\ast(-J\dot{y}^\ast(t))\;\hbox{a.e. on}\;[0,1]$ and
the Legendre reciprocity formula (cf. \cite[Proposition~II.1.15]{Ek90}) in convex analysis yields
\begin{eqnarray*}
 H^\ast(-J\dot{y}^\ast(t))&=&-H(aT{y}(Tt))+\langle-J\dot{y}^\ast(t), aT{y}(Tt)\rangle\\
&=&-(aT)^2H(y(t))+\langle-aTJ\dot{y}(Tt), aT{y}(Tt)\rangle\\
&=&-(aT)^2+(aT)^2\langle-J\dot{y}(Tt), {y}(Tt)\rangle\\
&=&-(aT)^2+2(aT)^2H({y}(Tt))=(aT)^2=T,\quad\hbox{a.e. on}\, [0,1],
\end{eqnarray*}
where the fourth equality comes from
the Euler formula \cite[Theorem~3.1]{YangWei08}.
Hence
$H^\ast(-J\dot{y}^\ast(t))=T$ a.e. on $[0,1]$ so that
 \begin{eqnarray*}
\int_0^1H^{\ast}(-J\dot{y}^{\ast}(t))dt=T.
\end{eqnarray*}
The definition of $\mu$ implies $T\ge \mu$ and it follows that  $A(y)\ge\mu$ by (\ref{e:desiredChar1}).

\begin{remark}\label{rem:action-period}
{\rm From the above proof we see that  $u$ given by Step 2 satisfies
\begin{eqnarray}\label{equalaction1}
&&\min\{A(x)>0\,|\,x\;\text{is a generalized}\;\Psi\hbox{-characteristic on}\;\mathcal{S}\}\nonumber\\
&=&A(x^{\ast})\nonumber\\
&=&I(u)\nonumber\\
&=&\min\{I(x)\,|\, x\in \mathcal{F}\,\&\,A(x)=1 \}.
\end{eqnarray}
Moreover, as in the periodic case we have
\begin{equation}\label{equalaction2}
\min\{I(x)\,|\, x\in \mathcal{F}\,\&\,A(x)=1 \}
=\left(\max\{A(x)\,|\, x\in \mathcal{F}\,\&\,I(x)=1 \}\right)^{-1}.
\end{equation}}
\end{remark}

\subsection{A key lemma}\label{sec:convex2}

\begin{lemma}\label{nointerior}
  Let $D\subset \mathbb{R}^{2n}$ be a compact  convex domain with
boundary $\mathcal{S}=\partial D$ and $0\in{\rm int}(D)$. If $\mathcal{S}$ is of class $C^{2n+2}$, then the set
   $$
    \Sigma^\Psi_S:=\{ A(x)\,|\, A(x)>0\;\hbox{and}\;x\;\hbox{is a }\;\Psi\hbox{-characteristic on}\;S\}
    $$
  {is a nowhere dense set} in $\mathbb{R}$.
 \end{lemma}

{This lemma shows that the set of actions of $\Psi$-characteristics on a suitable smooth convex hypersurface has no interior point,
which is crucial for the proof of (\ref{e:action-capacity1}) in Theorem~\ref{th:convex}.
In the proof of \cite[Proposition~4]{HoZe90} (the case $\Psi=I_{2n}$ of Theorem~\ref{th:convex})
the authors chose a smooth strictly convex surface $\tilde{\mathcal{S}}$ near $\mathcal{S}$ so that the set of actions of periodic orbits on $\tilde{\mathcal{S}}$
is discrete (\cite[page 422]{HoZe90}). We do not know whether there exists a similar result in our case.
For the action spectrum of a Hamiltonian map with compact support there exists a similar result,
Proposition~8 on the page 152 of \cite{HoZe94}, which  was proved by  Hofer and Zehnder  in \cite[pages 153-154]{HoZe94}
with an idea due to  Sikorav \cite{Sik90}. In order to avoid the arguments of smoothness for the action
functional on $H^{\frac{1}{2}}(S^1,\R^{2n})$ as given in \cite[Appendix~3]{HoZe94}
we shall use new techniques to discuss the smoothness of the action functional on a suitable subspace in $E^{\frac{1}{2}}$.}

 Fix $1<\alpha<2$. Since $\mathcal{S}$ is of class $C^{2n+2}$,
 the Minkowski functional $j_D: \mathbb{R}^{2n}\rightarrow\mathbb{R}$
 is $C^{2n+2}$ in $\mathbb{R}^{2n}\setminus\{0\}$ so that the  $C^1$ Hamiltonian function
 $$
 F: \mathbb{R}^{2n}\rightarrow \mathbb{R},\;z\mapsto (j_D(z))^{\alpha}
$$
is also $C^{2n+2}$ in $\mathbb{R}^{2n}\setminus\{0\}$.
To clarify the elements in the set $\Sigma^\Psi_S$, we only need to consider the action of $\Psi$-characteristic $x:[0,T^\ast]\rightarrow \mathcal{S}$ that satisfies
\begin{equation}\label{period1}
\dot{x}=J\nabla F(x)\quad\hbox{and}\quad x(T^\ast)=\Psi x(0).
\end{equation}
Clearly such a $\Psi$-characteristic $x$ has the action
\begin{equation}\label{period2}
A(x)=\alpha T^\ast/2.
\end{equation}
We only need to prove for an arbitrarily fixed $\sigma=\alpha T^\ast/2\in\Sigma^\Psi_S$
that $\Sigma^\Psi_S\cap(\sigma-\epsilon,\sigma+\epsilon)$ {is a nowhere dense set}
for some sufficiently small positive number $\epsilon$. To this end,
let us choose $0<\varepsilon_1<\varepsilon_2$ such that
    $B^{2n}({\varepsilon_1})\varsubsetneq B^{2n}({\varepsilon_2})\varsubsetneq D$
    and
    $$
    \max_{z\in B^{2n}({\varepsilon_2})}F(z)< \left(\frac{2(\sigma+\epsilon)}{\alpha}\right)^{\frac{\alpha}{\alpha-2}}.
    $$
    Take a smooth function $f:\mathbb{R}^{2n}\rightarrow [0,1]$ such that
   $$
   0\leq f\leq 1,\quad
   f|_{B_{\varepsilon_1}}=0 \quad\text{and}\quad f|_{B_{\varepsilon_2}^{c}}=1
   $$
   and define a Hamiltonian $\overline{F}:\mathbb{R}^{2n}\rightarrow \mathbb{R}$ by
   $$
   \overline{F}(z)=f(z)F(z)=f(z) (j_D(z))^{\alpha},\; \,\forall z\in \mathbb{R}^{2n}
   $$
   so that $\overline{F}\in C^{2n+2}(\mathbb{R}^{2n},\mathbb{R})$.
 If $x\in C^{1}([0,T],\mathbb{R}^{2n})$ satisfies
  \begin{equation}\label{period3}
\left\{
   \begin{array}{l}
   \dot{x}=J\nabla F(x),\quad F(x(t))\equiv 1,\\
    x(T)=\Psi x(0),\quad \frac{\alpha T}{2}\in (\sigma-\epsilon,\sigma+\epsilon)
   \end{array}
   \right.
\end{equation}
 it is easily computed  that
$$
y:[0,1]\rightarrow \mathbb{R}^{2n},\;
t\mapsto y(t)=T^{\frac{1}{\alpha-2}}x(tT)
$$
fulfils
$$
\dot{y}(t)=J\nabla F(y(t)),\quad y(1)=\Psi y(0)\quad\hbox{and}\quad F(y(t))=T^{\frac{\alpha}{\alpha-2}}
\geq  \left(\frac{2(\sigma+\epsilon)}{\alpha} \right)^{\frac{\alpha}{\alpha-2}}.
$$
Hence $y(t)\subset (B_{\varepsilon_2})^c$, $\forall t\in [0,1]$. Since
$\overline{F}=F$ on $(B_{\varepsilon_2})^c$, we have
$$
\dot{y}=J\nabla\overline{F}(y),\quad y(1)=\Psi y(0), \quad\text{and}\quad \overline{F}(y(t))=F(y(t))\;\forall t.
$$
Let $\mathbb{E}=E^{\frac{1}{2}}$ be defined by (\ref{espace}). Then $y$ is a critical point of the functional
$$
\Phi_{\overline{F}}:\mathbb{E}\rightarrow \mathbb{R},\;x\mapsto \frac{1}{2}\|x^+\|_{\mathbb{E}}^2-\frac{1}{2}\|x^-\|_{\mathbb{E}}^2-\int_0^{1}\overline{F}(x(t)) dt
$$
(which is well-defined since $\mathbb{E}$ embeds continuously into $L^2$ and so into $L^{\alpha}$ for $1<\alpha<2$),
and a direct computation yields
 \begin{eqnarray}\label{period4}
 \Phi_{\overline{F}}(y)&=&\frac{1}{2}\int_0^{1}\langle-J\dot{y},y\rangle-\int_0^1F(y(t))\nonumber\\
                       &=&\left(\frac{\alpha}{2}-1\right)F(y(t))\nonumber\\
                       &=&\left(\frac{\alpha}{2}-1\right)T^{\frac{\alpha}{\alpha-2}}.
 \end{eqnarray}
Note that all critical points of $\Phi_{\overline{F}}$ sit in the Banach space
$$
C^{2n+2}_\Psi:=\{z\in C^{2n+2}([0,1],\mathbb{R}^{2n})\,|\, z(1)=\Psi z(0)\}
$$
which is a subspace of $C^{2n+2}([0,1],\mathbb{R}^{2n})$. {In particular,}
$\Phi_{\overline{F}}|_{\mathbb{E}}$ and $\Phi_{\overline{F}}|_{C^1_\Psi}$
have the same critical value sets.

\begin{claim}\label{cl:smooth}
$\Phi_{\overline{F}}|_{C^1_\Psi}$ is of class $C^{2n+1}$.
\end{claim}

In order to prove this we need the following result  from the page 780 of \cite{Eell66}.

\begin{proposition}\label{prop:A.smooth}
Let $M$, $N$ and $P$ be finite-dimensional $C^\infty$-manifolds.
If both $M$ and $N$ are compact then the map ${\rm comp}:
C^{r+s}(N,P)\times C^r(M,N) \to C^r(M,P)$ given by ${\rm
comp}(f,g)=f\circ g$ is $C^s$. In particular, the evaluation map
$C^{r}(N,P)\times N \to P$ is $C^r$.
\end{proposition}

\begin{proof}[\it Proof of Claim~\ref{cl:smooth}]
It suffices to prove that $\Phi_{\overline{F}}$ is $C^{2n+1}$ in any ball
$B(\hat{x}, R)\subset C^1_\Psi$ which is centered at $\hat{x}\in C^1_\Psi$ and has
radius $R$. {Observe} that there exists $\hat{R}>0$ such that
$$
x([0,1])\subset B^{2n}(0, \hat{R})\quad\forall x\in B(\hat{x},R).
$$
 Hence $B(\hat{x},R)$ is contained in
$$
C^1_\Psi([0,1], B^{2n}(0, \hat{R})):=\{z\in C^1([0,1],B^{2n}(0, \hat{R}))\,|\, z(1)=\Psi z(0)\}
$$
which is an open subset of $C^1_\Psi$. Let $\bar{B}^{2n}(0, \hat{R})$ denote
the closure of ${B}^{2n}(0, \hat{R})$. By  Proposition~\ref{prop:A.smooth} and the $C^{2n+2}$-smoothness of $\overline{F}$, the map
$$
C^1_\Psi([0,1], \bar{B}^{2n}(0, \hat{R}))\to C^1([0,1], \mathbb{R}),\;x\mapsto\overline{F}\circ x
$$
is $C^{2n+1}$ and so is
$$
C^1_\Psi([0,1], \bar{B}^{2n}(0, \hat{R}))\ni x\mapsto
\int_0^{1}\overline{F}(x(t)) dt\in\mathbb{R}.
$$
It follows that
$$
C^1_\Psi([0,1], \bar{B}^{2n}(0, \hat{R}))\ni x\mapsto \Phi_{\overline{F}}(x)=\frac{1}{2}\|x^+\|_{\mathbb{E}}^2-\frac{1}{2}\|x^-\|^2_{\mathbb{E}}-\int_0^{1}\overline{F}(x(t)) dt
$$
is $C^{2n+1}$ since $C^1_\Psi([0,1], \bar{B}^{2n}(0, \hat{R}))\hookrightarrow E^{1/2}$ is smooth.
This implies the expected claim.
\end{proof}

Take a smooth function $g:[0,1]\rightarrow [0,1]$ such that $g$ equals $1$ (resp. $0$)
near $0$ (resp. $1$). Denote by $\phi^t$ the flow of $X_{\overline{F}}$.
Then by the $C^{2n+1}$-smoothness of $X_{\overline{F}}$ we have a $C^{2n+1}$ map
$$
\psi:[0,1]\times\mathbb{R}^{2n}\rightarrow \mathbb{R}^{2n},\;(t,z)\mapsto g(t)\phi^t(z)+(1-g(t))\phi^{t-1}(\Psi z).
$$
Clearly $\psi (0,z)=\phi^0(z)=z$, $\psi(1,z)=\phi^0(\Psi z)=\Psi z$ and thus
$\psi(1,z)=\Psi\psi(0,z)$.
Moreover,  for each $z\in \mathbb{R}^{2n}$ satisfying $\phi^1(z)=\Psi z$, it holds that
\begin{eqnarray*}
\psi(t,z)&=&g(t)\phi^t(z)+(1-g(t))\phi^{t-1}(\Psi z)\\
&=&g(t)\phi^t(z)+(1-g(t))\phi^{t-1}(\phi^{1} (z))\\
&=&g(t)\phi^t(z)+(1-g(t))\phi^{t}(z)\\
&=&\phi^t(z),\quad\forall \;0\le t\le 1.
\end{eqnarray*}

Clearly, Proposition~\ref{prop:A.smooth}
implies  that the composition
$$
{\rm comp}:C^{r+s}(N\times M,P)\times C^r(M,N\times M) \to C^r(M,P),\; (f,g)\mapsto
f\circ g
$$
is $C^s$. Note that the  map $N\ni p\mapsto g_p\in C^r(M,N\times M)$
defined by $g_p(m)=(p,m)$ for $m\in M$ is smooth.
Each  map $\varphi\in C^{r+s}(N\times M,P)$ gives rise to a $C^s$ map
$$
N\to C^r(M,P),\; p\mapsto {\rm comp}(\varphi, g_p)=\varphi(p,\cdot).
$$
Applying this claim to $N=[0,1]$, $M=\bar{B}^{2n}(0,R)$ for any given $R>0$, $P=\mathbb{R}^{2n}$
and the $C^{2n+1}$ map
$$
{\varphi=\psi|_{[0,1]\times\bar{B}^{2n}(0,R)}}: [0,1]\times \bar{B}^{2n}(0,R),\; (t,z)\mapsto\psi(t,z)\in\mathbb{R}^{2n}
$$
 we deduce that $\bar{B}^{2n}(0,R)\ni z\mapsto \psi(\cdot,z)\in C^1_\Psi$
is $C^{2n}$ since $\psi(\cdot,z):[0,1]\rightarrow \mathbb{R}^{2n}$
sits in $C^1_\Psi$.  So
\begin{equation}\label{e:smooth1}
\Omega:\mathbb{R}^{2n}\rightarrow C^1_\Psi,\;z\mapsto\Omega(z)=\psi(\cdot,z)
\end{equation}
is $C^{2n}$.  (It is not hard to prove this directly!)
This and Claim~\ref{cl:smooth} show that the composition
\begin{equation}\label{e:smooth2}
\Phi_{\overline{F}}|_{C^1_\Psi}\circ\Omega:\mathbb{R}^{2n}\rightarrow \mathbb{R}
\end{equation}
is of class $C^{2n}$.
Note that every critical point $y$ of the functional
$\Phi_{\overline{F}}|_{C^1_\Psi}$  has the form $y(t)=\phi^t(z_y)$ for some $z_y\in\mathbb{R}^{2n}$ satisfying
$\phi^1(z_y)=\Psi z_y$ and thus $y=\Omega(z_y)$. Hence $z_y$ is a critical point of $\Phi_{\overline{F}}|_{C^1_\Psi}\circ\Omega$.
In particular,  the critical values of $\Phi_{\overline{F}}|_{C^1_\Psi}$ (and hence $\Phi_{\overline{F}}$) are contained in the set of critical values of $\Phi_{\overline{F}}|_{C^1_\Psi}\circ\Omega$ which is a nowhere dense set by Sard theorem.

Now since $\{\Phi_{\overline{F}}(y)\,|\,y\in{\rm Crit}(\Phi_{\overline{F}})\}$ is a nowhere dense set,
so is
$$
\{\Phi_{\overline{F}}(y)\,|\,y(t)=T^{\frac{1}{\alpha-2}}x(tT)\;\hbox{and}\;x\;\hbox{satisfies}\;(\ref{period3})
     \}.
     $$
By (\ref{period2}) and (\ref{period4}), this set is equal to
$$
\left\{\left(\frac{\alpha}{2}-1\right)\left(\frac{2}{\alpha}A(x)\right)^{\frac{\alpha}{\alpha-2}}
\;\Biggm|\;x\;\hbox{is a $\Psi$-characteristic on $\mathcal{S}$ and}\; A(x)\in (\sigma-\epsilon,\sigma+\epsilon)\right\}.
$$
Hence $\Sigma^\Psi_S\cap(\sigma-\epsilon,\sigma+\epsilon)=\{A(x)\,|\,x
\;\hbox{is a $\Psi$-characteristic on $\mathcal{S}$
and}\; A(x)\in (\sigma-\epsilon,\sigma+\epsilon)\}$
is a nowhere dense set. Then Lemma~\ref{nointerior} follows.

\subsection{Proof of (\ref{e:action-capacity1}) for smooth and strictly convex $D$}\label{sec:convex3}

{The proof is similar to that of  \cite[Propposition~4]{HoZe90} and \cite{HoZe94} with slight change.
We sketch the proof with some details omitted which  can be found in \cite{JinLu}.}

\noindent{\bf Step 1}.\quad{\it Prove}
 \begin{equation}\label{e:action-capacity5}
   c^\Psi_{\rm HZ}(D,\omega_0)\ge A(x^{\ast}).
   \end{equation}
For small $0<\epsilon, \delta<1/2$, pick  a smooth
 function $f: [0,1]\rightarrow\mathbb{R}$ such that
   \begin{eqnarray*}
     &f(t)=0,& \,\,\,   t\le\delta, \\
     & f(t)=A(x^{\ast})-\varepsilon  ,& \,\,\, 1-\delta\le t, \\
    & 0\leq f'(t)<A(x^{\ast}),& \,\,\, \delta<t<1-\delta.
   \end{eqnarray*}
Define $H(x)=f(j^2_D(x))$ for $x\in D$. Then
   $H\in \mathcal{H}^\Psi(D, \omega_0)$. Let us prove that every solution
   $x:[0, T]\to D$ of the boundary value problem
   \begin{equation}\label{e:action-capacity6}
        \dot{x}=J\nabla H(x)=f'(j_D^2(x))J\nabla j_D^2(x)\quad\hbox{and}\quad
     x(T)=\Psi x(0)
     \end{equation}
with $0<T\le 1$ is constant. By contradiction we assume that $x=x(t)$
is a nonconstant solution of (\ref{e:action-capacity6}). Then
$j_D(x(t))$ is equal to a nonzero constant and thus $x(t)\ne 0$
for each $t\in [0, T]$. Moreover, $f'(j_D^2(x(t)))\equiv a\in (0, A(x^\ast))$.
Since $\nabla j_D^2(\lambda z)=\lambda\nabla j_D^2(z)$ for all $(\lambda,z)\in\mathbb{R}_+\times\mathbb{R}^{2n}$, multiplying $x(t)$
by a suitable positive number we may assume that $x([0,T])\subset\mathcal{S}=\partial D$
and
\begin{equation}\label{e:action-capacity7}
        \dot{x}=aJ\nabla j_D^2(x)\quad\hbox{and}\quad
     x(T)=\Psi x(0).
     \end{equation}
Note that $\langle \nabla j_D^2(z), z\rangle=j_D^2(z)=1$ for any $z\in\mathcal{S}$. We deduce
from (\ref{e:action-capacity7}) that
$$
0<A(x)=aT\le a<A(x^\ast),
$$
which contradicts (\ref{e:action2}).
This shows that $H\in \mathcal{H}^\Psi(D, \omega_0)$ is admissible and hence
$c^\Psi_{\rm HZ}(D,\omega_0)\ge m(H)=A(x^{\ast})-\epsilon$.
Let $\epsilon\to 0$ and we get (\ref{e:action-capacity5}).

\noindent{\bf Step 2.}\quad {\it Prove}
 \begin{equation}\label{e:action-capacity8}
   c^\Psi_{\rm HZ}(D,\omega_0)\le A(x^{\ast}).
   \end{equation}
Let $H\in \mathcal{H}^\Psi(D, \omega_0)$ satisfy $m(H)>A(x^\ast)$.
We wish to prove that the boundary value problem
   \begin{equation}\label{e:action-capacity9}
        \dot{x}=J\nabla H(x)\quad\hbox{and}\quad
     x(1)=\Psi x(0)
     \end{equation}
has a nonconstant solution $x:[0,1]\to D$.
By Lemma~\ref{nointerior} we have a small number $\epsilon>0$ such that
$A(x^\ast)+\epsilon\notin \Sigma^\Psi_S$ and $m(H)>A(x^\ast)+\epsilon$. This means that
the boundary value problem
   \begin{equation}\label{e:action-capacity10}
        \dot{x}=(A(x^\ast)+\epsilon)J\nabla j_D^2(x)\quad\hbox{and}\quad
     x(1)=\Psi x(0)
     \end{equation}
admits only the trivial solution $x\equiv 0$. (Otherwise, we have
$x(t)\ne 0\;\forall t\in [0, 1]$ as above. Thus after
 multiplying $x(t)$ by a suitable positive number we may assume that
  $x([0,1])\subset\mathcal{S}=\partial D$, which leads to
  $A(x)=A(x^\ast)+\epsilon$.)
For a fixed number $\delta>0$ we take a smooth
 function $f: [1, \infty)\rightarrow\mathbb{R}$ such that
   \begin{eqnarray*}
     &f(t)\ge (A(x^\ast)+\epsilon)t,& \,\,\,   t\ge 1, \\
     &f(t)=(A(x^\ast)+\epsilon)t,& \,\,\,   t\;\hbox{large}, \\
     & f(t)=m(H),& \,\,\, 1\le t\le 1+\delta, \\
    & 0\leq f'(t)\le A(x^{\ast})+\epsilon,& \,\,\, t>1+\delta.
   \end{eqnarray*}
With this $f$ we get an extension of $H$ as the following
$$
 \overline{H}(z)=\left\{
   \begin{array}{l}
     H(z),\quad\hbox{for}\;z\in D, \\
     f(j_D^2(z)), \quad\hbox{for}\;z\notin D.
   \end{array}
   \right.
$$
Let $\mathbb{E}=E^{\frac{1}{2}}=\mathbb{E}^{-}\oplus \mathbb{E}^0 \oplus \mathbb{E}^{+}$ be as in
(\ref{e:spaceDecomposition}) and $\Phi_{\overline{H}}$ be as in (\ref{e:EH.1.1}), that is,
 \begin{equation}\label{e:action-capacity11}
\Phi_{\overline{H}}(x)=\frac{1}{2}\|x^+\|_{\mathbb{E}}^2-\frac{1}{2}\|x^-\|_{\mathbb{E}}^2-\int^1_0\overline{H}(x(t))dt.
\end{equation}
   {If $x$ is a solution of  $\dot{x}(t)=X_{\overline{H}}(x(t))$ satisfying  $x(1)=\Psi x(0)$ and  $\Phi_{\overline{H}}(x)>0$,
   then it is nonconstant, sits in $D$ completely, and thus
   is a solution of  $\dot{x}=X_H(x)$ on $D$ (\cite[Lemma~4.6]{JinLu}).
    Hence we only need to find a critical point of $\Phi_{\overline{H}}(x)$
    with positive critical value (see \cite[Lemma~4]{HoZe90}).}

{The fact that $A(x^\ast)+\epsilon\notin \Sigma^\Psi_S$ implies the following lemma.}
\begin{lemma}\label{lem:PS}
   If  a sequence $(x_k)\subset\mathbb{E}$ is such that $\nabla \Phi_{\overline{H}}(x_k)\rightarrow 0$ in $\mathbb{E}$, then it
   has a convergent subsequence in $\mathbb{E}$. In particular, $\Phi_{\overline{H}}$ satisfies the Palais-Smale condition.
\end{lemma}

\begin{proof}
If $(x_k)$ is bounded in $\mathbb{E}$, as in the proof of Proposition~\ref{prop:PSmale} we deduce that
$(x_k)$ has a convergent subsequence. Without loss of generality, we assume
   $\lim_{k\rightarrow +\infty }\|x_k\|_{\mathbb{E}}=+\infty$.
         Let $y_k=\frac{x_k}{\|x_k\|_{\mathbb{E}}}$. Then $\|y_k\|_{\mathbb{E}}=1$ and satisfies
   \begin{equation}\label{e:4.6.1}
  y_k^+-y_k^--\frac{1}{\|x_k\|_{\mathbb{E}}}\nabla \mathfrak{b}(x_k)= y_k^+-y_k^--j^{\ast}\left(\frac{ \nabla \overline{H}(x_k)}{\|x_k\|_{\mathbb{E}}}\right)\rightarrow 0\quad\hbox{in}\quad \mathbb{E}.
   \end{equation}
By the construction of $\overline{H}$ and the proof of Proposition~\ref{prop:PSmale},
 passing to a subsequence (if necessary) we  may assume that  $y_{k}\rightarrow  y$ in $\mathbb{E}$ so that $\|y\|_{\mathbb{E}}=1$.
 Since $\overline{H}(z)=Q(z):=(A(x^\ast)+\epsilon)j_D^2(z)$ for $|z|$  large enough, arguing as in the proof of Proposition~\ref{prop:PSmale} we get
 $$
 \left\|\frac{\nabla \overline{H}(x_k)}{\|x_k\|_{\mathbb{E}}}-\nabla Q(y)\right\|_{L^2}\rightarrow 0
 $$
 so that (\ref{e:4.6.1}) becomes
 $$
 y^+-y^--j^{\ast}\nabla Q(y)=0.
 $$
  Hence $y$ satisfies  the boundary value problem
   (\ref{e:action-capacity10}) and thus $y=0$ because
   $A(x^\ast)+\epsilon\notin \Sigma^\Psi_S$.
   This contradicts the fact $\|y\|_{\mathbb{E}}=1$. That is,  $(x_k)$ must be bounded in $\mathbb{E}$.
\end{proof}

The $\Psi$-characteristic $x^\ast$ with minimal action in (\ref{e:action2}) can be
reparametrization as $x_0:[0,1]\to\mathbb{R}^{2n}$ such that (cf. \cite[\S4.3]{JinLu})
\begin{equation}\label{e:action-capacity12}
\left\{
   \begin{array}{l}
     \dot{x_0}=A(x^\ast)J\nabla j_D^2(x_0),\\
      x_0(1)=\Psi x_0(0),\;  A(x_0)=A(x^\ast),\\
    j_D(x_0(t))\equiv 1,\;\hbox{i.e.,}\; x_0([0,1])\subset\mathcal{S}.
   \end{array}
   \right.
\end{equation}
Denote by $x_0^+$ the projections of $x_0$ onto $\mathbb{E}^+$.
 Then $x_0^+\ne 0$. (Otherwise,
 a contradiction occurs because $0<A(x^\ast)=A(x_0)=-\frac{1}{2}\|x_0^-\|_\mathbb{E}^2$.)
 Following \cite{HoZe90}
 we define for $s>0$ and $\tau>0$
 \begin{eqnarray*}
 &&W_s:= \mathbb{E}^-\oplus \mathbb{E}^0\oplus sx_0^+,\\
 &&\Sigma_\tau:=\{x^-+x^0+sx_0^+\,|\,0\le s\le\tau,\;\|x^-+x^0\|_\mathbb{E}\le\tau\}.
 \end{eqnarray*}
Let $\partial\Sigma(\tau)$ denote the boundary of $\Sigma_\tau$ in $\mathbb{E}^-\oplus \mathbb{E}^0\oplus\mathbb{R}x_0^+$. Then
 \begin{equation}\label{boundary}
  \partial\Sigma_\tau=\{x=x^-+x^0 +sx_0^+\in\Sigma_\tau\,|\,\|x^-+x^0\|_{\mathbb{E}}=\tau \;\text{or}\; s=0 \;\text{or} \; s=\tau \}.
  \end{equation}
Repeating the proofs of Lemmas~5,~6 in \cite{HoZe90} leads to
\begin{lemma}\label{lem:HZ5}
  There exists a constant $C>0$ such that for any $s\ge 0$,
  $$
  \Phi_{\overline{H}}(x)\le-\epsilon\int^1_0j_D^2(x(t))dt+C,\quad\forall x\in W_s.
  $$
\end{lemma}

\begin{lemma}\label{lem:HZ6}
    $\Phi_{\overline{H}}|\partial\Sigma_\tau\le 0$ if $\tau>0$ is sufficiently large.
\end{lemma}

Arguing as in the proof of Lemma~9, Lemma~10 in Chapter~3 of \cite{HoZe94} we get
 \begin{lemma}\label{lem:HZ7}
 For $z_0\in {\rm Fix}(\Psi)\cap H^{-1}(0)$,
there exist constants $\alpha>0$ and $\beta>0$ such that
  $$
  \Phi_{\overline{H}}|\Gamma_\alpha\geq\beta>0,
  $$
  where $\Gamma_\alpha=\{z_0+x\,|\,x\in \mathbb{E}^+\,\&\,\|x\|_{\mathbb{E}}=\alpha\}$.
  \end{lemma}

Let $\phi^t$ be the negative gradient flow of $\Phi_{\overline{H}}$. As in \cite[pages 95-97]{HoZe94} the following lemma
can be proved  by the standard topological degree method (cf. \cite[Lemma~4.10]{JinLu}).

\begin{lemma}\label{positive+}
$\phi^t(\Sigma_\tau)\cap\Gamma_\alpha\neq \emptyset, \,\forall t\geq 0$.
 \end{lemma}

Let $\mathcal{F}=\{\phi^t(\Sigma_\tau)| t\geq 0\}$ and define
$$
c(\Phi_{\overline{H}}, \mathcal{F}):=\inf_{t\geq 0}\sup_{x\in \phi^t(\Sigma_\tau)}\Phi_{\overline{H}}(x).
$$
Lemmas~\ref{lem:HZ7},~\ref{positive+} imply
 $$
 0<\beta\leq \inf _{x\in\Gamma_\alpha}\Phi_{\overline{H}}(x)\leq \sup_{x\in \phi^t(\Sigma_\tau)}\Phi_{\overline{H}}(x)\;
\forall t\geq 0,
$$
and hence $c(\Phi_{\overline{H}}, \mathcal{F})\geq\beta>0$.
On the other hand,  since $\Sigma_\tau$ is bounded and Proposition~\ref{Lip} implies that $\Phi_{\overline{H}}$
maps  bounded sets into bounded sets we arrive at
$$
c(\Phi_{\overline{H}}, \mathcal{F})\leq \sup_{x\in\Sigma_\tau}\Phi_{\overline{H}}(x)<\infty.
$$
 Using the Minimax Lemma on \cite[page 79]{HoZe94}, we get a critical point $x$ of $\Phi_{\overline{H}}$
 with $\Phi_{\overline{H}}(x)>0$ and (\ref{e:action-capacity1}) is proved.

\subsection{Completing the proof of Theorem~\ref{th:convex} for general case}\label{sec:convex4}

By Proposition~1.12 and Corollary~2.41 in \cite{Kr15}
we may choose two sequences of $C^\infty$ strictly convex domains with  boundaries,
 $(D^+_k)$ and $(D^-_k)$, such that
 \begin{description}
 \item[(i)] $D^-_1\subset D^-_2\subset\cdots\subset D$ and $\cup^\infty_{k=1}D^-_k=D$,
 \item[(ii)] $D^+_1\supseteq D^+_2\supseteq\cdots\supseteq D$ and $\cap^\infty_{k=1}D^+_k=D$,
 \item[(iii)] for any small neighborhood $O$ of $\partial D$ there exists an integer
 $N>0$ such that $\partial D^+_k\cup\partial D^-_k\subset O\;\forall k\ge N$.
 \end{description}
Denote by $j_D, j_{D^+_k}$ and $j_{D^-_k}$ the Minkowski  functionals of
 $D, D^+_k$ and $D^-_k$ respectively. Let $H=(j_D)^2, H^+_k=(j_{D^+_k})^2$ and $H^{-}_k=(j_{D^-_k})^2$
for each $k\in\mathbb{N}$. Their  Legendre transformations are
$H^\ast, H_k^{+\ast}$ and $H_k^{-\ast}$, $k=1,2,\cdots$. Denote by
$$
I(u)=\int^1_0H^\ast(-J\dot{u}),\quad I^+_k(u)=\int^1_0H_k^{+\ast}(-J\dot{u}),\quad I^-_k(u)=\int^1_0H_k^{-\ast}(-J\dot{u})
$$
for $u\in\mathcal{A}$. Note that (i) and (ii) imply
 \begin{description}
 \item[(iv)] $j_{D^-_1}\ge j_{D^-_2}\ge\cdots\ge j_D$ and so $H_1^{-\ast}\le H_2^{-\ast}\le\cdots\le H^\ast$,
 \item[(v)] $j_{D^+_1}\le j_{D^+_2}\le\cdots\le j_D$ and so $H_1^{+\ast}\ge H_2^{+\ast}\ge\cdots\ge H^\ast$.
 \end{description}
These lead to
\begin{equation}\label{e:action-capacity13}
I^+_1(u)\ge I^+_2(u)\ge\cdots\ge I(u)\ge\cdots\ge I^-_2(u)\ge I^-_1(u),\quad\forall u\in  \mathcal{A}.
\end{equation}
By the first three steps in Section~\ref{sec:convex1} these functionals attain their
minimums on $\mathcal{A}$. It easily follows from (\ref{e:action-capacity13}) that
\begin{equation}\label{e:action-capacity14}
\min_{\mathcal{A}}I^+_1\ge \min_{\mathcal{A}}I^+_2\ge\cdots\ge \min_{\mathcal{A}}I\ge\cdots\ge
\min_{\mathcal{A}}I^-_2\ge \min_{\mathcal{A}}I^-_1.
\end{equation}
Now (\ref{e:action2}) gives rise to
\begin{equation}\label{e:action-capacity15}
\min_{\mathcal{A}}I=\min\{A(x)>0\,|\,x\;\text{is a generalized}\;\Psi\hbox{-characteristic on}\;\mathcal{S}\}
\end{equation}
and (\ref{e:action2})-(\ref{e:action-capacity1}) yield
\begin{eqnarray}\label{e:action-capacity16}
{c}^\Psi_{\rm HZ}(D^+_k,\omega_0)=\min_{\mathcal{A}}I^+_k\quad\hbox{and}\quad
{c}^\Psi_{\rm HZ}(D^-_k,\omega_0)=\min_{\mathcal{A}}I^-_k
\end{eqnarray}
for each $k\in\mathbb{N}$. By this, (\ref{e:action-capacity14})
and the monotonicity of ${c}^\Psi_{\rm HZ}$ we get
$$
\begin{array}{ccccc}
 {c}^\Psi_{\rm HZ}(D^+_k,\omega_0)& \ge& {c}^\Psi_{\rm HZ}(D,\omega_0)&\ge & {c}^\Psi_{\rm HZ}(D^-_k,\omega_0)\\
 \parallel&  &&  &\parallel\\
\min_{\mathcal{A}}I^+_k& \ge&  \min_{\mathcal{A}}I&\ge &\min_{\mathcal{A}}I^-_k
\end{array}
$$
Observe that $\lim_{k\to\infty}c^\Psi_{\rm HZ}(D^+_k,\omega_0)=
 {c}^\Psi_{\rm HZ}(D,\omega_0)$ and $\lim_{k\to\infty}{c}^\Psi_{\rm HZ}(D^-_k,\omega_0)
={c}^\Psi_{\rm HZ}(D,\omega_0)$ by Proposition~\ref{MonComf}(iii).
Hence ${c}^\Psi_{\rm HZ}(D,\omega_0)=\min_{\mathcal{A}}I$ by the squeezing theorem in calculus.
The desired result follows from this and (\ref{e:action-capacity15}).

\begin{remark}\label{rem:twocapacity}
{\rm

For $\Psi\in{\rm Sp}(2n,\mathbb{R})$ and an open set $O$ in $(\mathbb{R}^{2n},\omega_0)$
containing the origin $0\in\mathbb{R}^{2n}$
let $\mathcal{H}^\Psi_0(O,\omega_0)$ consist of all $H\in\mathcal{H}^\Psi (O,\omega_0)$ which vanish near $0$,
and define
\begin{equation}\label{caplow}
\underline{c}^\Psi_{\rm HZ}(O,\omega_0)=\sup \{\max H\,|\, H\in \mathcal{H} ^{\Psi}_0(O,\omega_0)\;\text{and}\;H\;\text{is $\Psi$-admissible}\}.
\end{equation}
Then $\underline{c}^\Psi_{\rm HZ}(O,\omega_0)\leq c^\Psi_{\rm HZ}(O,\omega_0)$ and $\underline{c}^\Psi_{\rm HZ}(O,\omega_0)\leq \underline{c}^\Psi_{\rm HZ}(O^\ast,\omega_0)$
for another open subset $O^\ast\supseteq O$.
If $B^{2n}(0,r)\subset O\subset B^{2n}(0,R)$ then ${c}^\Psi_{\rm HZ}(O,\omega_0)\le \left(\frac{2R}{r}\right)^2\underline{c}^\Psi_{\rm HZ}(O,\omega_0)$
(\cite[Proposition~1.3]{JinLu}).
  In particular, if this $O$ is convex, by the definition of
$\underline{c}^\Psi_{\rm HZ}$ it is easily proved that
\begin{equation}\label{innRegu}
\underline{c}^\Psi_{\rm HZ}(O,\omega_0)=\sup\{\underline{c}^\Psi_{\rm HZ}(K,\omega_0)\,|\,
K\;\hbox{convex bounded domain},\;K\ni 0,\;
\overline{K}\subset O\}.
\end{equation}
Note that  (\ref{e:action-capacity5}) and (\ref{e:action-capacity8})
imply $\underline{c}^\Psi_{\rm HZ}(D,\omega_0)={c}^\Psi_{\rm HZ}(D,\omega_0)$ for
a $C^\infty$ strictly convex bounded domain $D\subset \mathbb{R}^{2n}$ containing $0$.
(\ref{innRegu})  and the inner regularity of  $c^\Psi_{\rm HZ}$ in (\ref{cap+++}) lead to
$\underline{c}^\Psi_{\rm HZ}(O,\omega_0)={c}^\Psi_{\rm HZ}(O,\omega_0)$.}
\end{remark}

\section{Proof of Theorem~\ref{th:EH.1.6}}\label{sec:EH.2}
\setcounter{equation}{0}

Our arguments are closely related to  Sikorav's approach in \cite{Sik90}.
The proof will be completed by several propositions.

{By definition of the admissible deformation $\gamma$ (Definition~\ref{defdeform}) and
the theory of topological degree, the following composition and intersection properties hold
(see \cite[Section~3.1]{Sik90} and \cite[\S3]{JinLu}).}

\begin{proposition}\label{prop:EH.1.1}
  \begin{description}
   \item[(i)] For any  $\gamma \in\Gamma$ and $\tilde{\gamma} \in\Gamma$ there holds $\gamma \circ\tilde{\gamma} \in\Gamma$.
   \item[(ii)]    Denote by $S^+$ the unit sphere in $\mathbb{E}^+$. For any $e\in \mathbb{E}^+\setminus\{0\}$ and $\gamma\in\Gamma$ there holds
       \begin{equation}\label{nonempty}
       \gamma(S^+)\cap (\mathbb{E}^-\oplus \mathbb{E}^0\oplus \mathbb{R}_+e)\neq \emptyset.
       \end{equation}
  \end{description}
\end{proposition}

{The estimation for $c^{\Psi}_{\rm EH}(H)$ follows immediately, which is a slight change of Proposition~1 in Section~3.2 of \cite{Sik90}}.

\begin{proposition}\label{prop:EH.1.3}
If $H\in C^0(\mathbb{R}^{2n},\mathbb{R}_{\ge 0})$ then
\begin{equation}\label{e:EH.1.3}
c^{\Psi}_{\rm EH}(H)\le\sup_{z\in\mathbb{C}^n}\left(\frac{ \mathfrak{t}(\Psi)}{2}|z|^2-H(z)\right)
\end{equation}
where $\mathfrak{t}(\Psi)$ is defined in (\ref{e:TPsi}). Moreover, if
\begin{equation}\label{e:EH.1.4}
{\Psi=\Psi_1\oplus\Psi_2,\quad
\hbox{where}\;\Psi_1\in{\rm Sp}(2,\mathbb{R})\; \hbox{and} \;\Psi_2\in{\rm Sp}(2n-2,\mathbb{R}),}
\end{equation}
then we have
\begin{equation}\label{e:EH.1.5}
c^{\Psi}_{\rm EH}(H)\le\sup_{z\in\mathbb{C}^n}\left(\frac{\mathfrak{t}(\Psi_1)}{2}|z_1|^2-H(z)\right).
\end{equation}
\end{proposition}

\begin{proof}
   Let $e(t)=e^{ \mathfrak{t}(\Psi) Jt}X$ where $J$ is as in (\ref{e:standcompl}) and  $X\in\mathbb{R}^{2n}$ satisfies
   $e^{ \mathfrak{t}(\Psi)J}X=\Psi X$ and $|X|=1$. For any  $x=y+\lambda e$, where
$y\in \mathbb{E}^-\oplus \mathbb{E}^0$ and $\lambda>0$, there holds
\begin{equation}\label{e:EH.2.0}
\mathfrak{a}(x)\le\frac{1}{2}\|\lambda e\|^2_{\mathbb{E}}= \frac{ \mathfrak{t}(\Psi) }{2}\lambda^2
\end{equation}
 and
\begin{equation}\label{e:EH.2.1}
\int_0^1\langle x(t),e^{ \mathfrak{t}(\Psi) Jt}X\rangle dt=\int_0^1\langle \lambda e^{ \mathfrak{t}(\Psi) Jt}X,e^{ \mathfrak{t}(\Psi) Jt}X\rangle dt=\lambda.
\end{equation}
It follows that
$$
\mathfrak{a}(x)\le\frac{ \mathfrak{t}(\Psi) }{2}\left(\int_0^1\langle x(t),e^{ \mathfrak{t}(\Psi) Jt}X\rangle dt\right)^2\le\frac{ \mathfrak{t}(\Psi) }{2}\int_0^1|x(t)|^2 dt.
$$
By  Proposition~\ref{prop:EH.1.1}(ii), we get
$$
\inf_{x\in\gamma(S^+)}\Phi_H(x) \le\sup_{x\in \mathbb{E}^-\oplus \mathbb{E}^0\oplus\mathbb{R}_+e}\Phi_H(x)\le \sup_{z\in\mathbb{R}^{2n}}\frac{ \mathfrak{t}(\Psi) }{2}|z|^2-H(z),\quad\forall\,\gamma\in\Gamma
$$
and hence (\ref{e:EH.1.3}) is proved.

Now suppose that $\Psi$ has the form in (\ref{e:EH.1.4}). Let $\hat{e}(t)=e^{\mathfrak{t}(\Psi_1) Jt}\hat{X}$, where $\hat{X}=(X_1,0)\in\mathbb{R}^{2}\times\mathbb{R}^{2n-2}$ satisfies $
   e^{\mathfrak{t}(\Psi_1)J}\hat{X}=\Psi \hat{X}$ and $|\hat{X}|=1$.  For any $x=y+\lambda\hat{e}$
   where $y\in\mathbb{E}^-\oplus \mathbb{E}^0$ and $\lambda>0$, write
   $x(t)=(x_1(t),x_2(t))\in\mathbb{R}^{2}\times\mathbb{R}^{2n-2}$.
   Let $J$ also denote the complex structure on $\mathbb{R}^2$. Then as the reasoning of (\ref{e:EH.2.1}) we get
   $$
   \int_0^1\langle x_1(t),e^{ \mathfrak{t}(\Psi_1) Jt}X_1\rangle dt
   =\int_0^1\langle x(t),e^{ \mathfrak{t}(\Psi_1) Jt}\hat{X}\rangle dt=\int_0^1\langle \lambda e^{ \mathfrak{t}(\Psi_1) Jt}\hat{X},e^{ \mathfrak{t}(\Psi_1) Jt}\hat{X}\rangle dt=\lambda.
   $$
   The same arguments as the proof of (\ref{e:EH.1.3}) lead to (\ref{e:EH.1.5}).
    \end{proof}

{As a generalization of \cite[Proposition~2,Section~3.2]{Sik90} and \cite[Lemma~9,Chapter~3 ]{HoZe94} we have:}
\begin{proposition}\label{prop:EH.1.4}
If $H\in C^\infty(\mathbb{R}^{2n},\mathbb{R}_{\ge 0})$ satisfies satisfies (H1) and (H2), then
 $c^\Psi_{\rm EH}(H)>0$.
\end{proposition}

\begin{proof}
By the assumption (H1), we can take $z_0\in{\rm int}\,H^{-1}(0)\cap {\rm Fix}(\Psi)$.
Define $\gamma\in\Gamma$ by
$$
\gamma:\mathbb{E}\rightarrow\mathbb{E},\; x\mapsto \gamma(x)=z_0+\varepsilon x
$$
where $\varepsilon>0$ is a constant. Let us prove
$$
\inf_{y\in\gamma(S^+)}\Phi_H(y)>0
$$
for sufficiently small $\varepsilon$ as in the proof of \cite[page 93, Lemma~9]{HoZe90}. Since
\begin{equation}\label{e:EH.2.2}
\Phi_H(z_0+x)=1/2\|x\|^2_{\mathbb{E}}-\int_0^1H(z_0+x)\quad\forall x\in \mathbb{E}^+,
\end{equation}
it suffices to prove that
\begin{equation}\label{e:EH.2.3}
\lim_{\|x\|_{\mathbb{E}}\rightarrow 0}\frac{\int_0^1 H(z_0+x)}{\|x\|^2_{\mathbb{E}}}=0.
\end{equation}
 Otherwise, suppose there exists a sequence $(x_j)\subset\mathbb{E}$ and $d>0$ satisfying
  \begin{equation}\label{e:EH.2.4}
  \|x_j\|_{\mathbb{E}}\rightarrow 0 \quad \hbox{and} \quad \frac{\int_0^1 H(z_0+x_j)}{\|x_j\|^2_{\mathbb{E}}}
  \geq d>0\quad\forall j.
  \end{equation}
  Let $y_j=\frac{x_j}{\|x_j\|_{\mathbb{E}}}$ and hence $\|y_j\|_{\mathbb{E}}=1$. By Proposition
  \ref{compact}, $(y_j)$ has a convergent subsequence in $L^2$.
  By a standard result in $L^p$ theory (see \cite[Th.4.9]{Br11}),
  we have $w\in L^2$ and a subsequence of $(y_j)$, still denoted by $(y_j)$, such that
   $y_j(t)\rightarrow y(t)$ a.e. on $(0,1)$ and that
   $|y_j(t)|\leq w(t)$  a.e. on $(0,1)$ for each $j$.  Recall that we have assumed that $H$ vanishes near $z_0$.
  By (H2) and the Taylor expansion of $H$ at $z_0\in\mathbb{R}^{2n}$,
  we have constants $C_1>0$ and $C_2>0$ such that $H(z_0+z)\le C_1|z|^2$
  and  $H(z_0+z)\le C_2|z|^3$ for all $z\in\mathbb{R}^{2n}$.
  It follows that
  \begin{eqnarray*}
  &&\frac{H(z_0+x_j(t))}{\|x_j\|_{\mathbb{E}}^2}\leq C_1\frac{|x_j(t)|^2}{\|x_j\|_{\mathbb{E}}^2}
  =C_1|y_j(t)|^2\le C_1w(t)^2,\quad \hbox{a.e. on $ (0,1)$},\;\forall j,\\
 &&  \frac{H(z_0+x_j(t))}{\|x_j\|_{\mathbb{E}}^2}\leq C_2\frac{|x_j(t)|^3}{\|x_j\|_{\mathbb{E}}^2}
  =C_2|x_j(t)|\cdot |y_j(t)|^2\le C_2|x_j(t)|w(t)^2,\quad \hbox{a.e. on $(0,1)$},\;\forall j.
   \end{eqnarray*}
  The first claim in (\ref{e:EH.2.4}) implies that $(x_j)$
  has a subsequence such that
  $$
  x_{j_l}(t)\rightarrow 0, \quad \text{a.e. on }(0,1).
  $$
 Hence the Lebesgue dominated convergence theorem leads to
  $$
  \int_0^1 \frac{H(z_0+x_{j_l}(t))}{\|x_{j_l}\|_{\mathbb{E}}^2}\rightarrow 0.
  $$
  This contradicts the second claim in (\ref{e:EH.2.4}).
  \end{proof}

 Propositions~\ref{prop:EH.1.3},~\ref{prop:EH.1.4} show that
 $c^{\Psi}_{\rm EH}(H)$ is a finite positive number for each
$H\in C^\infty(\mathbb{R}^{2n},\mathbb{R}_{\ge 0})$ satisfying (H1) and (H2).
{Based on this fact, the proof of Theorem~\ref{th:EH.1.6} is given by a minmax argument
as in \cite[Section~3.4]{Sik90}. For sake of completeness we give its details here.}

\begin{proof}[\bf Proof of Theorem~\ref{th:EH.1.6}]
  Let us define
   $$
   \mathcal{F}:=\{\gamma(S^+)\,|\,\gamma\in\Gamma
   \;\text{and}\;\inf(\Phi_H|\gamma(S^+))>0\}.
   $$
  Then $c^\Psi_{\rm EH}(H)=\sup_{F\in\mathcal{F}}\inf_{x\in F}\Phi_H(x)$  since $c^\Psi_{\rm EH}(H)>0$. Note that the flow $\phi^u$ of $\nabla \Phi_H$  has the form
  $$
  \phi^u(x)=e^{-u}x^-+x^0+e^ux^++\widetilde{K}(u,x),
  $$
  where $\widetilde{K}:\mathbb{R}\times \mathbb{E}\rightarrow\mathbb{E}$ is compact.
  For a set $F=\gamma(S^+)\in\mathcal{F}$ where $\gamma\in\Gamma$,  $\alpha:=\inf(\Phi_H|\gamma(S^+))>0$ by definition of $\mathcal{F}$. Let $\rho:\mathbb{R}\rightarrow [0,1]$ be a smooth function such that $\rho(s)=0$ for $s\le 0$ and $\rho(s)=1$ for $s\ge \alpha$. Define a vector field $V$ on $\mathbb{E}$ by
  $$
  V(x)=x^+-x^--\rho(\Phi_H(x))\nabla \mathfrak{b}(x).
  $$
 Clearly $V$ is locally Lipschitz and has linear growth so that $V$ has a unique global flow which we will denote by $\gamma^u$. Moreover, it is obvious that $\gamma^u$
   has the same property as $\phi^u$ described above.
   For $x\in \mathbb{E}^-\oplus \mathbb{E}^0$, we have $\Phi_H(x)\le 0$ and
    hence  $V(x)=-x^-$ so that $\gamma^u(\mathbb{E}^-\oplus \mathbb{E}^0)=\mathbb{E}^-\oplus \mathbb{E}^0$ and
    $\gamma^u(\mathbb{E}\setminus(\mathbb{E}^-\oplus\mathbb{E}^0))=\mathbb{E}\setminus (\mathbb{E}^-\oplus \mathbb{E}^0)$,
    since $\gamma^u$ is a homeomorphism for each $u\in\mathbb{R}$. Therefore, $\gamma_u\in\Gamma$ for all $u\in\mathbb{R}$.

    Note that $V|\Phi_H^{-1}([\alpha,\infty])=\nabla \Phi_H(x)$. We have $\gamma^u(F)=\phi^u(F)$ for $u\ge 0$. Since $\Gamma$ is closed
    for composition operation, $\mathcal{F}$ is positively invariant under the flow $\phi^u $ of $\nabla\Phi_H$. Using Proposition~\ref{prop:PSmale} we can prove Theorem~\ref{th:EH.1.6} by a standard minimax argument.
    \end{proof}

\section{Proofs of Theorem~\ref{th:EHconvex}, \ref{th:EHproduct}}\label{sec:EH.3}
\setcounter{equation}{0}

Our proofs closely follow those of Theorems~6.5, 6.6 in \cite{Sik90}.

\subsection{Proof of Theorem~\ref{th:EHconvex}}\label{sec:EH.3.1}

{By assumption $Cl(D)\cap{\rm Fix}(\Psi)\ne\emptyset$.
We only consider the case $D\cap{\rm Fix}(\Psi)\ne\emptyset$ because this may lead to the case $Cl(D)\cap{\rm Fix}(\Psi)\subset\partial D$
by exterior regularity of the $\Psi$-EH capacity and a standard approximation argument as in Section~\ref{sec:convex4}.
Moreover, we can also assume that $D$ contains $0$ in its interior
by a translation argument (see the beginning of Section~\ref{sec:convex}).}

In this section, we denote
$$
a:=\min\Sigma_{\mathcal{S}}^\Psi=\min\{A(x)>0\,|\,x\;\text{is a }\;\Psi\hbox{-characteristic on}\;\mathcal{S}=\mathcal{\partial D}\}.
$$
Let  $j_D: \mathbb{R}^{2n}\rightarrow\mathbb{R} $ be the Minkowski functional of $D$,
$H(z)=j_D^2(z)$ and $H^\ast$ the Legendre transformation of $H$. Define $I$, $\mathcal{F}$ and $\mathcal{A}$ as in Section \ref{sec:convex1}.

\noindent{\bf Step 1} ({\it Prove $c^{\Psi}_{\rm EH}(D)\ge a$}).
\noindent For $\epsilon>0$, let
\begin{eqnarray}\label{e:EH.3.7}
\mathscr{E}_\epsilon(\mathbb{R}^{2n},D)
\end{eqnarray}
consist of $\overline{H}=f\circ H$, where $f\in C^\infty(\mathbb{R},\mathbb{R})$ satisfies
\begin{eqnarray}\label{e:EH.3.8}
f(s)=0\;\forall s\le 1,\quad  f'(s)\ge 0,\;\forall\;s\ge 1,\quad f'(s)=\alpha\in\mathbb{R}\setminus\Sigma^{\Psi}_\mathcal{S}
\;\hbox{if}\;f(s)\ge\epsilon
\end{eqnarray}
and $\alpha$ is required to satisfy
\begin{equation}\label{e:EH.3.8.1}
\alpha H(z)\ge \frac{ \mathfrak{t}(\Psi)}{2}|z|^2-C\quad\hbox{for}\, |z|\,\hbox{sufficiently large}
\end{equation}
where $C>0$ is a constant. Then $\mathscr{E}_\epsilon(\mathbb{R}^{2n},D)$ is a
cofinal family of $\mathcal{F}(\mathbb{R}^{2n},D)$, i.e.
for any $G\in \mathcal{F}(\mathbb{R}^{2n},D)$ there exists $\overline{H}\in\mathscr{E}_\epsilon(\mathbb{R}^{2n},D)$ such that $\overline{H}\ge G$. It follows that $c^{\Psi}_{\rm EH}(G)\ge c^{\Psi}_{\rm EH}(H)$. For each $\overline{H}\in\mathscr{E}_\epsilon(\mathbb{R}^{2n},D)$, by the last condition in
(\ref{e:EH.3.8}) and (\ref{e:EH.3.8.1})  $\Phi_{\overline{H}}$ satisfies the
$(PS)$ condition and $c^{\Psi}_{\rm EH}(\overline{H})$ is a positive critical value of $\Phi_{\overline{H}}$ (see Section~\ref{sec:EH.2}
and \cite[\S5.1]{JinLu}).
Corresponding to \cite[Lemma~3, Section~6.5]{Sik90} we have
\begin{lemma}\label{lem:5.1}
Any positive critical value $c$ of
$\Phi_{\overline{H}}$ satisfies $c>\min\Sigma^{\Psi}_\mathcal{S}-\epsilon$. In particular, $c^\Psi_{\rm EH}(\overline{H})>\min\Sigma^{\Psi}_\mathcal{S}-\epsilon$.
\end{lemma}
  \begin{proof}
    Let $x\in E$ be a critical point of $\Phi_{\overline{H}}$ with $\Phi_{\overline{H}}(x)>0$. Then
$$
-J\dot{x}(t)=\nabla \overline{H}(x(t))=f'(H(x(t)))\nabla H(x(t)),\quad x(1)=\Psi x(0),
$$
and $H(x(t))\equiv s_0$ (a nonzero constant). It follows that
\begin{eqnarray*}
\Phi_{\overline{H}}(x)&=&\frac{1}{2}\int^1_0\langle Jx(t),\dot{x}(t)\rangle dt-\int^1_0H(x(t))dt\\
&=&\frac{1}{2}\int^1_0\langle x(t),f'(s_0)\nabla H(x(t))\rangle dt-\int^1_0f(s)dt\\
&=&f'(s_0)s_0-f(s_0).
\end{eqnarray*}
Since $\Phi_{\overline{H}}(x)>0$, we get $\beta:=f'(s_0)>0$ and so $s_0>1$. Put
$$
y(t)=\frac{1}{\sqrt{s_0}}x(t/\beta).
$$
Then $H(y(t))=1$, $-J\dot{y}=\nabla H(y(t))$ and $y(\beta)=\Psi y(0)$.
These show that $f'(s_0)=\beta=A(y)\in\Sigma^{\Psi}_\mathcal{S}$.
Therefore $f(s_0)<\epsilon$ by definition of $\overline{H}$. It follows from these that
$$
\Phi_{\overline{H}}(x)=f'(s_0)s_0-f(s_0)> f'(s_0)-\epsilon\ge \min\Sigma^{\Psi}_\mathcal{S}-\epsilon.
$$
\end{proof}
By definition of the $\Psi$-EH capacity,
\begin{eqnarray*}
c_{\rm EH}^\Psi(D)&=&\inf\{c_{\rm EH}^\Psi(G)\,|\,G\in\mathcal{F}(\mathbb{R}^{2n},D)\}\\
&\ge&\inf\{c_{\rm EH}^\Psi(\overline{H})\,|\,\overline{H}\in\mathscr{E}_\epsilon(\mathbb{R}^{2n},D)
\}\\
&\ge& \min\Sigma^{\Psi}_\mathcal{S}-\epsilon=a-\epsilon, \forall \epsilon
\end{eqnarray*}
and hence $c_{\rm EH}^\Psi(D)\ge a$.

\noindent{\bf Step 2} ({\it Prove $c^{\Psi}_{\rm EH}(D)\le a$}).
{The proof of \cite[Theorem~6.5]{Sik90} can be carried here verbatim. However, we give the details for the convenience of readers.}
We only need to prove that for each $\varepsilon>0$, there exists $\tilde{H}\in  \mathscr{F}(\mathbb{R}^{2n}, D)$
such that
\begin{eqnarray}\label{e:EH.3.9}
c^\Psi_{\rm EH}(\tilde{H})< a+\varepsilon
\end{eqnarray}
which is reduced to prove: for any $h\in\Gamma$, there exists $x\in h(S^+)$ such that
\begin{eqnarray}\label{e:EH.3.10}
\Phi_{\tilde{H}}(x)< a+\varepsilon.
\end{eqnarray}
For $\tau$ sufficiently large, choose $H_\tau\in\mathscr{F}(\mathbb{R}^{2n}, D)$ satisfying
\begin{eqnarray}\label{e:EH.3.11}
H_\tau\ge \tau\left(H-\left(1+\frac{\varepsilon}{2a}\right)\right).
\end{eqnarray}
For $h\in\Gamma$, in order to choose $x\in h(S^+)$ satisfying (\ref{e:EH.3.10}) for $\tilde{H}=H_\tau$,
 we make some preparations as in \cite[Lemma~1, Section~6.5]{Sik90}. By arguments in Section~
\ref{sec:convex1}, there exists $w\in\mathcal{A}$  such that
$$
a:=\min\{I(u)\,|\,u\in\mathcal{A}\}=I(w)=A(x^{\ast})\quad\hbox{and}\quad A(w)=1.
$$
 Denote by $w^\ast$ the projections of $w$ onto $\mathbb{E}^\ast$ (according to the decomposition $\mathbb{E} = E^{1/2}
=\mathbb{E}^+\oplus \mathbb{E}^-\oplus \mathbb{E}^0$), $\ast=0,-,+$.  Then $w^+\ne 0$.
 (Otherwise, a contradiction occurs because
$1=A(w) = A(w^0\oplus w^-) =-\frac{1}{2}\|w^-\|^2_{\mathbb{E}}$.)
Put $y=w/\sqrt{a}$ so that $I(y)=1$ and $A(y)=\frac{1}{a}$.
Now for any $\lambda\in\mathbb{R}$ and $x\in \mathbb{E}$ it holds that
\begin{eqnarray*}
\lambda^2=I(\lambda y)&=&\int^1_0H^\ast(-\lambda J\dot{y}(t))dt\\
&=&\int^1_0\sup_{\zeta\in\mathbb{R}^{2n}}\{\langle\zeta, -\lambda J\dot{y}(t)\rangle- H(\zeta)\}dt\\
&\ge&\int^1_0\{\langle x(t), -\lambda J\dot{y}(t)\rangle- H(x(t))\}dt.
\end{eqnarray*}
This leads to
\begin{eqnarray*}
\int^1_0H(x(t))dt\ge\int^1_0\langle x(t), -\lambda J\dot{y}(t)\rangle dt-\lambda^2
=\lambda \int^1_0\langle x(t), - J\dot{y}(t)\rangle dt-\lambda^2.
\end{eqnarray*}
Taking
$$
\lambda=\frac{1}{2} \int^1_0\langle x(t), - J\dot{y}(t)\rangle dt
$$
we arrive at
\begin{eqnarray}\label{e:EH.3.1}
\int^1_0H(x(t))dt\ge\left(\frac{1}{2} \int^1_0\langle x(t), - J\dot{y}(t)\rangle dt\right)^2,\quad
\forall x\in \mathbb{E}.
\end{eqnarray}
Note that $y^+\ne 0$ and $\mathbb{E}^-\oplus \mathbb{E}^0+\mathbb{R}_+y=\mathbb{E}^-\oplus \mathbb{E}^0\oplus(\mathbb{R}_+y^+)$.
By the intersection property (ii) in Proposition~\ref{prop:EH.1.1} we derive
$$
h(S^+)\cap(\mathbb{E}^-\oplus \mathbb{E}^0+\mathbb{R}_+y)\ne\emptyset,\;\forall h\in\Gamma.
$$
For an $h\in\Gamma$ and $x\in h(S^+)\cap(\mathbb{E}^-\oplus \mathbb{E}^0+\mathbb{R}_+y)$,
consider the polynomial
$$
P(t)=\mathfrak{a}(x+ty)=\mathfrak{a}(x)+ t(\int^1_0\langle x, - J\dot{y}\rangle dt)+ \mathfrak{a}(y)t^2.
$$
  Writing $x=x^{-0}+ sy=x^{-0}+ sy^{-0}+ sy^+$ where $x^{-0}\in \mathbb{E}^-\oplus \mathbb{E}^0$,
  then $P(t)=\mathfrak{a}(x^{-0}+(t+s)y)$. Since $\mathfrak{a}|_{\mathbb{E}^-\oplus \mathbb{E}^0}\le 0$
we deduce that $P(-s)\le 0$. Moreover, $\mathfrak{a}(y)=1/a>0$ and we get
$$
P(t)\to+\infty\quad\hbox{as}\quad |t|\to+\infty.
$$
These imply that there exists $t_0\in\mathbb{R}$ such that $P(t_0)=0$. It follows that
$$
\left(\int^1_0\langle x, - J\dot{y}\rangle dt\right)^2-4\mathfrak{a}(y)\mathfrak{a}(x)\ge 0
$$
and so by (\ref{e:EH.3.1}) there holds
\begin{eqnarray}\label{e:EH.3.2}
\mathfrak{a}(x)&\le& (\mathfrak{a}(y))^{-1}\left(\frac{1}{2}\int^1_0\langle x, - J\dot{y}\rangle  dt\right)^2\nonumber\\
&=&a\left(\frac{1}{2}\int^1_0\langle x, - J\dot{y}\rangle dt \right)^2\nonumber\\
&\le&a\int^1_0H(x(t))dt.
\end{eqnarray}

\begin{proof}[\bf Proof of the fact $c^{\Psi}_{\rm EH}(D)\le a$.]
For any $\varepsilon>0$,
let $H_\tau$ be a function in $\mathscr{F}(\mathbb{R}^{2n}, D)$ satisfying (\ref{e:EH.3.11}). For any $h\in\Gamma$, let $x\in h(S^+)\cap(\mathbb{E}^-\oplus \mathbb{E}^0+\mathbb{R}_+y)\ne\emptyset$. \\
$\bullet$ If $\int^1_0H(x(t))dt\le\left(1+\frac{\varepsilon}{a}\right)$, then by $H_\tau\ge 0$ and (\ref{e:EH.3.2}), we have
$$
\Phi_{H_\tau}(x)\le \mathfrak{a}(x)\le a\int^1_0H(x(t))dt\le a\left(1+\frac{\varepsilon}{a}\right)<a+\varepsilon.
$$
$\bullet$ If $\int^1_0H(x(t))dt>\left(1+\frac{\varepsilon}{a}\right)$ then (\ref{e:EH.3.11}) implies
\begin{eqnarray}\label{e:EH.3.12}
\int^1_0H_\tau(x(t))dt&\ge& \tau\left(\int^1_0H(x(t))dt-\left(1+\frac{\varepsilon}{2a}\right)\right)\nonumber\\
&\ge&\tau \frac{\varepsilon}{2(a+\varepsilon)}\int^1_0H(x(t))dt.
\end{eqnarray}
Choose $\tau>0$ so large that
$$
{\tau\frac{\varepsilon}{2(a+\varepsilon)}>a.}
$$
Then (\ref{e:EH.3.12}) leads to
$$
\int^1_0H_\tau(x(t))dt\ge a\int^1_0H(x(t))dt
$$
and hence by (\ref{e:EH.3.2}) there holds
$$
\Phi_{H_\tau}(x)=\mathfrak{a}(x)-\int^1_0H_\tau(x(t))dt\le \mathfrak{a}(x)-a\int^1_0H(x(t))dt\le 0.
$$
In summary, in two cases we have $\Phi_{H_\tau}(x)<a+\varepsilon$ and hence $c^{\Psi}_{\rm EH}(D)\le a$.
\end{proof}

\noindent{\bf Step 3.}\quad {\it Prove $c_{\rm EH}^\Psi(\partial D)$=a.}
 Let
\begin{eqnarray}\label{e:EH.3.3}
\mathscr{E}_\epsilon(\mathbb{R}^{2n}, \partial D)
\end{eqnarray}
consist of $\overline{H}=f\circ H$ where $f\in C^\infty(\mathbb{R},\mathbb{R})$ satisfies
\begin{eqnarray}\label{e:EH.3.4}
&&f(s)=0\;\hbox{for}\; s\;\hbox{near}\;1,\quad f'(s)\le 0\;\forall s\le 1,
\quad f'(s)\ge 0\;\forall s\ge 1,\\
&&f'(s)=\alpha\in\mathbb{R}\setminus\Sigma^{\Psi}_\mathcal{S}
\;\hbox{if}\;s\ge 1\;\hbox{and}\;f(s)\ge\epsilon,\label{e:EH.3.5}
\end{eqnarray}
where $\alpha$ is also required to be so large that
\begin{eqnarray}\label{e:EH.3.6}
\alpha H(z)\ge \frac{ \mathfrak{t}(\Psi)}{2}|z|^2-C
\end{eqnarray}
for some constant $C>0$. Similar to Step~1, there holds that
$$
c^{\Psi}_{\rm EH}(\overline{H})>\min \Sigma^\Psi_{\partial D}-\epsilon,\quad\forall \overline{H}\in \mathscr{E}_\epsilon(\mathbb{R}^{2n},\partial D).
$$
It follows that $c^{\Psi}_{\rm EH}(\partial D)\ge a$. On the other hand, by the monotonicity of
$c^{\Psi}_{\rm EH}$ we get that $c^{\Psi}_{\rm EH}(\partial D)\le c^{\Psi}_{\rm EH}(D) =a$. Therefore $c^{\Psi}_{\rm EH}(\partial D)=a$.
\hfill$\Box$\vspace{2mm}

\subsection{Proof of Theorem~\ref{th:EHproduct}}\label{sec:EH.3.2}

{The following lemma is a slight change of \cite[Lemma~1,~Section~6.6]{Sik90}.}
\begin{lemma}\label{lem:9.1}
For a convex domain $D\subset\mathbb{R}^{2n}$ containing $0$ and  symplectic matrixes
$\Psi_1\in{\rm Sp}(2n,\mathbb{R})$ and  $\Psi_2\in{\rm Sp}(2k,\mathbb{R})$, it holds that
$c_{\rm EH}^{\Psi_1\oplus\Psi_2}(D\times\mathbb{R}^{2k})=c_{\rm EH}^{\Psi_1}(D)$.
\end{lemma}
\begin{proof}
It suffices to prove this lemma for a  convex bounded domain
 $D\subset\mathbb{R}^{2n}$ with $C^2$-smooth boundary $\mathcal{S}$. Let $H=j_D^2$.
 By the definition and monotonicity of $\Psi$-EH capacity we have
 $$
 c_{\rm EH}^{\Psi_1\oplus\Psi_2}(D\times\mathbb{R}^{2k})=\sup_R c^{\Psi_1\oplus\Psi_2}_{\rm EH}(E_R),\quad
 \hbox{where}\; E_R=\{(z,z')\in\mathbb{R}^{2n}\times\mathbb{R}^{2k}\,|\,H(z)+ (|z'|/R)^2<1\}.
 $$
 Since $E_R$ is convex
 and $\mathcal{S}_R:=\partial E_R$ is of class $C^{1,1}$ because $H$ is of class $C^{1,1}$ on
 $\mathbb{R}^{2n}$, Theorem~\ref{th:EHconvex} gives rise to
 $$
 c^{\Psi_1\oplus\Psi_2}_{\rm EH}(E_R)=\min\Sigma^{\Psi_1\oplus\Psi_2}_{\mathcal{S}_R}.
 $$
 Let   $(x,\hat{x}):[0,\lambda]\rightarrow \mathcal{S}_R$ with $\lambda>0$ satisfy
\begin{eqnarray}
&&\dot{x}=X_H(x)\quad\hbox{and}\quad\quad x(\lambda)=\Psi_1 x(0),\label{eq:1}\\
&&\dot{\hat{x}}=2J\hat{x}/R^2\quad\hbox{and}\quad\quad \hat{x}(\lambda)=\Psi_2 \hat{x}(0).\label{eq:2}
\end{eqnarray}
{Then $(x,\hat{x})$ is a $\Psi_1\oplus\Psi_2$-characteristic on $\mathcal{S}_R$ with action $\lambda$.
By (\ref{eq:2}), $|\hat{x}|$ is constant. If $\hat{x}\neq 0$, then by the boundary condition in (\ref{eq:2})
we get $\lambda\ge R^2\mathfrak{t}(\Psi_2)/2$, where $\mathfrak{t}(\Psi_2)$
is the smallest positive number satisfying $\det(\Psi_2-e^{sJ})=0$ (see (\ref{e:TPsi})).
If $|\hat{x}|\equiv0$, then $x$ lies on
$\mathcal{S}$ and $\lambda\in\Sigma^{\Psi_1}_{\mathcal{S}}$. Hence}
for $R>0$ large enough we arrive at
$$
c^{\Psi_1\oplus\Psi_2}_{\rm EH}(E_R)=\min\Sigma^{\Psi_1\oplus\Psi_2}_{S_R}=\min\Sigma^{\Psi_1}_\mathcal{S}=c^{\Psi_1}_{\rm EH}(D)
$$
and so the desired conclusion.
\end{proof}

{Based on the proof of Proposition~\ref{prop:EH.1.4} and Lemma~\ref{lem:5.1}, \cite[Lemma~2, Section~6.6]{Sik90}
can be generalized to the following (cf. \cite[Lemma~5.7]{JinLu}).}
\begin{lemma}\label{lem:9.2}
Let $D\subset\mathbb{R}^{2n}$ be a convex bounded domain
  with $C^2$-smooth boundary $\mathcal{S}$ and containing $0$. Let $\widetilde{H}\in\mathscr{F}(\mathbb{R}^{2n},D)$.
  Then for any $\epsilon>0$ there exists $\gamma\in\Gamma$ such that
   \begin{equation}\label{e:EH.3.14}
      \Phi_{\widetilde{H}}|\gamma(B^+\setminus\epsilon B^+)\ge c^{\Psi}_{\rm EH}(D)-\epsilon
      \quad\hbox{and}\quad \Phi_{\widetilde{H}}|\gamma(B^+)\ge 0,
   \end{equation}
    where $B^+$ is the closed unit ball in $E^+$ and $S^+=\partial B^+$.
\end{lemma}
{Then Theorem~\ref{th:EHproduct} follows from Lemma~\ref{lem:9.1} and Lemma~\ref{lem:5.1}.}

 \begin{proof}[\bf Proof of Theorem~\ref{th:EHproduct}]
\noindent{\bf Step 1} ({\it Prove} (\ref{e:product1})).\quad
By the approximation arguments
 we may assume that each $D_i\subset\mathbb{R}^{2n_i}$ is a $C^2$
 convex bounded domain containing  a fixed point $p_i$ of
 $\Psi_i$,  $1\le i\le k$.
 Since $c^{\Psi_i}_{\rm EH}(D_i-p_i)=c^{\Psi_i}_{\rm EH}(D_i)$, $1\le i\le k$, and
 $$
 c^{\Psi}_{\rm EH}((D_1-p_1)\times\cdots\times (D_k-p_k))=c^{\Psi}_{\rm EH}(D_1\times\cdots\times D_k),
 $$
 we may also assume that each $D_i$ contains the origin of $\mathbb{R}^{2n_i}$, $1\le i\le k$.
 Thus it follows from  the monotonicity of $\Psi$-EH capacity and Lemma~\ref{lem:9.1} that
 \begin{equation}\label{e:EH.3.13}
 c^{\Psi_1\oplus\cdots\oplus\Psi_k}_{\rm EH}(D_1\times\cdots\times D_k)\le \min_ic^{\Psi_i\oplus(\oplus_{j\neq i} \Psi_j)}_{\rm EH}(D_i\times\mathbb{R}^{2(n-n_i)})=\min_ic^{\Psi_i}_{\rm EH}(D_i).
 \end{equation}

In order to prove the converse inequality,
note that for each ${H}\in\mathscr{F}(\mathbb{R}^{2n}, D_1\times\cdots\times  D_k)$
we may choose $\widehat{H}_i\in\mathscr{F}(\mathbb{R}^{2n_i}, D_i)$, $i=1,\cdots,k$, such that
$$
\widehat{H}(z):=\sum \widehat{H}_i(z_i)\ge H(z)\quad\forall z.
$$
For each $i=1,\cdots,k$, by Lemma~\ref{lem:9.2} there exists  $\gamma_i\in\Gamma(\mathbb{R}^{2n_i})$
such that
$$
\Phi_{\widehat{H}_i}|\gamma_i(B_i^+\setminus(2k)^{-1}B^+_i)\ge c^{\Psi_i}_{EH}(D_i)-\epsilon,\qquad
\Phi_{\widehat{H}_i}|\gamma_i(B_i^+)\ge 0.
$$
For any $x=(x_1,\cdots,x_k)\in S^+\subset B^+_1\times\cdots\times B^+_k$ there exists some
$i_0$ such that
$$
x_{i_0}\in B_{i_0}^+\setminus(2k)^{-1}B_{i_0}^+.
$$
 Put $\gamma=\gamma_1\times\cdots\times\gamma_k$ and we arrive at
$$
\Phi_{\widehat{H}}(\gamma(x))=\sum\Phi_{\widehat{H}_i}(\gamma_i(x_i))\ge c^{\Psi_{i_0}}_{\rm EH}(D_{i_0})-\epsilon
\ge\min_i(c^{\Psi_i}_{\rm EH}(D_i)-\epsilon)
$$
and hence
$$
c^{\Psi}_{\rm EH}({H})\ge
c^{\Psi}_{\rm EH}(\widehat{H})=\sup_{h\in\Gamma}\inf_{x\in h(S^+)}\Phi_{\widehat{H}}(x)\ge
\min_ic^{\Psi_i}_{EH}(D_i)-\epsilon.
$$
 This leads to
\begin{equation}\label{e:EH.3.18}
c^{\Psi_1\oplus\cdots\oplus\Psi_k}_{\rm EH}(D_1\times\cdots\times D _k)\ge\min_i c^{\Psi_i}_{\rm EH}(D_i)
 \end{equation}
and by combining this  with (\ref{e:EH.3.13}) we get (\ref{e:product1}) .

\noindent{\bf Step 2} ({\it Prove} (\ref{e:product2})).\quad
As in Step 1, for each $i=1,\cdots,k$ we may assume: (i)
 $D_i\subset\mathbb{R}^{2n_i}$ is compact,
 convex, and has $C^2$-boundary; (ii) $\partial D_i$ contains a fixed point $p_i$ of
 $\Psi_i$ and ${\rm Int}(D_i)$ contains the origin of $\mathbb{R}^{2n_i}$.
Then Lemma~\ref{lem:9.2} holds for every $\widetilde{H}\in\mathcal{F}(\mathbb{R}^{2n},\mathcal{S})$
with $\mathcal{S}=\partial D_1\times\cdots\times \partial D_k$. Arguing as in Step 1 we get that
$$
c^{\Psi_1\oplus\cdots\oplus\Psi_k}_{EH}(\partial D_1\times\cdots\times \partial D_k)\ge\min_i c^{\Psi_i}_{EH}(D_i).
$$
Since $c^{\Psi_1\oplus\cdots\oplus\Psi_k}_{\rm EH}(\partial D_1\times\cdots\times \partial D_k)
\le  c^{\Psi}_{\rm EH}(D_1\times\cdots\times D_k)=\min_ic^{\Psi_i}_{\rm EH}(D_i)$ by
the monotonicity property of $c^{\Psi}_{\rm EH}$ and (\ref{e:product1}),
we obtain (\ref{e:product2}).
\end{proof}

\section{Proof of Theorem~\ref{th:EHcontact}}\label{sec:EH.4}
\setcounter{equation}{0}

The proof of \cite[Th.7.5.1]{Sik90} is different from
that of \cite[Prop.6]{EH89}.  The former can be  adapted to complete
the proof for Theorem~\ref{th:EHcontact} conveniently.

  Let $\lambda_0=\frac{1}{2}(qdp-pdq)$, where $(q, p)$ are the standard
  coordinates on $\mathbb{R}^{2n}$. Then $d\lambda_0=dq\wedge dp=\omega_0$ and
  for any $C^1$ path $x:[0,T]\to \mathbb{R}^{2n}$
  \begin{equation}\label{e:EH.4.1}
  A(x)=\frac{1}{2}\int_0^T \langle-J\dot{x},x\rangle dt=\int_{x}x^{\ast} \lambda_0.
  \end{equation}
  If $x$ is not closed, $\lambda_0$ can not be replaced by other primitives of $\omega_0$  in general.

  \begin{lemma}\label{lem:EH.4.1}
  For $\Psi\in{\rm} Sp(2n,\mathbb{R})$, let $X$ be a vector field defined on $\mathbb{R}^{2n}$ such that
  \begin{equation}\label{psi-inv}
  X(\Psi (z))=\Psi (X(z)),\quad \forall z\in\mathbb{R}^{2n}
  \end{equation}
   and
  suppose that $\lambda:=\imath_X\omega_0$ is a primitive of $\omega_0$.
  Let $x:[0,T]\to\mathbb{R}^{2n}$ satisfy $x(T)=\Psi x(0)$.
 Then
  \begin{equation}\label{e:EH.4.2}
  \int_xx^{\ast}\lambda_0=\int_xx^{\ast}\lambda.
  \end{equation}
    \end{lemma}

  \begin{proof}
  Let $X_0$ be the vector field on $\mathbb{R}^{2n}$ defined by
  $X_0(z)=\frac{1}{2}z$, $\forall z\in\mathbb{R}^{2n}$. Then
  we have
  $$
  \lambda_0=\imath_{X_0}\omega_0,\quad\hbox{and}\quad X_0(\Psi(z))=\Psi X_0(z),\,\forall z\in\mathbb{R}^{2n}.
  $$
  For a vector $Y\in T_z\mathbb{R}^{2n}=\mathbb{R}^{2n}$, we compute
  \begin{eqnarray*}
    \Psi^{\ast}\lambda_0(z)[Y]&=&\lambda_0(\Psi z)[\Psi Y]
                              =\omega_0(X_0(\Psi z),\Psi Y)
                              =\omega_0(\Psi X_0(z),\Psi Y)\\
                              &=&\Psi^{\ast}\omega_0(X_0(z),Y)
                              =\omega_0(X_0(z),Y)
                              =\imath_{X_0}\omega_0(z)[Y]
                              =\lambda_0(z)[Y].
  \end{eqnarray*}
    Hence
  \begin{equation}\label{e:EH.4.3}
  \Psi^{\ast}\lambda_0=\lambda_0.
  \end{equation}
  The same arguments lead to
  \begin{equation}\label{e:EH.4.4}
  \Psi^{\ast}\lambda=\lambda.
  \end{equation}
  Since $\lambda_0$ and $\lambda$ are primitives of $\omega_0$,
  $d(\lambda_0-\lambda)=0$ and there exists $F\in C^\infty(\mathbb{R}^{2n},\mathbb{R})$ such that $\lambda_0-\lambda=dF$.
  (\ref{e:EH.4.3}) and (\ref{e:EH.4.4}) lead to
  $\Psi^{\ast}dF=\Psi^{\ast}(\lambda_0-\lambda)=\lambda_0-\lambda=dF$,
    which implies that there exists a constant $C$ such that
  $F(\Psi(z))-F(z)=C$ for all $z\in\mathbb{R}^{2n}$.
    Since $\Psi(0)=0$,
    $F(\Psi(0))-F(0)=0$ and we get that $C=0$.
  Therefore
  $$
  \int_xx^{\ast}\lambda_0-\int_xx^{\ast}\lambda=\int_x x^\ast (dF)=F(x(1))-F(x(0))=
  F(\Psi x(0))-F(x(0))=0.
  $$
\end{proof}
Note that the Liouville vector field $X$ in Theorem~\ref{th:EHcontact} satisfies the condition in Lemma~\ref{lem:EH.4.1}.
 In the following part of this section, $X$ denotes
 the Liouville vector field in Theorem~\ref{th:EHcontact} and $\lambda=\imath_X\omega_0$.
Let $\phi^t$ be the local flow of $X$. For $\varepsilon>0$ sufficiently small, the map
 \begin{eqnarray}\label{e:EH.4.5}
\psi:(-\varepsilon, \varepsilon)\times\mathcal{S}\to \mathbb{R}^{2n},\; (s, z)\mapsto \phi^s(z),
\end{eqnarray}
is well defined and $\mathbb{R}^{2n}\setminus \cup_{t\in (-\varepsilon, \varepsilon)}\phi^t(\mathcal{S})$
 has two components. Moreover, since $X$ satisfies (\ref{psi-inv}), there holds
\begin{equation}\label{e:EH.4.6}
\Psi(\phi^t(z))=\phi^t(\Psi z),\;\forall (t,z)\in (-\varepsilon, \varepsilon)\times\mathcal{S}.
\end{equation}
Define $U:=\cup_{t\in (-\varepsilon, \varepsilon)}\phi^t(\mathcal{S})$ and
\begin{eqnarray}\label{e:EH.4.7}
K_\psi:U\to\mathbb{R},\;w\mapsto \tau
\end{eqnarray}
if $w=\phi^\tau(z)\in U$ where $z\in\mathcal{S}$. Let $X_{K_\psi}$ be the
Hamiltonian vector field associated to $K_\psi$ defined by $\omega_0(\cdot,X_{K_\psi})=dK_\psi$.
Then for  $w=\phi^\tau(z)\in U$ it holds that
  \begin{equation}\label{e:EH.4.8}
 \lambda_w(X_{K_\psi})=(\omega_0)_w(X(w), X_{K_\psi}(w))=1
 \end{equation}
and
\begin{equation}\label{e:EH.4.9}
 X_{K_\psi}(\phi^\tau(z))=e^{-\tau}
  d\phi^\tau(z)[X_{K_\psi}(z)],  \quad\forall (\tau,z)\in(-\varepsilon, \varepsilon)\times\mathcal{S}.
\end{equation}
Clearly (\ref{e:EH.4.6}) and (\ref{e:EH.4.9}) show that $y:[0,T]\to \mathcal{S}_\tau=\phi^\tau(\mathcal{S})$
   satisfies
   $$
   \dot{y}(t)=X_{K_\psi}(y(t))\quad \hbox{and} \quad y(T)=\Psi y(0)
   $$
    if and only if
    $y(t)=\psi(\tau, x(e^{-\tau} t))$, where $x:[0, e^{-\tau} T]\to\mathcal{S}$
    satisfies
    $$
    \dot{x}(t)=X_{K_\psi}(x(t))\quad \hbox{and} \quad x(e^{-\tau} T)=\Psi x(0).
    $$
    In addition, there holds
    $$\int y^\ast\lambda =e^\tau\int x^\ast\lambda .$$

The following proposition is a generalization of Lemma~\ref{nointerior}. Its proof is a slight change of the proof of \cite[Proposition~1,~Section~7.4]{Sik90}.
\begin{proposition}\label{prop:EH.4.1}

Let  $\mathcal{S}\subset(\mathbb{R}^{2n},\omega_0)$ be as in Theorem~\ref{th:EHcontact}.
Then the interior of
 $$
  \Sigma^{\Psi}_{\mathcal{S}}=\{A(x)>0\,|\,x\;\text{is a }\;\Psi\hbox{-characteristic on}\;\mathcal{S} \}
 $$
is empty.
\end{proposition}
\begin{proof}
{Assume that the neighborhood $U$ of $\mathcal{S}$ defined above (\ref{e:EH.4.7}) is contained in a ball $B^{2n}(0,R)$.}
Fix $0<\delta<\varepsilon$. Let ${\bf A}_\delta$ and ${\bf B}_\delta$
 denote the unbounded and bounded components of
  $\mathbb{R}^{2n}\setminus \cup_{t\in (-\delta, \delta)}\phi^t(\mathcal{S})$
  respectively.   Then
$$
\phi^\tau(\mathcal{S})\subset {\bf B}_\delta\quad\hbox{for}\quad -\varepsilon<\tau<-\delta.
$$
Define the following $H\in\mathscr{F}(\mathbb{R}^{2n})$ by
\begin{equation}\label{e:EH.4.10}
H(x)=\left\{\begin{array}{ll}
C_0\ge 0 &{\rm if}\;x\in {\bf B}_\delta,\\
f(\tau) &{\rm if}\;x=\phi^\tau(y),\;y\in\mathcal{S},\;\tau\in [-\delta,\delta],\\
C_1 &{\rm if}\;x\in {\bf A}_\delta\cap B^{2n}(0,R),\\
h(|x|^2) &{\rm if}\;x\in {\bf A}_\delta\setminus B^{2n}(0,R)
\end{array}\right.
\end{equation}
where $f:(-\varepsilon,\varepsilon)\to\mathbb{R}$ and $h:[0, \infty)\to\mathbb{R}$ are smooth functions satisfying
\begin{eqnarray}
&&f|(-\varepsilon,-\delta]=C_0,\quad f|[\delta,\varepsilon)=C_1,\label{e:EH.4.11}\\
&&sh'(s)-h(s)\le 0\quad\forall s.\label{e:EH.4.12}
\end{eqnarray}
For a fixed $T\in\Sigma^{\Psi}_{\mathcal{S}}$,
 we choose the above function $f$ such that for some constants $0<\epsilon_1<\delta$ and  $\overline{C}$,
\begin{eqnarray}\label{e:EH.4.13}
  f(u)=Tu+\overline{C}\ge 0,\quad\forall u\in [-\epsilon_1, \epsilon_1].
  \end{eqnarray}
     Clearly
  $O:=\{e^{-\tau}T\,|\,\tau\in (-\varepsilon_1,\varepsilon_1)\}$ is an open neighborhood of $T$.
  If $e^{-\tau}T\in\Sigma_{\mathcal{S}}^\Psi$ then there exists a $\Psi$-characteristic on $\mathcal{S}$
  which can be parameterized as $x:[0,1]\to\mathbb{R}^{2n}$ satisfying
  $$
  \dot{x}=e^{-\tau}X_H(x),\quad x(1)=\Psi x(0).
  $$
  Let $y=\phi^\tau(x)$. Then
  $$
  \dot{y}=X_H(y), \quad y(1)=\Psi y(0),
  $$
  and it follows that $\Phi_H(y)=T-f(\tau)=T-T\tau-\overline{C}$ is a critical value of $\Phi_{H}$.
  Hence if
  $O\cap \Sigma^{\Psi}_{\mathcal{S}}$ has an interior point, then
  $$
  \left\{\tau\in (-\varepsilon_1,\varepsilon_1)\,|\,e^{-\tau}T\in \Sigma^{\Psi}_{\mathcal{S}}\right\}\subset \left\{\tau\in (-\varepsilon_1,\varepsilon_1)\,|\,T-f(\tau)\;\hbox{is a critical value of }\; A_H\right\}
  $$
has nonempty interior and therefore the critical value set of $A_H$ has nonempty interior.
{However, arguing as in Section~\ref{sec:convex2}, the critical value set of $A_H$ has empty interior since $H$ is a smooth function on $\mathbb{R}^{2n}$.}
This is a contradiction. Hence  $\Sigma^{\Psi}_{\mathcal{S}}$ has empty interior.
\end{proof}

\begin{proof}[\bf Proof of Theorem~\ref{th:EHcontact}]
{The proof of \cite[Theorem~7.5]{Sik90} can be carried here almost verbatim. We give a sketch here.}
For $C>0$ large enough and $\delta>2\eta>0$ small enough, define
$H=H_{C,\eta}\in\mathscr{F}(\mathbb{R}^{2n})$
   adapted to $\psi$ as the following:
\begin{equation}\label{e:EH.4.14}
H_{C,\eta}(x)=\left\{\begin{array}{ll}
C\ge 0 &{\rm if}\;x\in {\bf B}_\delta,\\
f_{C,\eta}(\tau) &{\rm if}\;x=\psi(\tau,y),\;y\in\mathcal{S},\;\tau\in [-\delta,\delta],\\
C&{\rm if}\;x\in {\bf A}_\delta\cap B^{2n}(0,R),\\
h(|x|^2) &{\rm if}\;x\in {\bf A}_\delta\setminus B^{2n}(0,R)
\end{array}\right.
\end{equation}
where $B^{2n}(0,R)\supseteq\overline{\psi((-\varepsilon,\varepsilon)\times\mathcal{S})}$
(the closure of $\psi((-\varepsilon,\varepsilon)\times\mathcal{S})$),
$f_{C,\eta}:(-\varepsilon, \varepsilon)\to\mathbb{R}$ and $h:[0, \infty)\to\mathbb{R}$ are smooth functions satisfying
\begin{eqnarray*}
&&f_{C,\eta}|[-\eta,\eta]\equiv 0,\quad f_{C,\eta}(s)=C\;\hbox{if}\;|s|\ge 2\eta,\\
&&f'_{C,\eta}(s)s>0\quad\hbox{if}\quad \eta<|s|<2\eta,\\
&&f'_{C,\eta}(s)-f_{C,\eta}(s)>c^\Psi_{\rm EH}(\mathcal{S})+1\quad\hbox{if}\;s>0\;\hbox{and}\;\eta<f_{C,\eta}(s)<C-\eta,\\
&&h_{C,\eta}(s)=a_Hs+b\quad\hbox{for $s>0$ large enough}, a_H=C/R^2> \mathfrak{t}(\Psi),\\
&&sh'_{C,\eta}(s)-h_{C,\eta}(s)\le 0\quad\forall s\ge 0.
\end{eqnarray*}
Moreover, we assume
\begin{equation}\label{e:EH.4.15}
 \det\left(\exp\left(\frac{2C}{R^2}J\right)-\Psi\right)\neq 0.
\end{equation}
Such a family $H_{C,\eta}$ ($C\to+\infty$, $\eta\to 0$) can be chosen to be cofinal in
the set $\mathscr{F}(\mathbb{R}^{2n},\mathcal{S})$ and
also have the property that
\begin{equation}\label{e:EH.4.14.1}
C\le C'\Rightarrow H_{C,\eta}\le H_{C',\eta},\qquad \eta\le \eta'\Rightarrow H_{C,\eta}\ge H_{C,\eta'}.
\end{equation}
It follows that
$$
c^\Psi_{\rm EH}(\mathcal{S})=\lim_{\eta\to 0\&C\to+\infty}c^\Psi_{\rm EH}(H_{C,\eta}).
$$
By Proposition~\ref{prop:EH.1.2}(i) and (\ref{e:EH.4.14.1}), $\eta\le \eta'\Rightarrow c^\Psi_{\rm EH}(H_{C,\eta})\le c^\Psi_{\rm EH}(H_{C,\eta'})$, and hence
\begin{equation}\label{e:EH.4.14.2}
\Upsilon(C):=\lim_{\eta\to 0}c^\Psi_{\rm EH}(H_{C,\eta})
\end{equation}
exists and
$$
\Upsilon(C)=\lim_{\eta\to 0}c^\Psi_{\rm EH}(H_{C,\eta})\ge \lim_{\eta\to 0}c^\Psi_{\rm EH}(H_{C',\eta})=\Upsilon(C'),
$$
i.e.,  $C\mapsto\Upsilon(C)$ is non-increasing.
Then it is obvious that
\begin{equation}\label{e:EH.4.14.3}
 c_{\rm EH}^\Psi(\mathcal{S})=\lim_{C\to+\infty}\Upsilon(C).
\end{equation}
By the construction of $H_{C,\eta}$, $c^\Psi_{\rm EH}(H_{C,\eta})$ is  a positive critical value of
$\Phi_{H_{C,\eta}}$ and the associated
critical point $x\in \mathbb{E}$ gives rise to  a nonconstant $\Psi$-characteristic
 sitting in the interior of $U_\delta:= \cup_{t\in (-\delta, \delta)}\phi^t(\mathcal{S})$. It follows that
$$
c^\Psi_{\rm EH}(H_{C,\eta})=\Phi_{H_{C,\eta}}(x)=f'_{C,\eta}(\tau)-f_{C,\eta}(\tau)
$$
where $f'_{C,\eta}(\tau)\in e^\tau\Sigma^{\Psi}_{\mathcal{S}}$ and $\eta<\tau<2\eta$.
Choose $C>0$ sufficiently large and $\eta>0$ sufficiently small such that
$$
c^\Psi_{\rm EH}(H_{C,\eta})<c^\Psi_{\rm EH}(\mathcal{S})+1.
$$
Then the choice of $f$ below (\ref{e:EH.4.14}) implies
$$
\hbox{either $f_{C,\eta}(\tau)<\eta\quad$ or $\quad f_{C,\eta}(\tau)>C-\eta$.}
$$

Choose a sequence of positive numbers $\eta_n\to 0$. Passing to a subsequence we may assume two cases.

\noindent{\bf Case 1}. {\it Suppose that $c^\Psi_{\rm EH}(H_{C,\eta_n})=f'_{C,\eta_n}(\tau_n)-f_{C,\eta_n}(\tau_n)=e^{\tau_n}a_n-f_{C,\eta_n}(\tau_n)$, where
$a_n\in\Sigma_\mathcal{S}^\Psi$,$\eta_n<\tau_n<2\eta_n$ and $0\le f_{C,\eta_n}(\tau_n)<\eta_n$.}
Since $c^\Psi_{\rm EH}(H_{C,\eta_n})\to\Upsilon(C)$, the sequence $e^{-\tau_n}(c^\Psi_{\rm EH}(H_{C,\eta_n})+f_{C,\eta_n}(\tau_n))=a_n$
is a bounded sequence. Passing to a subsequence we may assume that $(a_n)$ is convergent.
Let $a_n\to a_C\in \overline{\Sigma_{\mathcal{S}}^\Psi}$ (the closure of $\Sigma^{\Psi}_{\mathcal{S}}$).
Note that
\begin{eqnarray*}
\lim_{n\to\infty}\left(e^{-\tau_n}(c^\Psi_{\rm EH}(H_{C,\eta_n})+f_{C,\eta_n}(\tau_n))\right)&=&
\lim_{n\to\infty}e^{-\tau_n}(\lim_{n\to\infty}c^\Psi_{\rm EH}(H_{C,\eta_n})+\lim_{n\to\infty}f_{C,\eta_n}(\tau_n))\\
&=&\Upsilon(C).
\end{eqnarray*}
Hence
\begin{equation}\label{e:EH.4.16}
\Upsilon(C)=a_C\in \overline{\Sigma_{\mathcal{S}}^{\Psi}}.
\end{equation}
Note that by standard arguments one can show that $\overline{\Sigma_{\mathcal{S}}^{\Psi}}=\Sigma^{\Psi}_{\mathcal{S}}\cup\{0\}$. Therefore
$\overline{\Sigma_{\mathcal{S}}^{\Psi}}$ also has empty interior.

\noindent{\bf Case 2}. {\it Suppose that $c^\Psi_{\rm EH}(H_{C,\eta_n})=f'_{C,\eta_n}(\tau_n)-f_{C,\eta_n}(\tau_n)=e^{\tau_n}a_n-f_{C,\eta_n}(\tau_n)
=e^{\tau_n}a_n-C-(f_{C,\eta_n}(\tau_n)-C)$, where
$a_n\in\Sigma_\mathcal{S}$, $\eta_n<\tau_n<2\eta_n$ and $C-\eta_n<f_{C,\eta_n}(\tau_n)\le C$ .} As in  Case 1 there holds
\begin{equation}\label{e:EH.4.17}
\Upsilon(C)+C=a_C\in \overline{\Sigma_{\mathcal{S}}^{\Psi}}.
\end{equation}

\noindent{\bf Step 1}. {\it Prove $c^\Psi_{\rm EH}(\mathcal{S})\in\overline{\Sigma^{\Psi}_{\mathcal{S}}}$.}
Suppose that there exists a sequence $C_n\uparrow +\infty$  such that $\Upsilon(C_n)=a_{C_n}\in \overline{\Sigma^{\Psi}_{\mathcal{S}}}$ for each $n$.
Then
$$
c^\Psi_{\rm EH}(\mathcal{S})=\lim_{n\to\infty}\Upsilon(C_n)\in \overline{\Sigma_\mathcal{S}^{\Psi}}.
$$
Otherwise, we have
\begin{equation}\label{e:EH.4.18}
\left.\begin{array}{ll}
&\hbox{there exist $\bar{C}>0$ such that
(\ref{e:EH.4.17}) holds}\\
&\hbox{for each $C\in (\bar{C}, +\infty)$.}
\end{array}\right\}
\end{equation}
Let us  prove that this case does not occur.
Note that (\ref{e:EH.4.18}) implies

\begin{claim}\label{cl:EH.4.3}
If $C<C'$ belong to $(\bar{C}, +\infty)$ then
\begin{equation}\label{e:EH.4.19}
\Upsilon(C)+C\ge \Upsilon(C')+C'.
\end{equation}
\end{claim}

\begin{proof}
Assume that for some $C'>C>\overline{C}$,
\begin{equation}\label{e:EH.4.20}
\Upsilon(C)+C<\Upsilon(C')+C'.
\end{equation}
 We shall prove:
\begin{equation}\label{e:EH.4.21}
\left.\begin{array}{ll}
 &\hbox{ for any given
$d\in (\Upsilon(C)+C,\Upsilon(C')+C')$}\\
&\hbox{ there exist $C_0\in (C, C')$ such that
$\Upsilon(C_0)+C_0=d$.}
\end{array}\right\}
\end{equation}
This contradicts the facts that
${\rm Int}(\Sigma^{\Psi}_{\mathcal{S}})=\emptyset$ and (\ref{e:EH.4.17})
holds for all large $C$.

Put $\Delta_d=\{C''\in (C, C')\,|\, C''+\Upsilon(C'')>d\}$. Since
$\Upsilon(C')+C'>d$ and
$\Upsilon(C')\le\Upsilon(C'')\le\Upsilon(C)$ for any $C''\in (C, C')$,
we obtain $\Upsilon(C'')+C''>d$ if $C''\in (C, C')$ is sufficiently close to $C'$.
Hence $\Delta_d\ne\emptyset$. Then $C_0:=\inf \Delta_d\in [C, C')$.
Let $(C_n'')\subset \Delta_d$ satisfy $C_n''\downarrow C_0$.
By $\Upsilon(C_n'')\le \Upsilon(C_0)$ we have
$d<C_n''+\Upsilon(C_n'')\le\Upsilon(C_0)+ C_n''$ for each $n\in\mathbb{N}$,
and thus $d\le\Upsilon(C_0)+C_0$ by letting $n\to\infty$.
Suppose that
\begin{equation}\label{e:EH.4.22}
d<\Upsilon(C_0)+C_0.
\end{equation}
Since $d>C+\Upsilon(C)$, this implies $C\ne C_0$ and so $C_0>C$.
For $\hat{C}\in (C, C_0)$, from $\Upsilon(\hat{C})\ge\Upsilon(C_0)$ and
(\ref{e:EH.4.22}) we derive that $\Upsilon(\hat{C})+\hat{C}>d$
if $\hat{C}$ is close to $C_0$. Therefore such $\hat{C}$ belongs to
$\Delta_d$, which contradicts $C_0=\inf\Delta_d$ and it follows that (\ref{e:EH.4.22}) does not hold. That is, $d=\Upsilon(C_0)+C_0$.
 (\ref{e:EH.4.21}) is proved. Since (\ref{e:EH.4.21}) contradicts the fact that
 $\overline{\Sigma^{\Psi}_{\mathcal{S}}}$ has empty interior, (\ref{e:EH.4.19}) does hold for all $\overline{C}<C<C'$.
\end{proof}

Since $\Xi:=\{C>\bar{C}\,|\, C\;\hbox{satisfying (\ref{e:EH.4.15})}\}$
is dense in $(\bar{C}, +\infty)$, it follows from Claim~\ref{cl:EH.4.3} that
$\Upsilon(C')+C'\le \Upsilon(C)+C$ if $C'>C$ is in
$\Xi$. Fix a $C^\ast\in\Xi$. Then
$$
\Upsilon(C')+C'\le \Upsilon(C^\ast)+C^\ast,\quad\forall C'\in\{C\in\Xi\,|\, C>C^\ast\}.
$$
Taking a sequence  $(C_n')\subset \{C\in\Xi\,|\, C>C^\ast\}$ such that $C_n'\to +\infty$,
we deduce that $\Upsilon(C_n')\to-\infty$. This contradicts the fact that
$\Upsilon(C_n')\to c^{\Psi}_{\rm EH}(\mathcal{S})>0$. Hence
(\ref{e:EH.4.18}) does not  hold!

\noindent{\bf Step 2.} {\it Prove}
\begin{equation}\label{e:EH.4.23}
c^\Psi_{\rm EH}(B)=c^\Psi_{\rm EH}(\mathcal{S}).
\end{equation}
Construct
\begin{equation}\label{e:EH.4.25}
\hat{H}_{C,\eta}(x)=\left\{\begin{array}{ll}
0 &{\rm if}\;x\in {\bf B}_\delta,\\
\hat{f}_{C,\eta}(\tau) &{\rm if}\;x=\psi(\tau,y),\;y\in\mathcal{S},\;\tau\in [-\delta,\delta],\\
C &{\rm if}\;x\in {\bf A}_\delta\cap B^{2n}(0,R),\\
\hat{h}(|x|^2) &{\rm if}\;x\in {\bf A}_\delta\setminus B^{2n}(0,R)
\end{array}\right.
\end{equation}
where $B^{2n}(0,R)\supseteq\overline{\psi((-\varepsilon,\varepsilon)\times\mathcal{S})}$, $\hat{f}_{C,\eta}:(-\varepsilon, \varepsilon)\to\mathbb{R}$ and $\hat{h}:[0, \infty)\to\mathbb{R}$ are smooth functions satisfying
\begin{eqnarray*}
&&\hat{f}_{C,\eta}|(-\infty,\eta]\equiv 0,\quad \hat{f}_{C,\eta}(s)=C\;\hbox{if}\;s\ge 2\eta,\\
&&\hat{f}'_{C,\eta}(s)s>0\quad\hbox{if}\quad \eta<s<2\eta,\\
&&\hat{f}'_{C,\eta}(s)-\hat{f}_{C,\eta}(s)>c^\Psi_{\rm EH}(\mathcal{S})+1\quad\hbox{if}\;s>0\;\hbox{and}\;
\eta<\hat{f}_{C,\eta}(s)<C-\eta,\\
&&\hat{h}_{C,\eta}(s)=a_Hs+b\quad\hbox{for $s>0$ large enough}, a_H=C/R^2> \mathfrak{t}(\Psi),\\
&&s\hat{h}'_{C,\eta}(s)-\hat{h}_{C,\eta}(s)\le 0\quad\forall s\ge 0,\\
&&\det\left(\exp\left(\frac{2C}{R^2}J\right)-\Psi\right)\neq 0.
\end{eqnarray*}
Then
\begin{equation}\label{e:EH.4.24}
c^\Psi_{\rm EH}(B)=\inf_{\eta>0, C>0}c^\Psi_{\rm EH}(\hat{H}_{C,\eta}).
\end{equation}
For $H_{C,\eta}$ in (\ref{e:EH.4.14}), we choose an associated
$\hat{H}_{C,\eta}$, where $\hat{f}_{C,\eta}|[0,\infty)={f}_{C,\eta}|[0,\infty)$
and $\hat{h}_{C,\eta}={h}_{C,\eta}$. Consider $H_s=sH_{C,\eta}+(1-s)\hat{H}_{C,\eta}$, $0\le s\le 1$ and put
$$
\Phi_s(x):=\Phi_{H_s}(x)\quad\forall x\in \mathbb{E}.
$$
It suffices to prove $c^\Psi_{\rm EH}(H_0)=c^\Psi_{\rm EH}(H_1)$. If $x$ is a critical point of $\Phi_s$ with $\Phi_s(x)>0$, then there holds $x([0,1])\in \mathcal{S}_\tau=\psi(\{\tau\}\times\mathcal{S})$
for some $\tau\in (\eta,2\eta)$. The choice of $\hat{H}_{C,\eta}$ shows
$H_s(x(t))\equiv {H}_{C,\eta}(x(t))$ for $t\in [0,1]$. This implies that
each $\Phi_s$ has the same positive critical values as $\Phi_{H_{C,\eta}}$.
By the continuity in Proposition~\ref{prop:EH.1.2}(ii), $s\mapsto c^\Psi_{\rm EH}(H_s)$ is continuous
and takes values in the set of positive critical values of $\Phi_{H_{C,\eta}}$
which has measure zero (see Section~\ref{sec:convex2}). Hence  $s\mapsto c^\Psi_{\rm EH}(H_s)$ is constant. In particular,
$$
c^\Psi_{\rm EH}(\hat{H}_{C,\eta})=c^\Psi_{\rm EH}(H_0)=c^\Psi_{\rm EH}(H_1)=c^\Psi_{\rm EH}(H_{C,\eta}).
$$
Summarizing the above arguments  we have proved that
$c^{\Psi}_{\rm EH}(\mathcal{S})=c^{\Psi}_{\rm EH}(B)\in \overline{\Sigma^{\Psi}_{\mathcal{S}}}$.
Note that $c^{\Psi}_{\rm EH}(B)>0$. Hence
$c^{\Psi}_{\rm EH}(\mathcal{S})=c^{\Psi}_{\rm EH}(B)\in \Sigma^{\Psi}_{\mathcal{S}}$.
\end{proof}

\section{Proofs of Theorems~\ref{MaSch}, \ref{Str}}\label{sec:appl}
\setcounter{equation}{0}

\subsection{Proof of Theorem~\ref{MaSch}}\label{sec:appl.1}

\noindent{\bf Step 1.}
Let $U=\bigcup_{\lambda \in I}\mathcal{S}_\lambda$ be a thickening of a compact and regular energy surface
$\mathcal{S}=\{x\in M\,|\, H(x)=0\}$ in a symplectic manifold $(M,\omega)$ as in (\ref{thicken}).
Corresponding to \cite[p. 109, Proposition~1]{HoZe94} we have:

\begin{claim}\label{cl:appl.1}
 For a converging sequence $\lambda_j\to\lambda^\ast$  in the interval $I$,
suppose that for every $\lambda_j$  the Hamiltonian boundary value problem
$$
 \dot{x}=X_H(x),\quad x(T_j)=\Psi x(0),\quad 0<T_j<\infty
 $$
 has a solution $x_j:[0,T_j]\to\mathcal{S}_{\lambda_j}$. If $T_j\le C$ for some constant $C>0$ and for all $j$,
 then $S_{\lambda^\ast}$ carries either a fixed point of $\Psi$ or
 a $\Psi$-characteristic $y:[0,T]\to S_{\lambda^\ast}$  satisfying
  $\dot{y}=X_H(y)$ and $0<T\le C$.
  \end{claim}
Indeed, for each $j$, the map $z_j:[0,1]\to S_{\lambda^\ast}$ defined by $z_j(t)=x_j(T_jt)$ satisfies
$\dot{z}_j(t)=T_jX_H(z_j(t))$ and $\Psi z_j(0)=z_j(1)$. By the Arzela-Ascoli theorem,
passing to a subsequence we may assume that $T_j\to T$ and $z_j\to z$ in $C^\infty([0,1], M)$.
Hence $z(1)=\Psi z(0)$ and $\dot{z}(t)=TX_H(z(t)$ for $0\le t\le 1$. Clearly, $0\le T\le C$.
If $T=0$ then $z$ is constant and therefore $z(0)\in S_{\lambda^\ast}$  is a fixed point of $\Psi$.
If $T>0$ then $y:[0,T]\to S_{\lambda^\ast}$ defined by $y(t):=z(t/T)$ is the desired
 $\Psi$-characteristic.

\noindent{\bf Step 2.} Along the ideas in \cite{MaSc05}, for each $n\in\mathbb{N}$ let $G_n$ be the set
of nonzero parameters $\lambda\in I=(-\varepsilon,\varepsilon)$ for which $S_{\lambda}$
contains either a fixed point of $\Psi$ or a $\Psi$-characteristic $y:[0,T]\to S_{\lambda}$  satisfying
  $$
  0<T\le n,\quad
  \dot{y}=X_H(y)\;\hbox{if $\lambda>0$},\quad \dot{y}=-X_H(y)\;\hbox{if $\lambda<0$}.
  $$
 The above claim implies that {$G_n\cup \{0\}$} is closed and therefore that
  $G=\cup^\infty_{n=1}G_n$ is a Lebesgue-measurable set. For $0<\delta<\varepsilon$ we define
  $$
  U_\delta:=\bigcup_{|\lambda|<\delta}\mathcal{S}_\lambda=\{x\in U\,|\, -\delta<H(x)<\delta\}
  $$
which is an open subset in $U$. Since $\mathcal{S}_0=\mathcal{S}$ has nonempty intersection with ${\rm Fix}(\Psi)$,
$U_\delta\cap{\rm Fix}(\Psi)\ne\emptyset$. It follows that $c^{\Psi}_{\rm HZ}(U_\delta,\omega)$ is well-defined
and Proposition~\ref{MonComf}(ii) implies that
 $c^{\Psi}_{\rm HZ}(U_{\delta_1},\omega)\le c^{\Psi}_{\rm HZ}(U_{\delta_2},\omega)\le c^{\Psi}_{\rm HZ}(U,\omega)$ for any $0<\delta_1<\delta_2<\varepsilon$.

\begin{claim}\label{cl:appl.2}
 For $\delta^\ast\in (0,\varepsilon)$, if
 there exist positive numbers
$L>0$ and $\mu\in (\delta^\ast, \varepsilon)$ such that
$$
c^{\Psi}_{\rm HZ}(U_{\delta},\omega)\le c^{\Psi}_{\rm HZ}(U_{\delta^\ast},\omega)+ L(\delta-\delta^\ast),\quad
\forall\delta\in[\delta^\ast, \mu],
$$
then $\delta^\ast\in G$ or $-\delta^\ast\in G$.
\end{claim}
In fact, for a fixed $\delta\in (\delta^\ast, \mu)$, by definition of $c^{\Psi}_{\rm HZ}$ we have $\tilde{H}\in \mathcal{H}_{ad}^{\Psi}(U_{\delta^\ast},\omega)$
such that $\max \tilde{H}>c^{\Psi}_{\rm HZ}(U_{\delta^\ast},\omega)-(\delta-\delta^\ast)$.
As in \cite{MaSc05} (or \cite[p. 315]{Ze10}) we take a smooth function $f:[0,\varepsilon)\to\mathbb{R}$ such that
\begin{description}
\item[(a)] $f(t)=\max \tilde{H}$ for $0\le t\le\delta^\ast$,
\item[(b)] $f(t)=c^{\Psi}_{\rm HZ}(U_{\delta},\omega)+(\delta-\delta^\ast)$ for $\delta\le t<\varepsilon$,
\item[(c)] $f'(t)>0$ for $t\in (\delta^\ast,\delta)$ and $f'(t)\in [0, L+3]$ for all $t\in [0,\varepsilon)$.
\end{description}
Define $F:U_{\delta}\to\mathbb{R}$ by setting $F=\tilde{H}$ on $U_{\delta^\ast}$ and
$F(x)=f(|H(x)|)$ for $x\in U_{\delta}\setminus U_{\delta^\ast}$. Then
$F\in\mathcal{H}^{\Psi}(U_{\delta},\omega)$ and $\max F=c^{\Psi}_{\rm HZ}(U_{\delta},\omega)+(\delta-\delta^\ast)>
c^{\Psi}_{\rm HZ}(U_{\delta},\omega)$. Hence for some $0<T\le 1$ we have a nonconstant differentiable
path $\gamma:[0,T]\to U_{\delta}$ satisfying $\dot{\gamma}=X_F(\gamma)$ and $\Psi\gamma(0)=\gamma(T)$.
Note that $\tilde{H}\in \mathcal{H}_{ad}^{\Psi}(U_{\delta^\ast},\omega)$ may be naturally extended into
an element in $\mathcal{H}_{ad}^{\Psi}( Cl(U_{\delta^\ast}),\omega)$ and
that $c^{\Psi}_{\rm HZ}(U_{\delta^\ast},\omega)=c^{\Psi}_{\rm HZ}( Cl(U_{\delta^\ast}),\omega)$,
where $Cl(A)$ is the closure of $A$.
Using the fact that $F$ is equal to a positive constant  along $\gamma$ we deduce that
$\gamma([0,T])$ is contained in $U_{\delta}\setminus Cl(U_{\delta^\ast})$.
This implies that $H\circ\gamma$ is equal to a constant $c$ in
$(\delta^\ast, \delta)$ or $(-\delta, -\delta^\ast)$  and \\
$\bullet$  $\dot{\gamma}(t)=f'(H(\gamma(t)))X_H(\gamma(t))$ on $[0,T]$ if $c\in (\delta^\ast, \delta)$,\\
$\bullet$  $\dot{\gamma}(t)=-f'(-H(\gamma(t)))X_H(\gamma(t))$ on $[0,T]$ if $c\in (-\delta, -\delta^\ast)$.\\
Let $\tau=f'(|c|)$ which belongs to $(0, L+3)$. Note that  the path $[0, \tau T]\ni t\to y(t)=\gamma(t/\tau)$
sits in $\mathcal{S}_c$ and satisfies $\Psi y(0)=y(\tau T)$ and
$$
   \left\{
   \begin{array}{l}
     \dot{x}=X_H(x)\quad\hbox{if}\quad c>0, \\
    \dot{x}=-X_H(x)\quad\hbox{if}\quad c<0.
   \end{array}
   \right.
$$
Moreover, $0<\tau T\le\tau\le L+3$.

Take a sequence $(\delta_j)$ in the interval $(\delta^\ast, \mu)$ such that $\delta_j\downarrow\delta^\ast$.
By the arguments above we have sequences $\lambda_j\in (\delta^\ast,\delta_j)$ and $T_j\in (0, L+3]$, $j=1,2,\cdots$,
such that for each $j$ there exists \\
$\bullet$ either a $\Psi$-characteristic $y_j:[0,T_j]\to S_{\lambda_j}$  satisfying
  $\dot{y}=X_H(y)$,\\
$\bullet$ or a $\Psi$-characteristic $y_j:[0,T_j]\to S_{-\lambda_j}$  satisfying
  $\dot{y}=-X_H(y)$.

\noindent
{Passing to a subsequence of $(\lambda_j)$ when necessary and using
Claim~\ref{cl:appl.1} we obtain: either $S_{\delta^\ast}$ carries
 a $\Psi$-characteristic $y:[0,T]\to S_{\delta^\ast}$  satisfying
  $\dot{y}=X_H(y)$ and $0<T\le L+3$, or  $S_{-\delta^\ast}$ carries
 a $\Psi$-characteristic $y:[0,T]\to S_{-\delta^\ast}$  satisfying
  $\dot{y}=-X_H(y)$ and $0<T\le L+3$, or $S_{\delta^\ast}\cup S_{-\delta^\ast}$
carries a fixed point of $\Psi$. Claim~\ref{cl:appl.2} is proved.}

\noindent{\bf Step 3.} {\it Prove statement (i).} For a nonzero $\lambda\in I$ let $\mathcal{P}(\mathcal{S}_\lambda,\Psi)$
consist of fixed points of $\Psi$ and $\Psi$-characteristics on
$\mathcal{S}_\lambda$ satisfying $\dot{y}={\rm sign}(\lambda)X_H(y)$, where
${\rm sign}(\lambda)=1$ if $\lambda>0$, and ${\rm sign}(\lambda)=-1$ if $\lambda<0$.
Since monotone nondecreasing function $(0,\varepsilon)\ni\delta\mapsto c^{\Psi}_{\rm HZ}(U_{\delta},\omega)\in\mathbb{R}$
is differentiable almost everywhere and thus Lipschitz continuous almost everywhere, we derive from Step 2
that
$$
\hat{G}:=\{\delta\in (0,\varepsilon)\,|\, \mathcal{P}(\mathcal{S}_\delta,\Psi)\ne\emptyset\;\hbox{or}\;
\mathcal{P}(\mathcal{S}_{-\delta},\Psi)\ne\emptyset\}=\{\delta\in (0,\varepsilon)\,|\, \delta\in G\;\hbox{or}\;
-\delta\in G\}
$$
has Lebesgue measure $\varepsilon$ and thus satisfies the requirement in (i).

\noindent{\bf Step 4.} {\it Prove statement (ii)}. For each $n\in\mathbb{N}$ let $\Lambda_n$ be the set
of nonzero parameters $\lambda\in I=(-\varepsilon,\varepsilon)$ for which $S_{\lambda}$
contains either a fixed point of $\Psi$ or a $\Psi$-characteristic $y:[0,T]\to S_{\lambda}$  satisfying
  $$
  0<T\le n,\quad
  \dot{y}=X_H(y)\quad\hbox{or}\quad \dot{y}=-X_H(y).
  $$
 Then $\Lambda_n$ contains $G_n$, and {$\Lambda_n\cup\{0\}$ is closed in $I$} and so  $\Lambda:=\cup^\infty_{n=1}\Lambda_n$ is a Lebesgue-measurable set.
By Step~3 the set $\{\delta\in (0,\varepsilon)\,|\, \delta\in \Lambda\;\hbox{or}\;
-\delta\in\Lambda\}$ has Lebesgue measure $\varepsilon$.
As in the proof of \cite{MaSc05} it follows from this that $\Lambda$ has Lebesgue measure $m(\Lambda)=2\varepsilon$.

\subsection{Proof of Theorem~\ref{Str}}\label{sec:appl.2}

We closely follow \cite{Str90, Str08}.
By the assumptions there exists $\delta_0>0$ such that
$\mathbb{R}^{2n}\setminus H^{-1}([1-\delta_0, 1+\delta_0])$
consists of a bounded component ${\bf B}$ containing the fixed point $z_0$ of $\Psi$
and an unbounded one  ${\bf A}$. As the arguments at the beginning
of Section~\ref{sec:convex} we may reduce to the case $z_0=0$.
 Then
 $r_0:=\sup\{\tau>0\,|\, B^{2n}(0,\tau)\subset{\bf B}\}>0$.
Clearly, by shrinking $\delta_0$ we can also assume
 that each $\beta\in [1-\delta_0, 1+\delta_0]$
is a regular value of $H$ and $\mathcal{S}_\beta:=H^{-1}(\beta)$ is diffeomorphic to $\mathcal{S}=\mathcal{S}_1$.
Let $\gamma$ be the diameter of $H^{-1}([1-\delta_0, 1+\delta_0])$.
Then $\gamma>r_0$. Fix $r\in (\gamma, 2\gamma)$. {Then
 $H^{-1}([1-\delta_0, 1+\delta_0])\cup{\bf B}$ is contained in
 the ball $B^{2n}(0,r)$.}

Fix a number $\beta_0\in (1-\delta_0, 1+\delta_0)$ and choose $0<\delta<(\delta_0-|1-\beta_0|)/3$.
Then the closure of  $I_0:=(\beta_0-\delta,\beta_0+\delta)$
is contained in $(1-\delta_0, 1+\delta_0)$.
Define
$$
U_{\delta}=H^{-1}\left([1-\delta_0+\delta,  1+\delta_0-\delta]\right).
$$
Let ${\bf A}_{\delta}$ and ${\bf B}_{\delta}$ be the unbounded and bounded components
of $\mathbb{R}^{2n}\setminus U_{\delta}$ respectively. Then
$U_\delta\cup {\bf B}_\delta\subset B^{2n}(0,r)$.
We modify  the constant $b$ and the smooth functions $f, g$
in \cite{HoZe87} such that
\begin{eqnarray*}
&&(\frac{\mathfrak{t}(\Psi)}{2}+\epsilon) r^2<b<\frac{2\mathfrak{t}(\Psi)+4\epsilon}{3} r^2,\\
&&f(s)=0\;\hbox{for}\;s\le-\delta,\quad f(s)=b\;\hbox{for}\;s\ge\delta,\quad f'(s)>0\;\hbox{for}\;|s|<\delta,\\
&&g(s)=b\;\hbox{for}\;s\le r,\quad g(s)\ge(\frac{\mathfrak{t}(\Psi)}{2}+\epsilon) s^2\;\hbox{for}\;s>r,\quad
g(s)=(\frac{\mathfrak{t}(\Psi)}{2}+\epsilon) s^2\;\hbox{for large $s$},\\
&&{\hbox{where $0<\epsilon\ll 1$ satisfies $\det(\Psi-e^{(\mathfrak{t}(\Psi)+2\epsilon)J})\neq 0$},}\\
&&0<g'(s)\le (\mathfrak{t}(\Psi)+2\epsilon)s\;\hbox{for}\;s>r.\;\hbox{(So $g(s)\le b+ (\frac{\mathfrak{t}(\Psi)}{2}+\epsilon)(s^2-r^2)\;\forall r$)}.
\end{eqnarray*}
Following \cite[Chap. II, \S9]{Str08},
for $m\in\mathbb{N}$ and $\beta\in I_0$ we define
$$
H_{\beta,m}(x)=\left\{\begin{array}{ll}
0 &{\rm if}\;x\in {\bf B}_{\delta},\\
f(m(H(x)-\beta)) &{\rm if}\;x\in U_\delta,\\
b &{\rm if}\;x\in {\bf A}_{\delta}\cap B^{2n}(0,r),\\
g(|x|) &{\rm if}\;x\in {\bf A}_{\delta}\setminus B^{2n}(0,r)
\end{array}\right.
$$
and a functional $\Phi_{m,\alpha}$ on the space $\mathbb{E}$ in (\ref{e:spaceDecomposition}) by
$$
{\Phi_{\beta,m}(x)=\Phi_{H_{\beta,m}}(x)=\frac{1}{2}(\|x^+\|^2_{\mathbb{E}}-\|x^-\|^2_{\mathbb{E}})-\int_0^1 H_{\beta,m}(x(t))dt.}
$$
Note that for any $m\in\mathbb{N}$ and $\beta\in I_0$ the function $H_{\beta,m}$ satisfies inequalities
\begin{eqnarray}\label{e:str.1}
&&-b+(\frac{\mathfrak{t}(\Psi)}{2}+\epsilon)|x|^2\le H_{\beta,m}(x)\le (\frac{\mathfrak{t}(\Psi)}{2}+\epsilon)|x|^2+ b\quad
     \forall x\in\mathbb{R}^{2n},\\
&&H_{\beta_1,m}\ge H_{\beta_2,m}\quad\hbox{for}\;\beta_1\le\beta_2,\;\beta_i\in I_0,\;i=1,2.\label{e:str.2}
\end{eqnarray}
Since $B^{2n}(0, r_0)\subset{\bf B}\subset{\bf B}_\delta$ ,
$H_{\beta,m}\equiv 0$ in $B^{2n}(0, r_0)$ and there exist constants $C_1>0, C_2>0$ independent of $m\in\mathbb{N}$
and $\beta\in I_0$ such that
\begin{eqnarray}\label{e:str.3}
|H_{\beta,m}(z_0+x)|\le C_1|x|^2\quad\hbox{and}\quad
|H_{\beta,m}(z_0+x)|\le C_2|x|^3\quad\forall x\in\mathbb{R}^{2n}.
\end{eqnarray}
Moreover, for $|x|$ large enough we have uniformly (in $m\in\mathbb{N}$ and $\beta\in I_0$)
$$
{\nabla H_{\beta,m}(x)= (\mathfrak{t}(\Psi)+ 2\epsilon)x\quad\hbox{and}\quad
(H_{\beta,m})_{xx}= (\mathfrak{t}(\Psi)+ 2\epsilon)I_{2n}.}
$$
It follows that for some positive constants $C_3, C_4$ independent of $m\in\mathbb{N}$
and $\beta\in I_0$ there holds
\begin{equation}\label{e:str.4}
{ |\nabla H_{\beta,m}(x)|\le(\mathfrak{t}(\Psi)+ 2\epsilon)|x|+C_3\quad\hbox{and}\quad
|(H_{\beta,m})_{xx}|\le C_4}
\end{equation}
on $\mathbb{R}^{2n}$. Then $\Phi_{\beta,m}$ is still a $C^1$-functional on $\mathbb{E}$ satisfying  the (PS) condition, and
each critical point $x$ of it is smooth and satisfies the Hamiltonian boundary value problem
\begin{equation}\label{e:str.5}
       \dot{x}=X_{H_{\beta,m}}(x)\quad\hbox{and}\quad x(1)=\Psi x(0).
\end{equation}
(See Propositions~\ref{prop:solution} and \ref{prop:PSmale}.)
By Proposition~\ref{Lip} the conditions in (\ref{e:str.4}) also insures that
$\nabla\Phi_{\beta,m}$ satisfies
\begin{equation}\label{e:str.6}
     \|\nabla\Phi_{\beta,m}(x)-\nabla\Phi_{\beta,m}(y)\|_{\mathbb{E}}\le \ell\|x-y\|_{\mathbb{E}}\quad\forall x,y\in \mathbb{E}
\end{equation}
for some constant $\ell>0$ independent of $m\in\mathbb{N}$
and $\beta\in I_0$.

As in the proofs of \cite[page 138, Lemma~9.2]{Str08} we have:

\begin{claim}\label{cl:appl.3}
Let $x\in\mathbb{E}$ be a critical point of $\Phi_{\beta,m}$ with $\Phi_{\beta,m}(x)>0$. Then
$H(x(t))\equiv h\in\mathbb{R}$ with $|h-\beta|<\delta/m$,
$$
T_x=mf'(m(H(x)-\beta))>0,
$$
and $y:[0,T_x]\to \mathbb{R}^{2n},\;t\mapsto x(t/T_x)$ satisfies:
$\dot{x}=X_{H}(x)$ and $x(T_x)=\Psi x(0)$.
\end{claim}

Let  $\mathbb{E}^+, \mathbb{E}^0, \mathbb{E}^-$ be as in (\ref{e:spaceDecomposition}).
Corresponding to \cite[page 138, Lemma~9.3]{Str08} there holds:
\begin{claim}\label{cl:appl.4}
There exist numbers $\alpha>0$, $\rho>0$ independent of $m\in\mathbb{N}$
and $\beta\in I_0$ such that
$\Phi_{\beta,m}(x)\ge\alpha$ for $x\in S^+_\rho=\{x\in \mathbb{E}^+\,|\,
\|x\|_{\mathbb{E}}=\rho\}$.
\end{claim}
\begin{proof}
Since we do not know whether the space $\mathbb{E}$ in (\ref{e:spaceDecomposition})
can be embedded into some $L^p([0,1],\mathbb{R}^{2n})$ with $p>2$,
the proof of \cite[page 138, Lemma~9.3]{Str08} does not work in the present case.
However, because of estimates in (\ref{e:str.3}),
as in the proof of Proposition~\ref{prop:EH.1.4}  we may still use
the method in the proof of \cite[page 93, Lemma~9]{HoZe90}.
 Indeed,  it suffices to prove that
\begin{equation}\label{e:str.7}
\lim_{\|x\|_{\mathbb{E}}\rightarrow 0}\frac{\int_0^1 H_{\beta,m}(x)}{\|x\|^2_{\mathbb{E}}}=0
\end{equation}
uniformly $m\in\mathbb{N}$ and $\beta\in I_0$.
 Otherwise, suppose there exist sequences $(x_j)\subset\mathbb{E}$, $(m_j)\subset\mathbb{N}$ and $(\beta_j)\subset I_0$,
   and $d>0$ satisfying
  \begin{equation}\label{e:str.8}
  \|x_j\|_{\mathbb{E}}\rightarrow 0 \quad \hbox{and} \quad \frac{\int_0^1 H_{\beta_j,m_j}(x_j)}{\|x_j\|^2_{\mathbb{E}}}
  \geq d>0\quad\forall j.
  \end{equation}
  Let $y_j=\frac{x_j}{\|x_j\|_{\mathbb{E}}}$. Then $\|y_j\|_{\mathbb{E}}=1$ and hence
   $(y_j)$ has a convergent subsequence in $L^2$.
  By  \cite[Th.4.9]{Br11}
  we have $w\in L^2$ and a subsequence of $(y_j)$, still denoted by $(y_j)$, such that
   $y_j(t)\rightarrow y(t)$ a.e. on $(0,1)$ and that
   $|y_j(t)|\leq w(t)$  a.e. on $(0,1)$ and for each $j$.
  As in the proof of Proposition~\ref{prop:EH.1.4} it follows from (\ref{e:str.3}) that
     \begin{eqnarray*}
  &&\frac{H_{\beta_j,m_j}(x_j(t))}{\|x_j\|_{\mathbb{E}}^2}\leq  C_1w(t)^2,\quad \hbox{a.e. on $ (0,1)$},\;\forall j,\\
 &&  \frac{H_{\beta_j,m_j}(x_j(t))}{\|x_j\|_{\mathbb{E}}^2}\leq  C_2|x_j(t)|w(t)^2,\quad \hbox{a.e. on $(0,1)$},\;\forall j,
   \end{eqnarray*}
   where $C_1$, $C_2$ are independent of $m_j$ and $\beta_j$.
   Moreover, by the first claim in (\ref{e:str.8})  $(x_j)$
  has a subsequence $x_{j_l}(t)\rightarrow 0$ a.e. on $(0,1)$.
  Using the Lebesgue dominated convergence theorem we deduce
    $$
  \int_0^1 \frac{H_{\beta_j,m_j}(x_{j_l}(t))}{\|x_{j_l}\|_{\mathbb{E}}^2}\rightarrow 0
  $$
  which contradicts the second claim in (\ref{e:str.8}).
  \end{proof}
Similar to Proposition~\ref{prop:EH.1.3},
let $\hat{e}(t)= \frac{1}{\mathfrak{t}(\Psi)}e^{ \mathfrak{t}(\Psi) Jt}X$ where $X\in\mathbb{R}^{2n}$ satisfies
$e^{\mathfrak{t}(\Psi)J}X=\Psi X$ and $|X|=1$. Then $\hat{e}\in \mathbb{E}^+$, $\|\hat{e}\|_{\mathbb{E}}=1$ and $\|\hat{e}\|_{L^2}=\frac{1}{\mathfrak{t}(\Psi)}$.
Define
$$
Q_R:=\{x=s\hat{e}+ x^0+x^-\in \mathbb{E}\,|\, \|x^0+x^-\|_{\mathbb{E}}\le R,\;0\le s\le R\}.
$$
Let $\partial Q_R$ denote the relative boundary of $Q_R$ in $\mathbb{E}^-\oplus \mathbb{E}^0\oplus\mathbb{R}\hat{e}$.

\begin{claim}\label{cl:appl.5}
There exists a number $R>\rho$,  independent of $m\in\mathbb{N}$
and $\beta\in I_0$, such that
$\Phi_{\beta,m}|_{\partial Q_R}\le 0$.
\end{claim}

\begin{proof}
For $x=s\hat{e}+ x^0+x^-$ with $s=R$ or $\|x^0+x^-\|_{\mathbb{E}}=R$, by (\ref{e:str.1}) we have
\begin{eqnarray*}
&&\mathfrak{a}(x)=\frac{1}{2}(\|s\hat{e}\|^2_{\mathbb{E}}-\|x^-\|_{\mathbb{E}}^2)=\frac{1}{2}(s^2-\|x^-\|_{\mathbb{E}}^2),\\
&&\int^1_0H_{\beta,m}(x(t))dt\ge -b+ (\frac{\mathfrak{t}(\Psi)}{2}+\epsilon)\int^1_0|x(t)|^2dt.
\end{eqnarray*}
Note that $\|s\hat{e}+ x^0+x^-\|^2_{L^2}=\|x^0\|^2_{L^2}+\|x^-\|^2_{L^2}+\frac{s^2}{\mathfrak{t}(\Psi)}$.
Hence we arrive at
\begin{eqnarray*}
\Phi_{\beta,m}(x)&\le&
\frac{1}{2}(s^2-\|x^-\|_{\mathbb{E}}^2)+
 b-(\frac{\mathfrak{t}(\Psi)}{2}+\epsilon)(\|x^0\|^2_{L^2}+\|x^-\|^2_{L^2}+\frac{s^2}{\mathfrak{t}(\Psi)})\\
 &=&-\frac{\epsilon s^2}{\mathfrak{t}(\Psi)}-\frac{1}{2}\|x^-\|_{\mathbb{E}}^2-
 (\frac{\mathfrak{t}(\Psi)}{2}+\epsilon)(\|x^0\|^2_{L^2}+\|x^-\|^2_{L^2}+b\\
 &\le& 0
\end{eqnarray*}
if $R>0$ is sufficiently large. Moreover, it is clear that $\Phi_{\beta,m}(x)\le 0$
for $x=s\hat{e}+ x^0+x^-$ with $s=0$.
\end{proof}

As in \cite[page 134]{Str08} let $\tilde{\Gamma}$ be the class of maps
$h\in C^0(\mathbb{E}, \mathbb{E})$ such that $h$ is homotopic to the identity through
a family of maps $h_t=L_t+K_t$, $0\le t\le T$, where $L_0={\rm id}_{\mathbb{E}}$, $K_0=0$
and $L_t:\mathbb{E}\to \mathbb{E}$ is a Banach space isomorphism satisfying $L_t(\mathbb{E}^\ast)=\mathbb{E}^\ast$,
$\ast=0,+,-$, and where $K_t$ is compact and $h_t(\partial Q_R)\cap S^+_\rho=\emptyset$
for each $t$.
Repeating the proofs of \cite[Lemmas~8.10,8.11]{Str08} we can obtain that $\partial Q_R$ and $S^+_\rho$
link with respect to $\tilde{\Gamma}$ and that the gradient flow $G:\mathbb{E}\times [0, \infty)\to \mathbb{E}$
given by
$$
\frac{\partial}{\partial t}G(x,t)=-\nabla\Phi_{\beta,m}(G(x,t))\quad\hbox{and}\quad G(x,0)=x
$$
exists globally and $G(\cdot,T)\in\tilde\Gamma$ for any $T\ge 0$.
Hence $G(\partial Q_R,t)\cap S^+_\rho\ne\emptyset$ for all $t\ge 0$.
By the standard arguments we deduce that
 \begin{equation}\label{e:str.9}
c(H_{\beta,m}):=\inf_{t\ge 0}\sup_{x\in Q_R}\Phi_{\beta,m}(G(x,t))\ge
\inf_{x\in S^+_\rho}\Phi_{\beta,m}(x)\ge\alpha
  \end{equation}
 are positive critical values for all $\beta, m$.
 On the other hand, for any $t\ge 0$ it holds that
 \begin{eqnarray*}
 \sup_{x\in Q_R}\Phi_{\beta,m}(G(x,t))&\le& \sup_{x\in Q_R}\Phi_{\beta,m}(G(x,0))\\
 &=&\sup_{x\in Q_R}\Phi_{\beta,m}(x)\\
&\le& \sup_{x\in \mathbb{E}^-\oplus \mathbb{E}^0\oplus\mathbb{R}_{\ge 0}e}\Phi_{H_{\beta,m}}(x)\\
&\le&\sup_{z\in\mathbb{C}^n}\left(\frac{ \mathfrak{t}(\Psi)}{2}|z|^2-H_{\beta,m}(z)\right)
 \end{eqnarray*}
where the final inequality is obtained as in the proof of (\ref{e:EH.1.3}) in
 Proposition~\ref{prop:EH.1.3}.
 By this and (\ref{e:str.1}) we also have
\begin{eqnarray}\label{e:str.10}
c(H_{\beta,m})\le\sup_{z\in\mathbb{C}^n}\left(\frac{ \mathfrak{t}(\Psi)}{2}|z|^2-H_{\beta,m}(z)\right)\le b.
\end{eqnarray}

\begin{claim}\label{cl:appl.6}
For fixed $x\in \mathbb{E}$ and $m\in\mathbb{N}$, $I_0\ni\beta\mapsto\Phi_{\beta,m}(x)$
 is monotone non-decreasing by (\ref{e:str.2}) and there holds
 $$
\frac{\partial}{\partial\beta}\Phi_{\beta,m}(x)=m\int^1_0f'(m(H(x(t))-\beta))dt.
$$
 In particular, for a critical point $x$ of $\Phi_{\beta,m}$ with $\Phi_{\beta,m}(x)>0$,
 $\frac{\partial}{\partial\beta}\Phi_{\beta,m}(x)$ is equal to $T_x$ given by
 Claim~\ref{cl:appl.3}.
 \end{claim}

Corresponding to the critical value $c(H_{\beta,m})$ in (\ref{cl:appl.6}) we have a critical point
$x_{\beta,m}\in \mathbb{E}$. Since for each $m\in\mathbb{N}$ the map $I_0\ni\beta\mapsto
c(H_{\beta,m})=\Phi_{\beta,m}(x_{\beta,m})$ is non-decreasing, as in the arguments on \cite[page 140]{Str08}
we have $C_\beta:=\lim\inf_{m\to\infty}\frac{\partial}{\partial\beta}c(H_{\beta,m})<\infty$
for almost every $\beta\in I_0$. Fixing such a $\beta$
we get a subsequence $\Lambda\subset\mathbb{N}$ such that
$\frac{\partial}{\partial\beta}c(H_{\beta,m})\to C_\beta$ as $m\in\Lambda$ and $m\to\infty$.
Repeating the proof of \cite[Lemma~9.4]{Str08} yields

\begin{claim}\label{cl:appl.7}
For any $m\in\Lambda$ there exists a critical point $x_{\beta,m}$ of $\Phi_{\beta,m}$ such that
 $\Phi_{\beta,m}(x_{\beta,m})=c(H_{\beta,m})\ge\alpha$
 and $T_{\beta,m}:=\frac{\partial}{\partial\alpha}\Phi_{\beta,m}(x_{\beta,m})\le C_\beta+4$.
 \end{claim}

By Claim~\ref{cl:appl.3}, $x_{\beta,m}$ is smooth and satisfies
$H(x_{\beta,m}(t))\equiv h_{\beta,m}\in (\beta-\delta/m, \beta+\delta/m)$ and
\begin{equation}\label{e:str.11}
\left\{
\begin{array}{ll}
&\dot{x}_{\beta,m}=X_{H_{\beta,m}}(x_{\beta,m})=mf'\left(m(H(x_{\beta,m})-\beta\right)X_H(x_{\beta,m})=T_{\beta,m}X_H(x_{\beta,m}),\\
&x_{\beta,m}(1)=\Psi x_{\beta,m}(0).
\end{array}\right.
\end{equation}
It follows that the sequences $(x_{\beta,m})$ and $(\dot{x}_{\beta,m})$ are uniformly bounded and equi-continuous.
Since $(T_{\beta,m})$ is bounded we may assume $T_{\beta,m}\to T\le C_\beta+1$.
By the Arz\'ela-Ascoli theorem we get a subsequence $x_{\beta,m_j}$ converging in $C^1([0,1],\mathbb{R}^{2n})$
to a solution of
\begin{equation}\label{e:str.12}
       \dot{x}=TX_{H}(x)\quad\hbox{and}\quad x(1)=\Psi x(0)
\end{equation}
with $H(x(t))\equiv \beta$. Note that $A(x_{\beta,m_j})\ge \Phi_{\beta,m_j}(x_{\beta,m_j})\ge\alpha$. Let $j\to\infty$ and we get
$A(x)\ge \alpha$. This implies that $x$ is non-constant and $T>0$.
Since $H(x(t))\equiv \beta$, we obtain that $x([0,1])\subset U_\delta$ and so
$$
\int^1_0H_{\beta,m_j}(x_{\beta,m_j}(t))dt\le b,\quad \forall j.
$$
This and (\ref{e:str.10}) lead to
$$
\alpha\le A(x_{\beta,m_j})= \Phi_{\beta,m_j}(x_{\beta,m_j})+ \int^1_0H_{\beta,m_j}(x_{\beta,m_j}(t))dt\le 2b<
\frac{16\mathfrak{t}(\Psi)+32\epsilon}{3}\gamma^2.
$$
Clearly $0<\epsilon\ll 1$ can be chosen to satisfy $0<\epsilon<\mathfrak{t}(\Psi)$.
Hence $\alpha\le A(x)<16\mathfrak{t}(\Psi)\gamma^2$.
Finally, $y(t):=x(t/T)$ sits in $\mathcal{S}_\beta$ with action
$A(y)=A(x)<16\mathfrak{t}(\Psi)\gamma^2$ satisfying $\dot{y}=X_{H}(y)$ and $y(T)=\Psi y(0)$.

\section{Proof of Theorem~\ref{th:convexDiff}}\label{sec:convexDiff}
\setcounter{equation}{0}

Under the assumptions of Theorem~\ref{th:convexDiff},  for each number $\epsilon$ with
$|\epsilon|$ small enough the set $D_\epsilon:=D(e_0+\epsilon)$
is a strictly convex bounded domain in $\mathbb{R}^{2n}$ with $0\in D_\epsilon$
and with $C^2$-boundary $\mathcal{S}_\epsilon=\mathcal{S}(e_0+\epsilon)$.
Following the notations in Theorem~\ref{th:convexDiff} and Section~\ref{sec:convex}
let $H_\epsilon=(j_{D_\epsilon})^2$ and $H^\ast_\epsilon$ denote the
Legendre transform of $H_\epsilon$. Both $H_\epsilon$ and $H^\ast_\epsilon$
are $C^{1,1}$ on $\mathbb{R}^{2n}$, $C^2$ on $\mathbb{R}^{2n}\setminus\{0\}$
and have positive Hessian matrixes at every point on $\mathbb{R}^{2n}\setminus\{0\}$.
Recall that
$H_\epsilon^\ast(x)=\langle\xi_\epsilon(x), x\rangle-H_\epsilon(\xi_\epsilon(x))$,
where $\nabla H_\epsilon(\xi_\epsilon(x))=x$ and $\nabla H_\epsilon^\ast(x)=\xi_\epsilon(x)$.
It was proved in \cite{Ned01} that $\epsilon\mapsto\xi_\epsilon$ is $C^1$
and $H_\epsilon(x)$ and $H^\ast_\epsilon(x)$ are $C^2$ functions of $\epsilon$ for  every fixed $x\in\mathbb{R}^{2n}\setminus\{0\}$.

Let $x^\ast:[0,\mu]\rightarrow \mathcal{S}_0$ satisfying
$$
\dot{x}=J\nabla H_0(x),\quad x(\mu)=\Psi x(0)
$$
be a $c^\Psi_{\rm HZ}$-carrier for $D_0$.
Then $\mu=A(x^\ast)=c^\Psi_{\rm HZ}(D_0,\omega_0)$.
By the proof in Step 3 of Section~\ref{sec:convex1}, for some $a_0\in{\rm Ker}(\Psi-I_{2n})\subset\mathbb{R}^{2n}$,
$$
u:[0, 1]\to \mathbb{R}^{2n},\;t\mapsto \frac{1}{\sqrt{\mu}}x^\ast(\mu t)-\frac{a_0}{\mu}
$$
belongs to $\mathcal{F}$ in (\ref{e:constrant1}) and satisfies
 $A(u)=1$ and
 \begin{equation}\label{e:Neduv.1}
-J\dot{u}(t)=\nabla H_0(\mu u(t)+ a_0)\quad\forall t\in [0,1].
\end{equation}
There holds
 $$
 c^\Psi_{\rm HZ}(D_0,\omega_0)=\int^1_0H^\ast_0(-J\dot{u})dt
 $$
and the arguments of Section~\ref{sec:convex1} also imply
\begin{equation}\label{e:Neduv.2}
\mathfrak{C}(\epsilon):=\mathscr{C}(e_0+\epsilon)=c^\Psi_{\rm HZ}(D_\epsilon,\omega_0)\le \int^1_0H^\ast_\epsilon(-J\dot{u})dt.
\end{equation}
By the Taylor's formula
\begin{equation}\label{e:Neduv.7}
H^\ast_\epsilon(-J\dot{u}(t))=H^\ast_0(-J\dot{u}(t))+\frac{\partial H^\ast_\epsilon}{\partial\epsilon}\Big|_{\epsilon=0}(-J\dot{u}(t))\epsilon+
\frac{1}{2}\frac{\partial^2 H^\ast_\epsilon}{\partial\epsilon^2}\Big|_{\epsilon=\tau}(-J\dot{u}(t))\epsilon^2
\end{equation}
where $0<\tau<\epsilon$. Let
$$T_{x^\ast}=2\int^{\mathscr{C}(e)}_0\frac{dt}{\langle\nabla\mathscr{H}(x^\ast(t)), x^\ast(t)\rangle}.$$
Then compute as in \cite{Ned01}
$$
\int^1_0\frac{\partial H^\ast_\epsilon}{\partial\epsilon}\Big|_{\epsilon=0}(-J\dot{u}(t))dt
=T_{x^\ast}
$$
and there exists a constant $K$ only depending on $\mathcal{S}_0$
 and $H_\epsilon$ with $\epsilon$ near $0$ such that
$$
\left|\frac{\partial^2 H^\ast_\epsilon}{\partial\epsilon^2}\Big|_{\epsilon=\tau}(-J\dot{u}(t))\right|\le
2K,\quad\forall t\in [0,1].
$$
Then for $\epsilon$ near $0$ there holds
\begin{equation}\label{e:8estimate}
\mathfrak{C}(\epsilon)\le c^\Psi_{\rm HZ}(D_0,\omega_0)+ T_{x^\ast}\epsilon+ K\epsilon^2.
\end{equation}

Recall that
$T^{\max}(e_0+\epsilon)$ and $T^{\min}(e_0+\epsilon)$ are the largest and smallest numbers in
the compact set $\mathscr{I}(e_0+\epsilon)$ defined by (\ref{e:convexDiff}).
By \cite[Lemma~4.1]{Ned01} and \cite[Corollary~4.2]{Ned01},
both  are functions of bounded variation in $\epsilon$
(and thus bounded near $\epsilon=0$), and $\epsilon\mapsto \mathscr{C}(e_0+\epsilon)$ is continuous.
As in the proof of \cite[Theorem~4.4]{Ned01}, using these and (\ref{e:8estimate}) we can
show that $\mathfrak{C}(\epsilon)$ has left and right derivatives
at $\epsilon=0$, i.e.,
\begin{eqnarray*}
&&\mathfrak{C}'_-(0)=\lim_{\epsilon\to0-}T^{\max}(e_0+\epsilon)=T^{\max}(e_0)\quad\hbox{and}\\
&&\mathfrak{C}'_+(0)=\lim_{\epsilon\to0+}T^{\min}(e_0+\epsilon)=T^{\min}(e_0),
\end{eqnarray*}
which complete the proof of the first part of Theorem~\ref{th:convexDiff}.
The final part is a direct consequence of the first one and
a modified version of the intermediate value theorem (cf. \cite[Theorem~5.1]{Ned01}).

\appendix
\section{Appendix:\quad Some facts  on symplectic matrixes}\label{app:A}\setcounter{equation}{0}
For a symplectic  matrix  $\Psi\in{\rm Sp}(2n,\mathbb{R})$, recall that
  $$
   g^{\Psi}:\mathbb{R}\rightarrow \mathbb{R}, \,s\mapsto \det (\Psi-e^{sJ}),
 $$
and $\mathfrak{t}(\Psi)$
is the smallest zero point of $g^{\Psi}$ in $(0, 2\pi]$.

\begin{lemma}\label{zeros}
   For ~$\Psi\in{\rm Sp}(2n, \mathbb{R})$, the set of zero points of the function $g^\Psi$
   in $(0, 2\pi]$ is a nonempty finite set. Moreover,
    $\mathfrak{t}(\Psi)=2\pi$ if $\Psi=I_{2n}$ and  $\mathfrak{t}(\Psi)=\pi$ if $\Psi=-I_{2n}$.
\end{lemma}
\begin{proof}
By \cite[Corollary 3]{Cl82} the function $g^\Psi$ must have a zero point in $(0, 2\pi]$.
 Since $g^\Psi$ is analytic,
  we get that  $g^\Psi$ has at most finitely many zero points in the interval $(0,2\pi]$.

Note that $J$ is unitarily similar to
   $$
   \left(
     \begin{array}{cc}
       \sqrt{-1}I_n & 0 \\
       0 & -\sqrt{-1}I_n
     \end{array}
   \right)\in GL(2n,\mathbb{C}).
   $$
Hence $e^{sJ}$ is unitarily similar to
   $$
   \left(
     \begin{array}{cc}
       e^{s\sqrt{-1}}I_n & 0 \\
       0 & e^{-s\sqrt{-1}}I_n
     \end{array}
   \right)\in GL(2n,\mathbb{C}).
   $$
Therefore
\begin{eqnarray*}
\det\left[I-\left(
     \begin{array}{cc}
       e^{s\sqrt{-1}}I_n & 0 \\
       0 & e^{-s\sqrt{-1}}I_n
     \end{array}
   \right)\right]=(1-e^{s\sqrt{-1}})^n(1-e^{-s\sqrt{-1}})^n
\end{eqnarray*}
and $\det(I-e^{sJ})=0$ if and only if $s\in 2\mathbb{Z}\pi$.
Similarly,
\begin{eqnarray*}
 \det\left[-I-\left(
     \begin{array}{cc}
       e^{s\sqrt{-1}}I_n & 0 \\
       0 & e^{-s\sqrt{-1}}I_n
     \end{array}
   \right)\right]=(-1-e^{s\sqrt{-1}})^n(-1-e^{-s\sqrt{-1}})^n
\end{eqnarray*}
and $\det(-I-e^{sJ})=0$ if and only if $ s\in\pi+2\mathbb{Z}\pi$.
 Hence the second claim in the lemma follows.
\end{proof}

\begin{remark}
  {\rm  In general, if $\Psi$ is not symplectic, $\det (\Psi-e^{sJ})$ may not have finitely many zero points in $(0,2\pi]$.
      For example,  the  matrix $
     \Psi=\left(\begin{array}{cc}
              1 & 0 \\
              0 & -1 \\
            \end{array}
          \right)$
      is not symplectic and
    it is easy to compute that
   \begin{eqnarray*}
   \det\left(\frac{1}{2}\Psi-e^{sJ}\right)\equiv \frac{3}{4}\;\;\forall s\in \mathbb{R}\quad\text{and}\quad
   \det( \Psi-e^{sJ})\equiv 0\;\;\forall s\in \mathbb{R}.
   \end{eqnarray*}
   }
\end{remark}
\vspace{2mm}

    \begin{lemma}\label{zeros.1}
   Let
   $$
   \Psi=\left(\begin{array}{cc}
       U & -V \\
       V & U
       \end{array} \right)
       \in{\rm Sp}(2n, \mathbb{R})\cap O(2n)
   $$
 and $e^{\sqrt{-1}\theta_1},\cdots,e^{\sqrt{-1}\theta_n}$
 $(0< \theta_1\leq\cdots\leq \theta_n\leq 2\pi)$ be eigenvalues of $U+iV$.
  Then the set of zero points of the function $g^\Psi$
   in $(0, 2\pi]$ is $\{\theta_1,\cdots,\theta_{n}\}$ and
    $\mathfrak{t}(\Psi)=\theta_1$.
\end{lemma}
\begin{proof}
For $x,y\in\mathbb{R}^n$,
 $$
 (e^{tJ }-\Psi)\left(\begin{array}{c}
                           x \\
                           y \\
                         \end{array}
                       \right)=0\quad\Leftrightarrow\quad (U+\sqrt{-1}V)(x+\sqrt{-1})=e^{\sqrt{-1}t}(x+\sqrt{-1}y)
 $$
     and thus $\det  (\Psi-e^{tJ})=0 \Leftrightarrow  t =\theta_j$ for some $1\leq j\leq n$.
\end{proof}

\begin{remark}
Let $\{e_1,f_1=Je_1,\cdots,e_n,f_n=Je_n\}$  be the standard basis of $\mathbb{R}^{2n}$, i.e. $e_j\in \mathbb{R}^{2n}$ is the unit vector whose the $j$-th component equals $1$
and others are zero. For $P$ and $\widetilde{\Psi}$ as in  (\ref{othsymp}),(\ref{similar}) and (\ref{diagonal}),
define $X_j=Pe_j$ and $Y_j=Pf_j$. Then  $\widetilde{\Psi}e_j=e^{\theta_jJ}e_j$ and $\widetilde{\Psi}f_j=e^{\theta_jJ}f_j$ for $j=1,\cdots,n$.
 So $\Psi X_j=e^{\theta_jJ} X_j$,
$\Psi Y_j=e^{\theta_jJ} Y_j$, $j=1,\cdots,n$, and
 \begin{equation}\label{basis}
 \{X_j, Y_j=JX_j\}_{1\leq i\leq n}
 \end{equation}
 is a symplectic and orthogonal basis of $(\mathbb{R}^{2n}, \omega_0, J)$.
\end{remark}

\vspace{5mm}
\noindent{\bf Declarations}\\


\noindent{\bf Conflict of interest} The authors have no conflicts of interest.


\medskip

\begin{tabular}{l}
Department of Mathematics, Civil Aviation University of China\\
 Tianjin  300300, The People's Republic of China\\
 E-mail address: rrjin@cauc.edu.cn\\
 \\
 School of Mathematical Sciences, Beijing Normal University\\
 Laboratory of Mathematics and Complex Systems, Ministry of Education\\
 Beijing 100875, The People's Republic of China\\
 E-mail address: gclu@bnu.edu.cn\\
\end{tabular}
\end{document}